\documentclass[11pt,oneside,reqno]{amsart}

\usepackage[margin=1in]{geometry}
\usepackage{amsmath,amssymb,amsthm,textcomp}
\usepackage{amsfonts,graphicx}
\usepackage{amsaddr}
\usepackage[mathscr]{eucal}
\usepackage{pgfplots}
\pgfplotsset{compat=1.15}
\usepackage{subfig,tikz}
\usepackage{array} 
\usepackage{booktabs}
\usepackage{multicol}
\usepackage{multirow}
\usepackage[utf8]{inputenc}
\usepackage{bm}

\usepackage{color}
 
\usepackage{float}
\usepackage{hyperref}
 \hypersetup{colorlinks,citecolor=blue}
\hypersetup{colorlinks=true,linkcolor=red,filecolor=magenta,    urlcolor=blue}

\usepackage{pslatex}
 \usepackage{array}
\usepackage[square,numbers]{natbib}
\allowdisplaybreaks[1]
\theoremstyle{definition}

\newcommand{\ncom}{\newcommand}

\ncom{\integ}[4]{\int_{#1}^{#2}\,{#3}\,d{#4}}
\ncom{\vspan}[1]{{{\rm\,span}\{ #1 \}}}
\ncom{\dm}[1]{ {\displaystyle{#1} } }
\ncom{\ri}[1]{{#1} \index{#1}}

\newtheorem{theorem}{\bf Theorem}[section]
\newtheorem{remark}{\bf Remark}[section]
\newtheorem{proposition}{Proposition}[section]

\newtheorem{corollary}{Corollary}[section]

\newtheorem{definition}{Definition}[section]
\newtheoremstyle
{remarkstyle}
{}
{11pt}
{}
{}
{\bfseries}
{:}
{     }
{\thmname{#1} \thmnumber{#2} }

\theoremstyle{remarkstyle}

\newcommand{\ceil}[1]{\lceil{#1}\rceil}

\DeclareFontFamily{U}{BOONDOX-calo}{\skewchar\font=45 }
\DeclareFontShape{U}{BOONDOX-calo}{m}{n}{
  <-> s*[1.05] BOONDOX-r-calo}{}
\DeclareFontShape{U}{BOONDOX-calo}{b}{n}{
  <-> s*[1.05] BOONDOX-b-calo}{}
\DeclareMathAlphabet{\mathcalboondox}{U}{BOONDOX-calo}{m}{n}
\SetMathAlphabet{\mathcalboondox}{bold}{U}{BOONDOX-calo}{b}{n}
\DeclareMathAlphabet{\mathbcalboondox}{U}{BOONDOX-calo}{b}{n}

\begin{document}
\color{black}       
\title{Non-Homogeneous Generalized Fractional Skellam Process}



\author{Kartik Tathe$^1$ and Sayan Ghosh$^2$}
\address{Department of Mathematics, Birla Institute of Technology and Science, Pilani, Hyderabad Campus, Jawahar Nagar, Kapra Mandal, Medchal District, Telangana 500078, India. }
\thanks{$^2$Corresponding author}
\email{$^1$kartikvtathe@gmail.com,$^2$sayan@hyderabad.bits-pilani.ac.in}





\subjclass[2020]{Primary: 60G22, 60G55; Secondary: 60G20, 60G51}
\keywords{Skellam process, non-homogeneous counting process, running average, time-changed subordination, LRD / SRD property.}
\begin{abstract}
This paper introduces the Non-homogeneous Generalized Skellam process (NGSP) and its fractional version NGFSP by time changing it with an independent inverse stable subordinator. We study distributional properties for NGSP and NGFSP including probability generating function, probability mass function (p.m.f.), factorial moments, mean, variance, covariance and correlation structure. Then we investigate the long and short range dependence structures for NGSP and NGFSP, and obtain the governing state differential equations of these processes along with their increment processes. We obtain recurrence relations satisfied by the state probabilities of Non-homogeneous generalized counting process (NGCP), NGSP and NGFSP. The weighted sum representations for these processes are provided. We further obtain martingale and renewal properties along with arrival time distribution for NGSP and NGFSP. An alternative version of NGFSP with a closed-form p.m.f. is introduced along with a discussion of its distributional and asymptotic properties. In addition, we study the running average processes of GCP and GSP which are of independent interest. The p.m.f. of NGSP and the simulated sample paths of GFSP, NGFSP and related processes are plotted. Finally, we discuss an application to a high-frequency financial data set pointing out the advantages of our model compared to existing ones.
\end{abstract}

\maketitle


\section{Introduction:}
 A non-negative integer-valued stochastic process $\{N(t)\}_{t\geq 0}$ satisfying $N(t_1) \leq N(t_2)$ for $t_1 < t_2$ is called a counting process. A counting process $\{N(t)\}_{t\geq 0}$ with independent and stationary increments is called a Poisson process with rate $\lambda>0$ if $N(t)$ has Poisson distribution with parameter $\lambda t$. Various generalizations of the Poisson process have been studied in the literature as discussed below.

\citet{LASKIN2003201} introduced the time fractional Poisson process using a special form of the fractional Kolmogorov–Feller equation which uses R-L fractional derivative instead of ordinary derivative in the governing state differential equation. \citet{Meerschaert2011} obtained the same process by time changing (replacing the time variable) the Poisson process with an independent inverse stable subordinator. Time fractional generalization provides the flexibility of non-stationary and dependent increments. \citet{Leonenko2017} introduced a non-homogeneous generalization of the time fractional Poisson process by replacing the time variable with an appropriate function of time (rate function).

For insurance modeling, \citet{Kostadinova2013} introduced another generalization of Poisson process, namely the Poisson process of order $k$ (PPoK) to study claim arrivals in groups of size $k$ where the number of arrivals in a group is uniformly distributed over $k$ points. Expanding further, \citet{Sengar2020} proposed a generalization of the above process to deal with dependent inter-arrival times.

Allowing various jump sizes with different intensities, \citet{Crescenzo2016} introduced the Generalized fractional Counting process (GFCP) and studied its connections with the Poisson process. They showed that it is equal in distribution to a compound Poisson process. The non-homogeneous version of generalized counting process (NGCP) was considered by \citet{Kataria2025}. 

The Skellam process is an integer valued L\'{e}vy process obtained by taking the difference of two independent Poisson processes. It was first studied by \citet{Irwin2018} for equal means and \citet{Skellam1946} for unequal means. \citet{Kerss2014} studied the time-fractional version of Skellam process by time-changing the Skellam process with an independent inverse stable subordinator. \citet{Kataria2022b} introduced and studied the Generalized Fractional Skellam process (GFSP) by considering the difference of two independent GFCPs. Recently, \citet{Gupta2020} studied another generalization of Skellam process, namely the Skellam process of order $k$ by taking the difference of two independent PPoKs along with its time-changed version. They also discussed the space-fractional and tempered space-fractional Skellam processes by time-changing the Skellam process with independent stable subordinator and tempered stable subordinator, respectively.

The applications of Skellam process cover various areas, such as modeling of intensity difference of pixels in cameras (see \citet{Hwang2007}) and modeling the difference in number of goals of two opponent teams in a football game (see \citet{Karlis2008}). \citet{Kerss2014} studied the use of fractional Skellam process in high frequency financial data analysis for modeling up and down jump processes related to tick-by-tick stock price data. Also they fitted the inter arrival time data using both exponential and Mittag-Leffer distributions and showed that the Mittag-Leffler distribution provides a better fit than the exponential distribution. However, it can be observed from figure (3) of \citet{Kerss2014} that the fractional Skellam process is not always suitable for modeling the arrivals in the up and down jump processes. In practice, the arrivals may not consistently occur at constant rates and also the magnitude of jumps (price changes) may fluctuate depending on various factors. Therefore further refinement is required to accommodate higher-order jumps and varied patterns in arrivals. Such type of data may also arise in other situations like queuing systems and traffic flow. These issues motivate us to propose a more flexible model than existing ones and investigate its characteristics for better insight.
  
In this article, we introduce the Non-homogeneous Generalized Skellam process (NGSP) and its fractional version (NGFSP). The remaining paper is organized as follows. In Section \ref{sec 2}, we discuss some preliminaries including definitions and results. In Section \ref{sec 3}, we introduce the NGSP and study its distributional properties.  In Section \ref{sec 4}, the fractional version of NGSP namely, the NGFSP is introduced and its characteristics are studied. Section \ref{sec 5} deals with the increment processes of NGSP and NGFSP. In Section \ref{sec 6}, we introduce an alternative version of NGFSP and study its distributional and dependence properties along with some limiting results. Section \ref{sec 7} discusses the running average processes for GCP and GSP along with their important characteristics. In Section \ref{sec 8}, we plot the p.m.f. of NGSP and the simulated sample paths of NGFSP and related processes. 
Finally, in Section \ref{sec 9}, we provide an application of our model to high frequency financial data.

\section{Preliminaries} \label{sec 2}
In this section, we present some preliminary results and definitions which will be used later in the paper.

\subsection{Multinomial Theorem}\label{multinomial thm} For $x_{1}, x_{2}, ..., x_{m}\in \mathbb{R}$, we have
\begin{equation*}
    (x_{1}+x_{2}+...+x_{m})^{n}= \sum_{\substack{k_{1}+k_{2}+...+k_{m}=n \\ k_{1},k_{2},...,k_{m}\geq 0}}\binom{n}{k_{1},k_{2},...,k_{m}}\prod_{j=1}^{m}x_{j}^{k_{j}}
\end{equation*}
where 
\begin{equation*}
    \binom{n}{k_{1},k_{2},...,k_{m}}=\frac{n!}{k_{1}!k_{2}!...k_{m}!}\,\,.
\end{equation*}

\subsection{Modified Bessel's function of first kind}\label{mod Bessel fun} The modified Bessel's function of first kind is given as (see \citet{Abramowitz1972})
\begin{align*}
    I_{n}(z)=\left(\frac{z}{2}\right)^{n}\sum_{k=0}^{\infty}\frac{\left(\frac{z^{2}}{4}\right)^{k}}{k!\Gamma(n+k+1)}
    =\sum_{k=0}^{\infty}\frac{z^{n+2k}}{2^{n+2k}k!(n+k)!}
\end{align*}
and satisfies the following properties (see \citet{Sneddon1966}, p. 128)

\begin{align*}
    I_{n}(z)&=I_{-n}(z),\quad
    \frac{d}{dz}I_{n}(z)=\frac{1}{2}\left(I_{n-1}(z)+I_{n+1}(z)\right).
\end{align*}

\subsection{R-L fractional derivative}\label{R-L derivative} For $\gamma>0$, the R-L fractional derivative is defined as (see \citet{Kilbas2006})
\begin{equation*}
D_{t}^{\gamma}f(t):=
\begin{cases}
\frac{1}{\Gamma(m-\gamma)}\frac{d^\alpha}{dt^\alpha}\int_{0}^{t}\frac{f(s)}{(t-s)^{\gamma+1-m}}ds, & \quad m-1<\gamma<m,\\
\frac{d^m}{dt^m}f(t), & \quad \gamma=m,
\end{cases}
\end{equation*}
where $m$ is a positive integer. 

\subsection{Caputo–Djrbashian fractional derivative}\label{caputo}
The Caputo–Djrbashian (C-D) fractional derivative is defined as (see \citet{Kilbas2006})
\begin{equation*}
    \frac{d^\alpha}{dt^\alpha}f(t)=\frac{1}{\Gamma(1-\alpha)}\int_{0}^{t}\frac{df(\tau)}{d\tau}\frac{d\tau}{(t-\tau)^{\alpha}}, \quad 0<\alpha<1.
\end{equation*}
Its Laplace transform is 
\begin{equation*}
    \mathcal{L}\left\{\frac{d^\alpha}{dt^\alpha}f\right\}(s)=s^\alpha \mathcal{L}\{f\}(s)-s^{\alpha-1}f(0^+).
\end{equation*}
The following relationship holds for the R-L fractional derivative and the Caputo–Djrbashian fractional derivative (see \citet{Meerschaert2013}):
\begin{equation*}
\frac{d^\alpha}{dt^\alpha}f(t)=D_{t}^{\alpha}f(t)-f(0+)\frac{t^{-\alpha}}{\Gamma(1-\alpha)}.
\end{equation*}

\subsection{Stable subordinator and its inverse}\label{stable subordinator} A stable subordinator $\{D_{\alpha}(t)\}_{t\geq 0}$, $0<\alpha<1$, is a non-decreasing L\'{e}vy process. The Laplace transform of its density is given by $\mathbb{E}(e^{-lD_\alpha(t)})=e^{-tl^\alpha}$, $l>0$ and the associated Bernstein function is $f(l)=l^\alpha$. Its first passage time $\{Y_\alpha(t)\}_{t\geq 0}$ is called the inverse stable subordinator and defined as  
\begin{equation*}
Y_\alpha(t):=\inf\{x\geq 0:D_{\alpha}(x)>t \}.
\end{equation*}
The density of $Y_\alpha(t)$ denoted by $h_\alpha(x,t)$ (see \citet{Meerschaert2011}) has the Laplace transform (see \citet{Meerschaert2013}) $\tilde{h}_\alpha(x,l)=l^{\alpha-1}e^{-xl^{\alpha}}.$
The mean, variance and covariance (see \citet{Leonenko2014}) of $Y_\alpha(t)$ are 
\begin{align*}
\mathbb{E}(Y_{\alpha}(t))=\frac{t^\alpha}{\Gamma(\alpha+1)}, \quad
\mathbb{V}(Y_\alpha(t))=\left(\frac{2}{\Gamma(2\alpha+1)}-\frac{1}{\Gamma^2(\alpha+1)}\right)t^{2\alpha},\\
\text{Cov}\left(Y_\alpha(s),Y_\alpha(t)\right)=\frac{1}{\Gamma^2(\alpha+1)}\left(\alpha s^{2\alpha}B(\alpha,\alpha+1)+F(\alpha:s,t)\right),\quad 0<s<t,
\end{align*}
where $F(\alpha:s,t)=\alpha t^{2\alpha}B(\alpha,\alpha+1:s/t)-(ts)^\alpha$. Here $B(\alpha,\alpha+1)$ and $B(\alpha,\alpha+1:s/t)$ denote the beta function and the incomplete beta function, respectively. Moreover from \citet{Leonenko2014}, we have
\begin{equation*}
F(\alpha:s,t)\sim\frac{-\alpha^2}{(\alpha+1)}\frac{s^{\alpha+1}}{t^{1-\alpha}} \quad \text{as} \quad t\rightarrow \infty.
\end{equation*}
Thus we obtain for fixed $s$ and large $t$
\begin{equation*}
\text{Cov}\left(Y_\alpha(s),Y_\alpha(t)\right)\sim\frac{1}{\Gamma^2(\alpha+1)}\left(\alpha s^{2\alpha}B(\alpha,\alpha+1)-\frac{\alpha^2}{(\alpha+1)}\frac{s^{\alpha+1}}{t^{1-\alpha}}\right),\quad 0<s<t.
\end{equation*}

\begin{remark}
    The Dickman subordinator and its inverse (see \citet{Gupta2024}) are generalizations of the stable and inverse stable subordinators, respectively. Moreover \citet{Gupta2024} introduced the generalized C-D and generalized R-L fractional derivatives corresponding to the Dickman subordinator. So the results obtained using C-D and R-L fractional derivatives in this paper may be studied using these generalized derivatives too.    
\end{remark}

\subsection{LRD and SRD properties}\label{LRD/SRD} For fixed $s>0$, suppose the correlation function of a stochastic process $\{Y(t)\}_{t\geq 0}$ has the following asymptotic behavior (see \citet{DOVIDIO2014} or \citet{Maheshwari2016}): 
\begin{equation*}
\text{Corr}\left(Y(s), Y(t)\right)\sim c(s)t^{-\theta} \quad\text{as} \quad t\rightarrow \infty \quad \text{for some} \quad c(s)>0.
\end{equation*}
Then $\{Y(t)\}_{t\geq 0}$ has long range dependence (LRD) property if $\theta\in (0,1)$ and short range dependence (SRD) property if $\theta\in (1,2)$. 

\subsection{Lemma} (\citet{Xia2018}) If $\{X(t)\}_{t\ge 0}$ is a L\'{e}vy process and its Riemann integral is defined by $Y(t)=\int_{0}^{t}X(s)ds,$
then the characteristic function of $Y(t)$ satisfies
\begin{equation*}
    \phi_{Y(t)}(u)=\mathbb{E}[e^{iuY(t)}]=e^{t\left(\int_{0}^{1}\log\phi_{X(1)}(tuz)dz\right)}, \quad u\in\mathbb{R} \label{running average lemma}
\end{equation*}
where $\phi_{X(1)}$ is the characteristic function of $X(t)$ at $t=1$.


\subsection{Non-homogeneous generalized counting process (NGCP)}\label{ngcp}
A special case of the GFCP corresponding to $\alpha=1$ is called the generalized counting process (GCP). \citet{Kataria2025} introduced a non-homogeneous version of the GCP (NGCP). Suppose $\{\mathcal{M}(t)\}_{t\geq 0}$ is a counting process with deterministic and time-dependent intensity functions $\lambda_j:[0,\infty)\to [0,\infty)$ for $ 1\leq j \leq k$ and cumulative rate functions $\Lambda_{j}(t)=\int_{0}^{t}\lambda_{j}(u)du$ such that $\mathcal{M}(0)=0$. Then $\{\mathcal{M}(t)\}_{t\geq 0}$ is a NGCP if it has independent increments and its transition probabilities are given by
\begin{align*}
P(\mathcal{M}(t+h)=n\mid\mathcal{M}(t)=m) &=
\begin{cases}
1 - \sum_{j=1}^{k}\lambda_j(t) + o(h), & n=m, \\
\sum_{j=1}^{k}\lambda_j(t) + o(h), & n=m+j, j=1,2,\ldots,k, \\
o(h), & n>m+k,
\end{cases}
\end{align*}
where $o(h)\rightarrow 0$ as $h\rightarrow 0$. The mean, variance and covariance of NGCP are given by
\begin{align*}
    \mathbb{E}\left(\mathcal{M}(t)\right)=\sum_{j=1}^{k}j\Lambda_j(t), \quad
    \mathbb{V}\left(\mathcal{M}(t)\right)=\sum_{j=1}^{k}j^2\Lambda_j(t), \quad\text{and} \quad \text{Cov}\left(\mathcal{M}(s), \mathcal{M}(t)\right)=\sum_{j=1}^{k}j^2\Lambda_j(s\wedge t).
\end{align*}

\section{Non-Homogeneous generalized Skellam Process}\label{sec 3}
 In this section, we introduce the Non-homogeneous Generalized Skellam Process and study its various characteristics.  
 \begin{definition} \label{defngsp} The Non-Homogeneous Generalized Skellam Process (NGSP) is defined as
\begin{equation*}
    \mathcalboondox{S}(t):=\mathcal{M}_{1}(t)-\mathcal{M}_{2}(t)
\end{equation*}
where $\{\mathcal{M}_{1}(t)\}_{t\geq 0}$ and $\{\mathcal{M}_{2}(t)\}_{t\geq 0}$ are two independent NGCPs with intensity parameters 
\begin{align*}
    \lambda_{j} : [0,\infty)\rightarrow [0,\infty)\quad
    \text{and} \quad \gamma_{j} :[0,\infty)\rightarrow [0,\infty), \quad j=1,2,...,k
\end{align*}
 respectively such that $\mathcal{M}_{1}(0)=0$ and $\mathcal{M}_{2}(0)=0$.
\end{definition}
The cumulative rate functions of $\{\mathcal{M}_{1}(t)\}_{t\geq 0}$ and $\{\mathcal{M}_{2}(t)\}_{t\geq 0}$ are 
\begin{align*}
    \Lambda_{j}(t)=\int_{0}^{t}\lambda_{j}(u)du <\infty\quad
    \text{and} \quad T_{j}(t)=\int_{0}^{t}\gamma_{j}(u)du<\infty
\end{align*}
respectively. If $0\leq s<t$, then
\begin{align*}
    \Lambda_{j}(s,t)=\int_{s}^{t}\lambda_{j}(u)du =\Lambda_{j}(t)-\Lambda_{j}(s)
    \quad \text{and} \quad T_{j}(s,t)=\int_{s}^{t}\gamma_{j}(u)du =T_{j}(t)-T_{j}(s).
\end{align*}
\begin{remark}
Note that for $\lambda_{j}(t)=\lambda_{j}$ and $\gamma_{j}(t)=\gamma_{j}$, NGSP reduces to the Generalized Skellam Process (see \citet{Kataria2022b}). Moreover if $\lambda_j(t)=\lambda$ and $\gamma_{j}(t)=\gamma$, then NGSP reduces to the Skellam Process of order $k$ (see \citet{Gupta2020}) and to the usual Skellam process for $k=1$.    
\end{remark}

\subsection{Probability Generating Function of NGSP}
The probability generating function (p.g.f.) of NGSP for $|u|<1$ can be obtained using the p.g.f. of NGCP (see \citet{Kataria2025}) as follows:
\begin{align}\label{2}
   G_{\mathcalboondox{S}}(u,t)=\exp\left(\sum_{j=1}^{k}\Lambda_{j}(t)(u^{j}-1)\right) \exp\left(\sum_{j=1}^{k}T_{j}(t)(u^{-j}-1)\right)= \exp\left(\sum_{j=1}^{k}\left(\Lambda_{j}(t)(u^{j}-1)+T_{j}(t)(u^{-j}-1)\right)\right).
\end{align}
It satisfies the following differential equation:
\begin{equation*}
     \frac{d}{dt}G_{\mathcalboondox{S}}(u,t)= G_{\mathcalboondox{S}}(u,t)\left\{ \sum_{j=1}^{k}\left[ \left(\frac{d}{dt}\Lambda_{j}(t)\right)(u^{j}-1)+\left(\frac{d}{dt}T_{j}(t)\right)(u^{-j}-1)\right] \right\}.
\end{equation*}

\subsection{Mean, Variance and Covariance}
Using the mean, variance and covariance of NGCP (see Section \ref{ngcp}), we can obtain the corresponding expressions for NGSP as follows:
\begin{align*}
     \mathbb{E}(\mathcalboondox{S}(t))&=\sum_{j=1}^{k}j\left(\Lambda_{j}(t)-T_{j}(t)\right), \quad
    \mathbb{V}\left(\mathcalboondox{S}(t)\right)= \sum_{j=1}^{k}j^{2}\left(\Lambda_{j}(t)+T_{j}(t)\right)\,\,,\\
    \text{Cov}\left(\mathcalboondox{S}(s),\mathcalboondox{S}(t)\right) &= \text{Cov}(\mathcal{M}_1(s)-\mathcal{M}_2(s), \mathcal{M}_1(t)-\mathcal{M}_2(t))=\sum_{j=1}^{k}j^{2}\left(\Lambda_{j}(s\wedge t)+T_{j}(s\wedge t)\right).
\end{align*}
Also note that $\mathbb{V}\left(\mathcalboondox{S}(t)\right)-\mathbb{E}(\mathcalboondox{S}(t))
    =\sum_{j=1}^{k}\Lambda_{j}(t)(j^2-j)+\sum_{j=1}^{k}T_{j}(t)(j^2+j) > 0,$
which implies the NGSP exhibits over-dispersion.

\begin{remark}
The $k$th raw moment of NGSP is given by
\begin{align*}
\mathbb{E}(\mathcalboondox{S}^k(t)) &= \left(u\frac{\partial}{\partial u}\right)^k G_{\mathcalboondox{S}}(u,t)\big\vert_{u=1} = \sum_{j=1}^{k}j!\left(\Lambda_{j}(t)-T_{j}(t)\right),\quad k\ge 1.
\end{align*}
\end{remark}

\subsection{Correlation structure of NGSP}
For fixed $s<t$, the correlation function of NGSP is
\begin{align*}
    \text{Corr}\left(\mathcalboondox{S}(s), \mathcalboondox{S}(t)\right)=\frac{\text{Cov}\left(\mathcalboondox{S}(s),\mathcalboondox{S}(t)\right)}{\sqrt{\mathbb{V}\left(\mathcalboondox{S}(s)\right)}\sqrt{\mathbb{V}\left(\mathcalboondox{S}(t)\right)}}&=\frac{\sum_{j=1}^{k}j^{2}\left(\Lambda_{j}(s)+T_{j}(s)\right)}{\sqrt{\sum_{j=1}^{k}j^{2}\left(\Lambda_{j}(s)+T_{j}(s)\right)}\sqrt{\sum_{j=1}^{k}j^{2}\left(\Lambda_{j}(s)+T_{j}(s)\right)}}\\
    &=\frac{\sqrt{\sum_{j=1}^{k}j^{2}\left(\Lambda_{j}(s)+T_{j}(s)\right)}}{\sqrt{\sum_{j=1}^{k}j^{2}\left(\Lambda_{j}(t)+T_{j}(t)\right)}}\,\,.\label{14}
\end{align*}
    Assume both NGCPs have Weibull rate functions given by
\begin{equation*}
    \Lambda_{j}(t)=\left(\frac{t}{b_{1j}}\right)^{c_{1j}},\,\,T_{j}(t)=\left(\frac{t}{b_{2j}}\right)^{c_{2j}};\quad b_{1j}, b_{2j}> 0,~  c_{1j}, c_{2j}\geq0\,\,
\end{equation*}
and $c_{1j}\geq c_{2j}$ without loss of generality. Then the correlation function reduces to
\begin{align*}
    \text{Corr}\left(\mathcalboondox{S}(s), \mathcalboondox{S}(t)\right)&=\sqrt{\frac{\sum_{j=1}^{k}j^{2}\left(\left(\frac{s}{b_{1j}}\right)^{c_{1j}}+\left(\frac{s}{b_{2j}}\right)^{c_{2j}}\right)}{\sum_{j=1}^{k}j^{2}\left(\left(\frac{t}{b_{1j}}\right)^{c_{1j}}+\left(\frac{t}{b_{2j}}\right)^{c_{2j}}\right)}}
    = \sqrt{\frac{\sum_{j=1}^{k}j^{2}s^{c_{2j}}\left(\left(\frac{s^{c_{1j}-c_{2j}}}{b_{1j}^{c_{1j}}}\right)+\left(\frac{1}{b_{2j}}\right)^{c_{2j}}\right)}{\sum_{j=1}^{k}j^{2}\left(\left(\frac{t^{c_{1j}-c_{2j}}}{b_{2j}^{c_{2j}}}\right)+\left(\frac{1}{b_{2j}}\right)^{c_{2j}}\right)}}t^{-c_{2j}/2}
    \end{align*}
    which may be expressed as
    \begin{equation*}
    \text{Corr}\left(\mathcalboondox{S}(s), \mathcalboondox{S}(t)\right)= C_{0}t^{-d/2}, \,\,\text{where}\,\,d=c_{2j}\,\,\text{and}\,\,  C_{0}=\sqrt{\frac{\sum_{j=1}^{k}j^{2}s^{c_{2j}}\left(\left(\frac{s^{c_{1j}-c_{2j}}}{b_{1j}^{c_{1j}}}\right)+\left(\frac{1}{b_{2j}}\right)^{c_{2j}}\right)}{\sum_{j=1}^{k}j^{2}\left(\left(\frac{t^{c_{1j}-c_{2j}}}{b_{2j}^{c_{2j}}}\right)+\left(\frac{1}{b_{2j}}\right)^{c_{2j}}\right)}}.\label{15}
    \end{equation*}
Hence we observe that for Weibull rate functions, NGSP possesses the LRD property if $d\in(0,2)$ and the SRD property if $ d\in(2,4)$.

\subsection{Factorial moments for NGSP}
The $r^{th}$ factorial moment of NGSP is defined as 
\begin{equation*}
\Psi_{\mathcalboondox{S}}(r,t) = \mathbb{E}((\mathcalboondox{S}(t))_r) = \mathbb{E}[\mathcalboondox{S}(t)(\mathcalboondox{S}(t)-1)(\mathcalboondox{S}(t)-2)\cdots(\mathcalboondox{S}(t)-r+1)]=\frac{d^{r}}{du^{r}}G_{\mathcalboondox{S}}(u,t)\bigg|_{u=1} 
\end{equation*}
for $r\ge 1$, where $G_{\mathcalboondox{S}}(u,t)$ is the p.g.f. of NGSP. The following result provides the explicit expression for such moments of NGSP. 
\begin{theorem}\label{factorial.ngsp}
    The $r^{th}$ factorial moment of NGSP for $r\geq 1$ is given by 
\begin{equation*}
    \Psi_{\mathcalboondox{S}}(r,t)=\sum_{m=0}^{r-1}\binom{r-1}{m}\frac{d^{m}}{du^{m}}G_{\mathcalboondox{S}}(u,t)\bigg|_{u=1}\left\{\sum_{j=1}^{k}\left[\Lambda_{j}(t)P(j,r-m)+T_{j}(t)P(-j,r-m)\right]\right\}
\end{equation*}
where $P(a,b)=a(a-1)\cdots(a-b+1)$ for $a\ge b$.
\end{theorem}

\begin{proof}
From \eqref{2}, the p.g.f. of NGSP is given by
    \begin{equation*}
    G_{\mathcalboondox{S}}(u,t)=\exp\left[\sum_{j=1}^{k}\left(\Lambda_{j}(t)(u^{j}-1)+T_{j}(t)(u^{-j}-1)\right)\right].
\end{equation*}
For $r=1$, we have
\begin{align*}
\Psi_{\mathcalboondox{S}}(1,t) = \frac{d}{du}G_{\mathcalboondox{S}}(u,t)\bigg|_{u=1} &= G_{\mathcalboondox{S}}(u,t)\sum_{j=1}^{k}\left(\Lambda_{j}(t)(ju^{j-1})+T_{j}(t)(-ju^{-j-1})\right)\bigg|_{u=1} \\
&=\binom{0}{0}G_{\mathcalboondox{S}}(u,t)\bigg|_{u=1}\left\{\sum_{j=1}^{k}\left[\Lambda_{j}(t)P(j,1)+T_{j}(t)P(-j,1)\right]\right\} 
\end{align*}
which implies the proposition statement is true for $r=1$. For $r=2$, we have
\begin{align*}
\Psi_{\mathcalboondox{S}}(2,t) = \frac{d^2}{du^2}G_{\mathcalboondox{S}}(u,t)\bigg|_{u=1}
&=G_{\mathcalboondox{S}}(u,t)\sum_{j=1}^{k}\left(\Lambda_{j}(t)(j(j-1)u^{j-2})+T_{j}(t)((-j)(-j-1)u^{-j-2})\right)\bigg|_{u=1} \\
&+\frac{d}{du}G_{\mathcalboondox{S}}(u,t)\sum_{j=1}^{k}\left(\Lambda_{j}(t)(ju^{j-1})+T_{j}(t)(-ju^{-j-1})\right)\bigg|_{u=1}\\
&=\binom{1}{0}G_{\mathcalboondox{S}}(u,t)\bigg|_{u=1}\sum_{j=1}^{k}\left[\Lambda_{j}(t)P(j,2)+T_{j}(t)P(-j,2)\right] \\
&+\binom{1}{1}\frac{d}{du}G_{\mathcalboondox{S}}(u,t)\bigg|_{u=1}\sum_{j=1}^{k}\left[\Lambda_{j}(t)P(j,1)+T_{j}(t)P(-j,1)\right]
\end{align*}
which implies the proposition statement is true for $r=2$. Now assume the statement is true for $r=n$, i.e., 
\begin{equation*}
\Psi_{\mathcalboondox{S}}(n,t)=\sum_{m=0}^{n-1}\binom{n-1}{m}\frac{d^{m}}{du^{m}}G_{\mathcalboondox{S}}(u,t)\bigg|_{u=1}\left\{\sum_{j=1}^{k}\left[\Lambda_{j}(t)P(j,n-m)+T_{j}(t)P(-j,n-m)\right]\right\}
\end{equation*}
For $r=n+1$, we have
\begin{align*}
\Psi_{\mathcalboondox{S}}(n+1,t) &= \frac{d^{n+1}}{du^{n+1}}G_{\mathcalboondox{S}}(u,t)\bigg|_{u=1} = \frac{d}{du}\Psi_{\mathcalboondox{S}}(n,t)\bigg|_{u=1} \\
&=\frac{d}{du}\sum_{m=0}^{n-1}\binom{n-1}{m}\frac{d^{m}}{du^{m}}G_{\mathcalboondox{S}}(u,t)\sum_{j=1}^{k}\left[\Lambda_{j}(t)P(j,n-m)u^{j-n+m}+T_{j}(t)P(-j,n-m)u^{-j-n+m}\right]\bigg|_{u=1} \\
&=\sum_{m=0}^{n-1}\binom{n-1}{m}\frac{d^{m+1}}{du^{m+1}}G_{\mathcalboondox{S}}(u,t)\sum_{j=1}^{k}\left[\Lambda_{j}(t)P(j,n-m)u^{j-n+m}+T_{j}(t)P(-j,n-m)u^{-j-n+m}\right]\bigg|_{u=1} \\
&+\sum_{m=0}^{n-1}\binom{n-1}{m}\frac{d^{m}}{du^{m}}G_{\mathcalboondox{S}}(u,t)\sum_{j=1}^{k}\left[\Lambda_{j}(t)P(j,n-m+1)u^{j-n+m-1}+T_{j}(t)P(-j,n-m+1)u^{-j-n+m-1}\right]\bigg|_{u=1} \\
&= G_{\mathcalboondox{S}}(u,t)\bigg|_{u=1}\left[\Lambda_{j}(t)P(j,n+1)+T_{j}(t)P(-j,n+1)\right] + \frac{d^{n}}{du^{n}}G_{\mathcalboondox{S}}(u,t)\bigg|_{u=1}\sum_{j=1}^{k}\left[\Lambda_{j}(t)P(j,1)+T_{j}(t)P(-j,1)\right] \\
&+\sum_{m=1}^{n-1}\left[\binom{n-1}{m-1}+\binom{n-1}{m}\right]\frac{d^{m}}{du^{m}}G_{\mathcalboondox{S}}(u,t)\bigg|_{u=1}\sum_{j=1}^{k}\left[\Lambda_{j}(t)P(j,n+1-m)+T_{j}(t)P(-j,n+1-m)\right] \\
&=\sum_{m=0}^{n}\binom{n}{m}\frac{d^{m}}{du^{m}}G_{\mathcalboondox{S}}(u,t)\bigg|_{u=1}\left\{\sum_{j=1}^{k}\left[\Lambda_{j}(t)P(j,n+1-m)+T_{j}(t)P(-j,n+1-m)\right]\right\}
\end{align*}
which implies the proposition statement is true for $r=n+1$. Hence the result follows by mathematical induction.
\end{proof}

\subsection{\textbf{State Probabilities of NGSP}}

\begin{theorem}\label{20}
    The state probabilities of NGSP are given by
\begin{equation*}
    p(n,t)= P\{\mathcalboondox{S}(t)=n\} = e^{-(A+B)}\left(\frac{A}{B}\right)^{\frac{|n|}{2}}I_{|n|}\left(2\sqrt{AB}\right), \quad n\in \mathbb{Z}
\end{equation*}
where $\sum_{j=1}^{k}\Lambda_{j}(t)=A$, $\sum_{j=1}^{k}T_{j}(t)=B$ and $I_{n}(.)$ is the modified Bessel's function of first kind. 
\end{theorem} 

\begin{proof}
    We know that the marginal distribution of the increment process $\{\mathcal{I}(t,v)\}_{t\geq 0} $ of NGCP with cumulative intensity parameter $\Lambda_{j}(t)$ for $v>0$ and $n>0$ is given by (see \citet{Kataria2025})
\begin{align*}
    q_{n}(t,v)=P\{\mathcal{I}(t,v)=n\}&=P\{\mathcal{M}(t+v)-\mathcal{M}(v)=n\}=\sum_{\Omega(k,n)}\prod_{j=1}^{k}\frac{\left(\Lambda_{j}(v,t+v)\right)^{x_{j}}}{x_{j}!}e^{-\sum_{j=1}^{k}\Lambda_{j}(v,t+v)}\label{16}
\end{align*}
where $\Omega(k,n)= \{(x_1,x_2,...,x_k)|\sum_{j=1}^{k}jx_j=n \,\,\text{for non-negative integers}\,\,\, x_j\}$. Substituting $v=0$ in the above equation, the marginal distribution of NGCP can be obtained as follows:
\begin{equation*}
    q_{n}(t,0)=P\{\mathcal{M}(t)=n\}=\sum_{\Omega(k,n)}\prod_{j=1}^{k}\frac{\left(\Lambda_{j}(t)\right)^{x_{j}}}{x_{j}!}e^{-\sum_{j=1}^{k}\Lambda_{j}(t)}.\label{17}
\end{equation*}

Hence the state probabilities of NGSP for $n\in\mathbb{Z}$ are given by 
\begin{align}
    p(n,t)
    &=P\{\mathcal{M}_{1}(t)-\mathcal{M}_{2}(t)=n\}\mathbb{I}_{\{n\geq0\}}+P\{\mathcal{M}_{1}(t)-\mathcal{M}_{2}(t)=n\}\mathbb{I}_{\{n<0\}}\nonumber\\
    &= \sum_{m=0}^{\infty}[P\{\mathcal{M}_{1}(t)=m+n\}P\{\mathcal{M}_{2}(t)=m\}\mathbb{I}_{\{n\geq0\}}+P\{\mathcal{M}_{2}(t)=m+|n|\}P\{\mathcal{M}_{1}(t)=m\}\mathbb{I}_{\{n<0\}}]\nonumber\\
    &= \sum_{m=0}^{\infty}\left(\sum_{\Omega(k,m+n)}\prod_{j=1}^{k}\frac{\left(\Lambda_{j}(t)\right)^{x_{j}}}{x_{j}!}e^{-\sum_{j=1}^{k}\Lambda_{j}(t)}\right)\left(\sum_{\Omega(k,m)}\prod_{j=1}^{k}\frac{\left(T_{j}(t)\right)^{x_{j}}}{x_{j}!}e^{-\sum_{j=1}^{k}T_{j}(t)}\right)\mathbb{I}_{\{n\geq0\}}\nonumber\\
    &+ \sum_{m=0}^{\infty}\left(\sum_{\Omega(k,m+|n|)}\prod_{j=1}^{k}\frac{\left(T_{j}(t)\right)^{x_{j}}}{x_{j}!}e^{-\sum_{j=1}^{k}T_{j}(t)}\right)\left(\sum_{\Omega(k,m)}\prod_{j=1}^{k}\frac{\left(\Lambda_{j}(t)\right)^{x_{j}}}{x_{j}!}e^{-\sum_{j=1}^{k}\Lambda_{j}(t)}\right)\mathbb{I}_{\{n<0\}}\label{18}
\end{align}

For $n\geq 0$, putting $x_{j}=a_{j}$ and $m=y+\sum_{j=1}^{k}(j-1)a_{j}$ in \eqref{18}, we get
\begin{align}
    p(n,t) &=\sum_{m=0}^{\infty}\left(\sum_{\Omega(k,m+n)}\prod_{j=1}^{k}\frac{\left(\Lambda_{j}(t)\right)^{x_{j}}}{x_{j}!}e^{-\sum_{j=1}^{k}\Lambda_{j}(t)}\right)\left(\sum_{\Omega(k,m)}\prod_{j=1}^{k}\frac{\left(T_{j}(t)\right)^{x_{j}}}{x_{j}!}e^{-\sum_{j=1}^{k}T_{j}(t)}\right)\nonumber\\
    &= \sum_{y=0}^{\infty}\left(\sum_{\sum_{j=1}^{k}a_{j}=y+n}\prod_{j=1}^{k}\frac{\left(\Lambda_{j}(t)\right)^{a_{j}}}{a_{j}!}e^{-\sum_{j=1}^{k}\Lambda_{j}(t)}\right)\left(\sum_{\sum_{j=1}^{k}a_{j}=y}\prod_{j=1}^{k}\frac{\left(T_{j}(t)\right)^{a_{j}}}{a_{j}!}e^{-\sum_{j=1}^{k}T_{j}(t)}\right)\nonumber\\
 &= e^{-\sum_{j=1}^{k}\left(\Lambda_{j}(t)+T_{j}(t)\right)}\sum_{y=0}^{\infty}\left[\frac{1}{(n+y)!y!}\left(\sum_{j=1}^{k}\Lambda_{j}(t)\right)^{n+y}\left(\sum_{j=1}^{k}T_{j}(t)\right)^{y}\right]\label{19}
\end{align}    
where we have used the multinomial theorem (see Section \ref{multinomial thm}) in the last step. Now denoting $\sum_{j=1}^{k}\Lambda_{j}(t)=A$ and $\sum_{j=1}^{k}T_{j}(t)=B$, and using modified Bessel's function of the first kind (see Section \ref{mod Bessel fun}) in \eqref{19}, we have
\begin{align}
    p(n,t)= e^{-(A+B)}\left(\frac{A}{B}\right)^{\frac{n}{2}}I_{n}\left(2\sqrt{AB}\right)\,\,.\label{I}
\end{align}
Similarly, for $n<0$, $p(n,t)$ can be obtained by replacing $n$ with $-n$ in \eqref{I}. Thus

\begin{align*}
    p(n,t)&= e^{-(A+B)}\left(\frac{A}{B}\right)^{\frac{|n|}{2}}I_{|n|}\left(2\sqrt{AB}\right), \quad n\in \mathbb{Z}\,\,.
\end{align*}
\end{proof}

\begin{theorem}
The state probabilities of NGSP satisfy the following system of differential equations:
\begin{align*}
    \frac{d}{dt}p(n,t)&= \frac{1}{2}p(n-1,t)\sum_{j=1}^{k}\left(\lambda_{j}(t)+\frac{A}{B}\gamma_{j}(t)\right)-p(n,t)\sum_{j=1}^{k}\left(\lambda_{j}(t)+\gamma_{j}(t)\right)\\
    & +\frac{n}{2}p(n,t)\sum_{j=1}^{k}\left(\frac{1}{A}\lambda_{j}(t)-\frac{1}{B}\gamma_{j}(t)\right)+\frac{1}{2}p(n+1,t)\sum_{j=1}^{k}\left(\gamma_{j}(t)+\frac{B}{A}\lambda_{j}(t)\right).
\end{align*}    
\end{theorem}\label{ngsp.de}
    
\begin{proof}
    Note that $|n|=n$ for all $n\in \mathbb{Z}$. From Theorem \ref{20}, we have
\begin{align}
    \frac{d}{dt}p(n,t)&=\frac{d}{dt}\left(e^{-(A+B)}\left(\frac{A}{B}\right)^{\frac{n}{2}}I_{n}\left(2\sqrt{AB}\right)\right)\nonumber\\
    &= \left(e^{-(A+B)}\left(\frac{A}{B}\right)^{\frac{n}{2}}\right)\left(\frac{1}{2}\left(I_{n-1}(2\sqrt{AB})+I_{n+1}(2\sqrt{AB})\right)\right)\frac{d}{dt}(2\sqrt{AB})\nonumber\\
    & +I_{n}\left(2\sqrt{AB}\right)\frac{d}{dt}\left(e^{-(A+B)}\left(\frac{A}{B}\right)^{\frac{n}{2}}\right)\label{21}
\end{align}
where the last step follows from Section \ref{mod Bessel fun}. Now 
\begin{align}
    \frac{d}{dt}\left(e^{-(A+B)}\left(\frac{A}{B}\right)^{\frac{n}{2}}\right)&= e^{-(A+B)}\frac{d}{dt}\left(\frac{A}{B}\right)^{\frac{n}{2}}+\left(\frac{A}{B}\right)^{\frac{n}{2}}\frac{d}{dt}e^{-(A+B)}\nonumber\\
    &=e^{-(A+B)}\frac{n}{2}\left(\frac{A}{B}\right)^{\frac{n}{2}-1}\frac{d}{dt}\left(\frac{A}{B}\right)-\left(\frac{A}{B}\right)^{\frac{n}{2}}e^{-(A+B)}\frac{d}{dt}(A+B)\,\,\nonumber \\
    &= e^{-(A+B)}\left[\frac{n}{2}\frac{A^{\frac{n}{2}-1}}{B^{\frac{n}{2}}}\left(\sum_{j=1}^{k}\lambda_{j}(t)\right)-\frac{n}{2}\frac{A^{\frac{n}{2}}}{B^{\frac{n}{2}+1}}\left(\sum_{j=1}^{k}\gamma_{j}(t)\right)-\left(\frac{A}{B}\right)^{\frac{n}{2}}\left(\sum_{j=1}^{k}\lambda_{j}(t)+\sum_{j=1}^{k}\gamma_{j}(t)\right)\right]\,\,\label{22}
\end{align}
and
\begin{equation}
    \frac{d}{dt}I_{n}\left(2\sqrt{AB}\right)=2\frac{1}{2\sqrt{AB}}\frac{d}{dt}(AB)
    =\sqrt{\frac{A}{B}}\left(\sum_{j=1}^{k}\gamma_{j}(t)\right)+\sqrt{\frac{B}{A}}\left(\sum_{j=1}^{k}\lambda_{j}(t)\right)\,\,.\label{23}
\end{equation}
Using \eqref{22} and \eqref{23} in \eqref{21}, we obtain
\begin{align*}
\frac{d}{dt}p(n,t)
    &=\frac{1}{2}e^{-(A+B)}\left(\frac{A}{B}\right)^{\frac{n-1}{2}}\frac{A}{B}I_{n-1}(2\sqrt{AB})\left(\sum_{j=1}^{k}\gamma_{j}(t)\right)+\frac{1}{2}e^{-(A+B)}\left(\frac{A}{B}\right)^{\frac{n+1}{2}}I_{n+1}(2\sqrt{AB})\left(\sum_{j=1}^{k}\gamma_{j}(t)\right)\\
    &+\frac{1}{2}e^{-(A+B)}\left(\frac{A}{B}\right)^{\frac{n-1}{2}}I_{n-1}(2\sqrt{AB})\left(\sum_{j=1}^{k}\lambda_{j}(t)\right)+\frac{1}{2}e^{-(A+B)}\left(\frac{A}{B}\right)^{\frac{n+1}{2}}\frac{B}{A}I_{n+1}(2\sqrt{AB})\left(\sum_{j=1}^{k}\lambda_{j}(t)\right)\\
    &+\frac{n}{2}e^{-(A+B)}\left(\frac{A}{B}\right)^{\frac{n}{2}}\left(\frac{1}{A}\right)I_{n}(2\sqrt{AB})\left(\sum_{j=1}^{k}\lambda_{j}(t)\right)-\frac{n}{2}e^{-(A+B)}\left(\frac{A}{B}\right)^{\frac{n}{2}}\left(\frac{1}{B}\right)I_{n}(2\sqrt{AB})\left(\sum_{j=1}^{k}\gamma_{j}(t)\right)\\
    &-e^{-(A+B)}\left(\frac{A}{B}\right)^{\frac{n}{2}}I_{n}(2\sqrt{AB})\left(\sum_{j=1}^{k}\lambda_{j}(t)\right)-e^{-(A+B)}\left(\frac{A}{B}\right)^{\frac{n}{2}}I_{n}(2\sqrt{AB})\left(\sum_{j=1}^{k}\gamma_{j}(t)\right) \nonumber \\
&=\frac{1}{2}p(n-1,t)\frac{A}{B}\left(\sum_{j=1}^{k}\gamma_{j}(t)\right)+\frac{1}{2}p(n+1,t)\left(\sum_{j=1}^{k}\gamma_{j}(t)\right)+\frac{1}{2}p(n-1,t)\left(\sum_{j=1}^{k}\lambda_{j}(t)\right)\nonumber\\
    &+\frac{1}{2}p(n+1,t)\frac{B}{A}\left(\sum_{j=1}^{k}\lambda_{j}(t)\right)+\frac{n}{2}p(n,t)\frac{1}{A}\left(\sum_{j=1}^{k}\lambda_{j}(t)\right)-\frac{n}{2}p(n,t)\frac{1}{B}\left(\sum_{j=1}^{k}\gamma_{j}(t)\right)\nonumber\\
    &-p(n,t)\left(\sum_{j=1}^{k}\lambda_{j}(t)\right)-p(n,t)\left(\sum_{j=1}^{k}\gamma_{j}(t)\right) \qquad (\text{from Theorem}~ \ref{20}) \nonumber\\
    &=\frac{1}{2}p(n-1,t)\sum_{j=1}^{k}\left(\lambda_{j}(t)+\frac{A}{B}\gamma_{j}(t)\right)-p(n,t)\sum_{j=1}^{k}\left(\lambda_{j}(t)+\gamma_{j}(t)\right)\nonumber\\
    & +\frac{n}{2}p(n,t)\sum_{j=1}^{k}\left(\frac{1}{A}\lambda_{j}(t)-\frac{1}{B}\gamma_{j}(t)\right)+\frac{1}{2}p(n+1,t)\sum_{j=1}^{k}\left(\gamma_{j}(t)+\frac{B}{A}\lambda_{j}(t)\right)\,\,.
\end{align*}
\end{proof}


\subsection{Recurrence relations satisfied by the state probabilities of NGSP}

\begin{theorem}\label{recurrence.ngsp}
    The state probabilities of NGSP satisfy the following recurrence relation
\begin{equation*}
    p(n,t)=\frac{1}{n}\sum_{j=1}^{k}j\left(\Lambda_{j}(t)p(n-j,t)-T_{j}(t)p(n+j,t)\right),\quad n\geq 1\,.
\end{equation*}
\end{theorem}

\begin{proof}
    Using the definition of p.g.f. for NGSP, we have 
\begin{align}
\frac{d}{du}G_{\mathcalboondox{S}}(u,t)&=\frac{d}{du}\sum_{i=0}^{\infty}u^{i}p(i,t)=\sum_{i=0}^{\infty}(i+1)p(i+1,t)u^{i}.\label{32}
\end{align}
Also from \eqref{2}, we have 
\begin{align}
\frac{d}{du}G_{\mathcalboondox{S}}(u,t)
    &=\left(\sum_{j=1}^{k}j\Lambda_{j}(t)u^{j-1}+\sum_{j=1}^{k}(-j)T_{j}(t)u^{-j-1}\right)G_{\mathcalboondox{S}}(u,t).\label{33}
\end{align}
Comparing the RHS of \eqref{32} and \eqref{33}, we get
\begin{align}
    \sum_{i=0}^{\infty}(i+1)p(i+1,t)u^{i}&=\sum_{j=1}^{k}j\left(\Lambda_{j}(t)u^{j-1}-T_{j}(t)u^{-j-1}\right)\left(\sum_{i=0}^{\infty}u^{i}p(i,t)\right)\nonumber\\
    &=\sum_{j=1}^{k}j\Lambda_{j}(t)\left(\sum_{i=j-1}^{k-1}p(i-j+1,t)u^{i}+\sum_{i=k}^{\infty}p(i-j+1,t)u^{i}\right)\nonumber\\
    &-\sum_{j=1}^{k}jT_{j}(t)\left(\sum_{i=-j-1}^{k-1}p(i+j+1,t)u^{i}+\sum_{i=k}^{\infty}p(i+j+1,t)u^{i}\right)\nonumber\\
    &=\sum_{i=0}^{k-1}\sum_{j=1}^{i+1}j\Lambda_{j}(t)p(i-j+1,t)u^{i}+\sum_{i=k}^{\infty}\sum_{j=1}^{k}\Lambda_{j}(t)p(i-j+1,t)u^{i}\nonumber\\
    &-\sum_{i=-k-1}^{k-1}\sum_{j=-i-1}^{k}jT_{j}(t)p(i+j+1,t)u^{i}-\sum_{i=k}^{\infty}\sum_{j=1}^{k}jT_{j}(t)p(i+j+1,t)u^{i}\,\,.\label{34}
\end{align}
Now equating the coefficients of $u^{i}$ for $i\leq k-1$ on both sides of \eqref{34}, we obtain
\begin{equation*}
    (i+1)p(i+1,t)=\sum_{j=1}^{i+1}j\Lambda_{j}(t)p(i-j+1,t)-\sum_{j=-i-1}^{k}jT_{j}(t)p(i+j+1,t)
\end{equation*}
which reduces to 
\begin{equation}
    p(n,t)=\frac{1}{n}\left(\sum_{j=1}^{n}j\Lambda_{j}(t)p(n-j,t)-\sum_{j=-n}^{k}jT_{j}(t)p(n+j,t)\right), \quad -k\leq n\leq k\label{35}
\end{equation}
as in the first summation $p(n-(n+1),t), p(n-(n+2),t),...,p(n-k,t)$ are all vanishing and in the second summation $T_{-n}(t), T_{-n+1}(t),...,T_{-1}(t)$ all are zeroes. Also $p(-i,t)=0$ for $1\leq i\leq k$. 
Again, by equating the coefficients of $u^{i}$ for $i\geq k$ on both sides of \eqref{34}, we obtain
\begin{equation*}
    (i+1)p(i+1,t)=\sum_{j=1}^{k}j\Lambda_{j}(t)p(i-j+1,t)-\sum_{j=1}^{k}jT_{j}(t)p(i+j+1,t)
\end{equation*}
which reduces to
\begin{equation}
    p(n,t)=\frac{1}{n}\left(\sum_{j=1}^{k}j\Lambda_{j}(t)p(n-j,t)-\sum_{j=1}^{k}jT_{j}(t)p(n+j,t)\right), \quad n\geq k+1.\label{36}
\end{equation}
Combining \eqref{35} and \eqref{36}, we have
\begin{equation*}
    p(n,t)=\frac{1}{n}\sum_{j=1}^{k}j\left[\Lambda_{j}(t)p(n-j,t)-T_{j}(t)p(n+j,t)\right], \quad n\geq 1.
\end{equation*}
\end{proof}
The next result follows from Theorem \ref{recurrence.ngsp}.
\begin{proposition}
    The state probabilities of NGCP satisfy the following recurrence relation
\begin{equation*}
    q(n,t)=\frac{1}{n}\sum_{j=1}^{min\{n,k\}}j\Lambda_{j}(t)q(n-j,t),\quad n\geq 1.
\end{equation*}
\end{proposition}

\subsection{Transition probabilities of NGSP}
\begin{theorem}\label{transition.ngsp}
The transition probabilities of NGSP satisfy
\begin{equation*}
    P\left(\mathcalboondox{S}(t+\delta)=m\mid\mathcalboondox{S}(t)=n\right)=\begin{cases}
        \lambda_i(t)\delta+o(\delta), & m>n,\, m=n+i, i=1,2,...k;\\
        \mu_i(t)\delta+o(\delta), & m<n, m=n-i,\, i=1,2,...k;\\
        1-\sum_{i=1}^{k}\lambda_i(t)\delta-\sum_{i=1}^{k}\mu_i(t)\delta+o(\delta), & m=n;\\
        o(\delta), & \text{otherwise},
    \end{cases}
\end{equation*}
i.e., at most $k$ events can occur in a very small interval of time and even though the probability for more than $k$ events is non-zero, it is negligible.
\end{theorem} 

\begin{proof}
Denote $\mathcal{M}_{1}(t)$ and $\mathcal{M}_{2}(t)$ to be the first and second processes respectively in Definition \ref{defngsp}. Then      
\begin{align*}
    P\left(\mathcalboondox{S}(t+\delta)=n+i\mid\mathcalboondox{S}(t)=n\right)
    &=\sum_{j=1}^{k-i} P(\text{the first process has $i+j$ arrivals and the second process has $j$ arrivals}) \\
    &+ P(\text{the first process has $i$ arrivals and the second process has $0$ arrivals}) + o(\delta)\\
    &= \sum_{j=1}^{k-i} \left(\lambda_i(t)\delta+o(\delta)\right)*\left(\mu_i(t)\delta+o(\delta)\right) + \left(\lambda_i(t)\delta+o(\delta)\right)*\left(1-\sum_{i=1}^{k}\mu_i(t)\delta+o(\delta)\right) \\
    &+ o(\delta) =  \lambda_i(t)\delta+o(\delta)\,\,.
\end{align*}
Similarly, we have 
\begin{align*}
P\left(\mathcalboondox{S}(t+\delta)=n-i\mid\mathcalboondox{S}(t)=n\right)
    &=\sum_{j=1}^{k-i} P(\text{the first process has $j$ arrivals and the second process has $i+j$ arrivals}) \\
    &+ P(\text{the first process has $0$ arrivals and the second process has $i$ arrivals}) + o(\delta)\\
    &= \sum_{j=1}^{k-i} \left(\lambda_i(t)\delta+o(\delta)\right)*\left(\mu_i(t)\delta+o(\delta)\right) + \left(1-\sum_{i=1}^{k}\lambda_i(t)\delta+o(\delta)\right)*\left(\mu_i(t)\delta+o(\delta)\right) \\
    &+ o(\delta)=\mu_i(t)\delta+o(\delta)\,\,.
\end{align*}
Finally
\begin{align*}
    P\left(\mathcalboondox{S}(t+\delta)=n\mid\mathcalboondox{S}(t)=n\right)
    &=\sum_{j=1}^{k} P(\text{the first process has $j$ arrivals and the second process has $j$ arrivals}) \\
    &+ P(\text{the first process has $0$ arrivals and the second process has $0$ arrivals}) + o(\delta)\\
    &= \sum_{j=1}^{k} \left(\lambda_i(t)\delta+o(\delta)\right)*\left(\mu_i(t)\delta+o(\delta)\right) + \left(1-\sum_{i=1}^{k}\lambda_i(t)\delta+o(\delta)\right)*\left(1-\sum_{i=1}^{k}\mu_i(t)\delta+o(\delta)\right) \\
    &+ o(\delta) = 1-\sum_{i=1}^{k}\lambda_i(t)\delta -\sum_{i=1}^{k}\mu_i(t)\delta+o(\delta)
\end{align*}
which proves the theorem.
\end{proof}

\subsection{NGSP as a weighted sum}




\begin{theorem}\label{wsum.ngsp}
The NGSP is equal in distribution to a weighted sum of $k$ independent Non-homogeneous Skellam processes (NSPs). 
\end{theorem}

\begin{proof}
Denote the NSPs by $\{\mathcal{S}_{j}(t)\}_{t\geq 0}:=\{\mathcal{N}_{1j}(t)\}_{t\geq 0}-\{\mathcal{N}_{2j}(t)\}_{t\geq 0}$, where $\{\mathcal{N}_{1j}(t)\}_{t\geq 0}$ and $\{\mathcal{N}_{2j}(t)\}_{t\geq 0}$ are independent non-homogeneous Poisson processes with respective cumulative intensity functions $\Lambda_{j}(t)$ and $T_{j}(t)$ for $j=1,2,...,k$. The m.g.f.'s of $\mathcal{N}_{1j}(t)$ and $\mathcal{N}_{2j}(t)$ are given by
\begin{align*}
\mathbb{E}\left(\exp\left(u\mathcal{N}_{1j}(t)\right)\right) &= \exp\left(\Lambda_{j}(t)(e^{u}-1)\right),\quad\mathbb{E}\left(\exp\left(u\mathcal{N}_{2j}(t)\right)\right) = \exp\left(T_{j}(t)(e^{u}-1)\right). 
\end{align*}
So the m.g.f. of the weighted sum $\sum_{j=1}^{k}j\mathcal{S}_{j}(t)$ is
\begin{align}
\mathbb{E}\left(\exp\left(u\sum_{j=1}^{k}j\mathcal{S}_{j}(t)\right)\right)
&=\mathbb{E}\left[\exp\left(u\sum_{j=1}^{k}j\mathcal{N}_{1j}(t)\right)\right]\mathbb{E}\left[\exp\left(-u\sum_{j=1}^{k}j\mathcal{N}_{2j}(t)\right)\right] \nonumber\\
&=\exp\left(\sum_{j=1}^{k}\left(\Lambda_{j}(t)(e^{uj}-1)+T_{j}(t)(e^{-uj}-1)\right)\right).\label{42}
\end{align}
The m.g.f. of NGSP can be obtained from \eqref{2} by replacing `$u$' with `$e^u$' as follows:
\begin{equation}
    \mathbb{E}\left(\exp\left(u\mathcalboondox{S}(t)\right)\right)=\exp\left(\sum_{j=1}^{k}\left(\Lambda_{j}(t)(e^{uj}-1)+T_{j}(t)(e^{-uj}-1)\right)\right)\label{43}
\end{equation}
Therefore from \eqref{42} and \eqref{43}, we get 
\begin{equation*}
    \mathcalboondox{S}(t)\overset{d}{=}\sum_{j=1}^{k}j\mathcal{S}_{j}(t)=\sum_{j=1}^{k}j\left(\mathcal{N}_{1j}(t)-\mathcal{N}_{2j}(t)\right).
\end{equation*}
\end{proof}
The next result follows from Theorem \ref{wsum.ngsp}.
\begin{proposition}
The NGCP is equal in distribution to a weighted sum of $k$ independent Non-homogeneous Poisson processes (NPPs).
\end{proposition}

    
\subsection{Martingale Property of NGSP}
\begin{theorem}\label{martingale}
    The process $\left\{\mathcalboondox{S}(t)-\sum_{j=1}^{k}j\left(\Lambda_{j}(t)-T_{j}(t)\right)\right\}_{t\geq 0}$ is a martingale with respect to the natural filtration $\mathcal{F}_{t}=\sigma\left(\left\{\mathcalboondox{S}(s)\right\},~0<s\leq t\right)$.
\end{theorem}

\begin{proof}
For all $0<s\leq t$, we have 
\begin{align}
\mathbb{E}\left(\mathcalboondox{S}(t)-\sum_{j=1}^{k}j\left(\Lambda_{j}(t)-T_{j}(t)\right)\big|\left\{\mathcalboondox{S}(s)\right\}\right)
    &=\mathbb{E}\left(\mathcal{M}_{1}(t)-\sum_{j=1}^{k}j\Lambda_{j}(t)\big|\left\{\mathcal{M}_{1}(s)-\mathcal{M}_{2}(s)\right\}\right)\nonumber\\
    &-\mathbb{E}\left(\mathcal{M}_{2}(t)-\sum_{j=1}^{k}jT_{j}(t)\big|\left\{\mathcal{M}_{1}(s)-\mathcal{M}_{2}(s)\right\}\right).\label{46}
\end{align}
Using the tower property of conditional expectation, the first term on the RHS of \eqref{46} may be written as 
\begin{align}\label{47}
&\mathbb{E}\left[\mathbb{E}\left(\mathcal{M}_{1}(t)-\sum_{j=1}^{k}j\Lambda_{j}(t)\big|\left\{\mathcal{M}_{1}(s)-\mathcal{M}_{2}(s),\mathcal{M}_{1}(s)\right\}\right)\big|\left\{\mathcal{M}_{1}(s)-\mathcal{M}_{2}(s)\right\}\right] \nonumber \\
    &=\mathbb{E}\left[\mathbb{E}\left(\mathcal{M}_{1}(t)-\sum_{j=1}^{k}j\Lambda_{j}(t)\big|\left\{\mathcal{M}_{1}(s)\right\}\right)\big|\left\{\mathcal{M}_{1}(s)-\mathcal{M}_{2}(s)\right\}\right]
\end{align}
where we have used the conditional independence between $\left\{\mathcal{M}_{1}(t)-\sum_{j=1}^{k}j\Lambda_{j}(t)\right\}_{t\geq 0}$ and $\left\{\mathcal{M}_{1}(s)-\mathcal{M}_{2}(s)\right\}$ given $\left\{\mathcal{M}_{1}(s)\right\}$. Since
 the process $\left\{\mathcal{M}_{1}(t)-\sum_{j=1}^{k}j\Lambda_{j}(t)\right\}_{t\geq 0}$ is a martingale (see \citet{Kataria2025}) with respect to the natural filtration $\mathcal{F}_t=\sigma\left(\mathcal{M}(s),0<s\leq t\right)$, we have
\begin{align}\label{47a}
    \mathbb{E}\left(\mathcal{M}_{1}(t)-\sum_{j=1}^{k}j\Lambda_{j}(t)\big|\left\{\mathcal{M}_{1}(s)\right\}\right)
    &=\mathcal{M}_{1}(s)-\sum_{j=1}^{k}j\Lambda_{j}(s).
\end{align}
Using \eqref{47} and \eqref{47a}, the first term on the RHS of \eqref{46} reduces to
\begin{equation}\label{48}
\mathbb{E}\left(\mathcal{M}_{1}(t)-\sum_{j=1}^{k}j\Lambda_{j}(t)\big|\left\{\mathcal{M}_{1}(s)-\mathcal{M}_{2}(s)\right\}\right) = \mathcal{M}_{1}(s)-\sum_{j=1}^{k}j\Lambda_{j}(s) 
\end{equation}
as $\sigma(\left\{\mathcal{M}_{1}(s)\right\}\subset\sigma(\left\{\mathcal{M}_{1}(s)-\mathcal{M}_{2}(s)\right\})$. Similarly, the second term on the RHS of \eqref{46} reduces to 
\begin{equation}
    \mathbb{E}\left(\mathcal{M}_{2}(t)-\sum_{j=1}^{k}jT_{j}(t)\big|\left\{\mathcal{M}_{1}(s)-\mathcal{M}_{2}(s)\right\}\right)=\mathcal{M}_{2}(s)-\sum_{j=1}^{k}jT_{j}(s).\label{49}
\end{equation}
Substituting \eqref{48} and \eqref{49} in \eqref{46}, we obtain 
\begin{align*}
    \mathbb{E}\left(\mathcalboondox{S}(t)-\sum_{j=1}^{k}j\left(\Lambda_{j}(t)-T_{j}(t)\right)\big|\left\{\mathcalboondox{S}(s)\right\}\right)
    &=\mathcalboondox{S}(s)-\sum_{j=1}^{k}j\left(\Lambda_{j}(s)-T_{j}(s)\right)
\end{align*}
for all $0<s\le t$, thereby proving the result.    
\end{proof}

\subsection{Arrival time and first passage time of NGSP}
\begin{proposition}\label{first passage.ngsp}
     Let $T_n$ be the time of first upcrossing of level $n$ for NGSP given by
\begin{equation*}
    T_n=\text{inf}\{s\geq 0:\mathcalboondox{S}(s)\geq n\}\,.
\end{equation*}
Then
\begin{equation*}
    F_{T_n}(t)= 
    \sum_{i\in (0,t]}\sum_{m=n}^{\infty} e^{-(A+B)}\left(\frac{A}{B}\right)^{\frac{|m|}{2}}I_{|m|}\left(2\sqrt{AB}\right),
\end{equation*}
where $\sum_{j=1}^{k}\Lambda_{j}(i)=A$ and $\sum_{j=1}^{k}T_{j}(i)=B$ for $i\in (0,t]$.
\end{proposition}
\begin{proof} We have
    \begin{align*}
        F_{T_n}(t)=1-P\{T_n > t\}=1-P\left(\bigcap_{i\in (0,t]}\left\{\mathcalboondox{S}(i)\le n-1\right\}\right)
        =P\left(\bigcap_{i\in (0,t]}\left\{\mathcalboondox{S}(i)\le n-1\right\}\right)^c
        &=P\left(\bigcup_{i\in (0,t]}\left\{\mathcalboondox{S}(i)> n-1\right\}\right) \\
        &=\sum_{i\in (0,t]}\sum_{m=n}^{\infty}P\left\{\mathcalboondox{S}(i)=m\right\}.
    \end{align*}
 The result follows using Theorem \ref{20}.
\end{proof}
\begin{proposition}\label{arrival.ngsp}
    Consider the $n^{th} $ arrival time of NGSP or the arrival time of $n^{th}$ NGSP event defined as 
\begin{equation*}
\tau_n=min\{t\geq 0: \mathcalboondox{S}(t)=n\}\,.
\end{equation*}
Then
\begin{align*}
    F_{\tau_n}(t)&=\sum_{i\in(0,t]} e^{-(A+B)}\left(\frac{A}{B}\right)^{\frac{|n|}{2}}I_{|n|}\left(2\sqrt{AB}\right),
\end{align*}
where $\sum_{j=1}^{k}\Lambda_{j}(i)=A$ and $\sum_{j=1}^{k}T_{j}(i)=B$ for $i\in (0,t]$.
\end{proposition}
\begin{proof} We have
    \begin{align*}
        F_{\tau_n}(t)=1-P\left(\tau_n > t\right)=1-P\left(\bigcap_{i\in (0,t]}\left\{\mathcalboondox{S}(i)\neq n\right\}\right)=1-P\left(\bigcap_{i\in (0,t]}\left\{\mathcalboondox{S}(i)= n\right\}^c\right)
        &=P\left(\bigcup_{i\in (0,t]}\left\{\mathcalboondox{S}(i)= n\right\}\right)\\
        &= \sum_{i\in(0,t]}P\left\{\mathcalboondox{S}(i)= n\right\}.
    \end{align*}
   The result follows using Theorem \ref{20}.
\end{proof}

\section{Fractional variant of NGSP}\label{sec 4}
In this section, we introduce the fractional variant of NGSP and study its various characteristics.
\begin{definition}\label{defngfsp}
    We define the fractional variant of NGSP by time changing it with an independent inverse stable subordinator $\{Y_{\alpha}(t)\}_{t\geq 0} $ and call it the Non-homogeneous Generalized Fractional Skellam process (NGFSP) denoted by
\begin{equation*}
\mathcalboondox{S}^{\alpha}(t):=\begin{cases}
\mathcalboondox{S}(Y_{\alpha}(t)) & \text{if}~0<\alpha<1,\\
\mathcalboondox{S}(t) & \text{if}~\alpha=1.
\end{cases}
\end{equation*} 
\end{definition}
The fractional variant of NGSP can also be obtained by taking the difference of two independent NGFCPs (the fractional variant of NGCP) as
\begin{equation*}
    \mathcalboondox{S}^{\alpha}(t):=\begin{cases}
\mathcal{M}_{1}^{\alpha}(t)-\mathcal{M}_{2}^{\alpha}(t) & \text{if}~0<\alpha<1,\\
\mathcalboondox{S}(t) & \text{if}~\alpha=1,
    \end{cases}
\end{equation*}
where $\{\mathcal{M}_1^{\alpha}(t)\}_{t\geq 0}$ and $\{\mathcal{M}_2^{\alpha}(t)\}_{t\geq 0}$ are two independent NGFCPs. 
\begin{remark}
Note that if $\lambda_j(t) = \lambda$ and $\gamma_j(t) =\gamma$ for $1\le j\le k$, then NGFSP reduces to the Fractional Skellam process of order $k$ (FSPoK) introduced in \citet{Kataria2024}.
\end{remark}

\subsection{Probability generating function of NGFSP}
The p.g.f. of NGFSP is given by
\begin{align*}
G_{\mathcalboondox{S}}^{\alpha}(u,t)=& \int_{0}^{\infty}G_{\mathcalboondox{S}}(u,x)h_{\alpha}(x,t)dx
=  \int_{0}^{\infty}\exp\left(\sum_{j=1}^{k}\left(\Lambda_{j}(x)(u^{j}-1)+T_{j}(x)(u^{-j}-1)\right)\right)h_{\alpha}(x,t)dx 
\end{align*}
where $h_{\alpha}(x,t)$ is the density of the inverse stable subordinator $\{Y_{\alpha}(t)\}_{t\geq 0} $.

\subsection{Mean, Variance, and Covariance} We will first find the moments of NGFCP, which will be then used to find the moments of NGFSP. Using the moments of NGCP in Section \ref{ngcp}, the mean of NGFCP is 
\begin{align}
    \mathbb{E}\left(\mathcal{M}^\alpha(t)\right)=\mathbb{E}\left[\mathbb{E}\left(\mathcal{M}\left(Y_\alpha(t)\right)|Y_\alpha(t)\right)\right]
    &=\int_{0}^{\infty}\mathbb{E}\left(\mathcal{M}(x)\right)h_\alpha(x,t)dx
    =\sum_{j=1}^{k}j\mathbb{E}\left[\Lambda_j\left(Y_\alpha(t)\right)\right]\label{expx}
    \end{align}
and the second moment is 
    \begin{align}  \mathbb{E}\left[\left(\mathcal{M}^\alpha(t)\right)^2\right]=\int_{0}^{\infty}\mathbb{E}\left[\left(\mathcal{M}(x)\right)^2\right]h_\alpha(x,t)dx
        &=\int_{0}^{\infty}\left[\mathbb{V}\left(\mathcal{M}(x)\right)+\left(\mathbb{E}\left(\mathcal{M}(x)\right)\right)^2\right]h_\alpha(x,t)dx\nonumber\\
        &=\sum_{j=1}^{k}j^2\mathbb{E}\left[\Lambda_j(Y_\alpha(t))\right]+\mathbb{E}\left(\left[\sum_{j=1}^{k}j\Lambda_j\left(Y_\alpha(t)\right)\right]^2\right).\label{expxsqre}
    \end{align}
    Using \eqref{expx} and \eqref{expxsqre}, the variance of NGFCP is 
\begin{align}
\mathbb{V}\left(\mathcal{M}^\alpha(t)\right)
&=\sum_{j=1}^{k}j^2\mathbb{E}\left[\Lambda_j(Y_\alpha(t))\right]+\mathbb{V}\left(\sum_{j=1}^{k}j\Lambda_j\left(Y_\alpha(t)\right)\right).\label{varx}
\end{align}

Next, using the law of total covariance, we can find the covariance of NGFCP as follows:
\begin{align}
    \text{Cov}\left(\mathcal{M}^\alpha(s), \mathcal{M}^\alpha(t)\right)&=\text{Cov}\left(\mathcal{M}(Y_\alpha(s)), \mathcal{M}(Y_\alpha(t))\right)\nonumber\\
    &=\mathbb{E}\left[\text{Cov}\left(\mathcal{M}(Y_\alpha(s)), \mathcal{M}(Y_\alpha(t))\mid Y_\alpha(s), Y_\alpha(t)\right)\right]\nonumber\\
    &+\text{Cov}\left[\mathbb{E}\left(\mathcal{M}(Y_\alpha(s))\mid Y_\alpha(s), Y_\alpha(t)\right), \mathbb{E}\left(\mathcal{M}(Y_\alpha(t))\mid Y_\alpha(s), Y_\alpha(t)\right)\right]. \label{cov}
\end{align}
The first term on the R.H.S. of \eqref{cov} can be written as 
\begin{align}
&\mathbb{E}\left\{\mathbb{E}\left[\mathcal{M}(Y_\alpha(s))\mathcal{M}(Y_\alpha(t))|Y_\alpha(s), Y_\alpha(t)\right]
-\mathbb{E}\left[\mathcal{M}(Y_\alpha(s))|Y_\alpha(s), Y_\alpha(t)\right]\mathbb{E}\left[\mathcal{M}(Y_\alpha(t))|Y_\alpha(s), Y_\alpha(t)\right]\right\}\nonumber\\
&=\int_{0}^{\infty}\int_{0}^{\infty}\mathbb{E}\left[\mathcal{M}(x)\mathcal{M}(y)\right]f_{\left(Y_\alpha(s), Y_\alpha(t)\right)}(x,y)dxdy-\int_{0}^{\infty}\int_{0}^{\infty}\mathbb{E}\left[\mathcal{M}(x)\right]\mathbb{E}\left[\mathcal{M}(y)\right]f_{\left(Y_\alpha(s), Y_\alpha(t)\right)}(x,y)dxdy\nonumber\\
&=\int_{0}^{\infty}\int_{0}^{\infty} \text{Cov}\left(\mathcal{M}(x), \mathcal{M}(y)\right)f_{\left(Y_\alpha(s), Y_\alpha(t)\right)}(x,y)dxdy\nonumber\\
&=\mathbb{E}\left(\sum_{j=1}^{k}j^2\Lambda_j(Y_\alpha(s)\wedge Y_\alpha(t))\right)\quad (\text{from section}~ \ref{ngcp})\nonumber\\
&=\mathbb{E}\left(\sum_{j=1}^{k}j^2\Lambda_j(Y_\alpha(s\wedge t))\right)\ \quad(\text{since}\,\, \{Y_\alpha(t)\}_{t\geq 0} \,\,\text{is an increasing process}).\label{cov1}
\end{align}
The second term on the R.H.S. of \eqref{cov} can be written as
\begin{align}
&\mathbb{E}\left[\mathbb{E}\left(\mathcal{M}(Y_\alpha(s))|Y_\alpha(s), Y_\alpha(t)\right)\mathbb{E}\left(\mathcal{M}(Y_\alpha(t))|Y_\alpha(s), Y_\alpha(t)\right)\right]   
-\mathbb{E}\left[\mathbb{E}\left(\mathcal{M}(Y_\alpha(s))|Y_\alpha(s), Y_\alpha(t)\right)\right]\mathbb{E}\left[\mathbb{E}\left(\mathcal{M}(Y_\alpha(t))|Y_\alpha(s), Y_\alpha(t)\right)\right]\nonumber\\
&=\int_{0}^{\infty}\int_{0}^{\infty}\mathbb{E}\left[\mathcal{M}(x)\right]\mathbb{E}\left[\mathcal{M}(y)\right]f_{\left(Y_\alpha(s), Y_\alpha(t)\right)}(x,y)dxdy-\mathbb{E}\left[\mathcal{M}(Y_\alpha(s))\right]\mathbb{E}\left[\mathcal{M}(Y_\alpha(t))\right]\nonumber\\
&=\int_{0}^{\infty}\int_{0}^{\infty}\left(\sum_{j=1}^{k}j\Lambda_j(x)\right)\left(\sum_{j=1}^{k}j\Lambda_j(y)\right)f_{\left(Y_\alpha(s), Y_\alpha(t)\right)}(x,y)dxdy-\mathbb{E}\left[\mathcal{M}(Y_\alpha(s))\right]\mathbb{E}\left[\mathcal{M}(Y_\alpha(t))\right]\nonumber\\
&=\mathbb{E}\left[\left(\sum_{j=1}^{k}j\Lambda_j(Y_\alpha(s))\right)\left(\sum_{j=1}^{k}j\Lambda_j(Y_\alpha(t)\right)\right]-\left(\sum_{j=1}^{k}j\mathbb{E}\left[\Lambda_j(Y_\alpha(s))\right]\right)\left(\sum_{j=1}^{k}j\mathbb{E}\left[\Lambda_j(Y_\alpha(t))\right]\right)\nonumber\\&= \text{Cov}\left(\sum_{j=1}^{k}j\Lambda_j(Y_\alpha(s)), \sum_{j=1}^{k}j\Lambda_j(Y_\alpha(t))\right)\,.\label{cov2}
\end{align}
On substituting \eqref{cov1} and \eqref{cov2} in \eqref{cov}, we have
\begin{align}
    \text{Cov}\left(\mathcal{M}^\alpha(s), \mathcal{M}^\alpha(t)\right)=\mathbb{E}\left(\sum_{j=1}^{k}j^2\Lambda_j(Y_\alpha(s\wedge t))\right)+ \text{Cov}\left(\sum_{j=1}^{k}j\Lambda_j(Y_\alpha(s)), \sum_{j=1}^{k}j\Lambda_j(Y_\alpha(t))\right)\,.\label{covfinal}
\end{align}
Using \eqref{expx}, \eqref{varx} and \eqref{covfinal}, we can obtain the mean, variance, and covariance of NGFSP as follows:
\begin{align*}
\mathbb{E}\left(\mathcalboondox{S}^\alpha(t)\right)&=\sum_{j=1}^{k}j\mathbb{E}\left[\Lambda_j\left(Y_\alpha(t)\right)-T_j\left(Y_\alpha(t)\right)\right]\,,\\
\mathbb{V}\left(\mathcalboondox{S}^\alpha(t)\right)& =\sum_{j=1}^{k}j^2\mathbb{E}\left[\Lambda_j(Y_\alpha(t))+T_j(Y_\alpha(t))\right]+\mathbb{V}\left(\sum_{j=1}^{k}j\left[\Lambda_j\left(Y_\alpha(t)\right)+T_j\left(Y_\alpha(t)\right)\right]\right),\\
\text{Cov}\left(\mathcalboondox{S}^\alpha(s), \mathcalboondox{S}^\alpha(t)\right)
&=\mathbb{E}\left(\sum_{j=1}^{k}j^2\left(\Lambda_j(Y_\alpha(s\wedge t))+T_j(Y_\alpha(s\wedge t))\right)\right)+ \text{Cov}\left(\sum_{j=1}^{k}j\Lambda_j(Y_\alpha(s)), \sum_{j=1}^{k}j\Lambda_j(Y_\alpha(t))\right)\nonumber\\
&+ \text{Cov}\left(\sum_{j=1}^{k}jT_j(Y_\alpha(s)), \sum_{j=1}^{k}jT_j(Y_\alpha(t))\right).
\end{align*}
Note that
\begin{align*} 
&\mathbb{V}\left(\mathcalboondox{S}^\alpha(t)\right)-\mathbb{E}\left(\mathcalboondox{S}^\alpha(t)\right)\\&=\sum_{j=1}^{k}j^2\mathbb{E}\left[\Lambda_j(Y_\alpha(t))+T_j(Y_\alpha(t))\right]+\mathbb{V}\left(\sum_{j=1}^{k}j\left[\Lambda_j\left(Y_\alpha(t)\right)+T_j\left(Y_\alpha(t)\right)\right]\right)-\sum_{j=1}^{k}j\mathbb{E}\left[\Lambda_j\left(Y_\alpha(t)\right)-T_j\left(Y_\alpha(t)\right)\right]\\
    &=\sum_{j=1}^{k}(j^2-j)\mathbb{E}\left[\Lambda_j(Y_\alpha(t))\right]+\sum_{j=1}^{k}(j^2+j)\mathbb{E}\left[T_j(Y_\alpha(t))\right]+\mathbb{V}\left(\sum_{j=1}^{k}j\left[\Lambda_j\left(Y_\alpha(t)\right)+T_j\left(Y_\alpha(t)\right)\right]\right)\geq 0,
\end{align*}    
which implies NGFSP exhibits over-dispersion. Similarly, it can be shown that NGFCP is also over-dispersed. 

\subsection{Factorial moments for NGFSP}
Using the factorial moments of NGSP and the p.g.f. of NGFSP, we can obtain the factorial moments for NGFSP. 
\begin{theorem}
    The $r^{th}$ factorial moment of NGFSP for $r\geq 1$ is given by 
\begin{equation*}
    \Psi_{\mathcalboondox{S}^{\alpha}}(r,t)=\sum_{m=0}^{r-1}\binom{r-1}{m}\frac{d^{m}}{du^{m}}G^{\alpha}_{\mathcalboondox{S}}(u,t)\bigg|_{u=1}\left\{\sum_{j=1}^{k}\left[\Lambda_{j}(t)P(j,r-m)+T_{j}(t)P(-j,r-m)\right]\right\}.
\end{equation*}
\end{theorem}

\begin{proof}
The proof is similar to that of Theorem \ref{factorial.ngsp} and hence omitted.
\end{proof}

\subsection{Transition and state probabilities for NGFSP}
The following results on transition probabilities and recurrence relations for state probabilities of NGFSP correspond to those for NGSP.
\begin{theorem}
The transition probabilities of NGFSP satisfy
\begin{equation*}
    P\left(\mathcalboondox{S}^{\alpha}(t+\delta)=m\mid\mathcalboondox{S}^{\alpha}(t)=n\right)=\begin{cases}
        \lambda_i(t)\delta+o(\delta), & m>n,\, m=n+i, i=1,2,...k;\\
        \mu_i(t)\delta+o(\delta), & m<n, m=n-i,\, i=1,2,...k;\\
        1-\sum_{i=1}^{k}\lambda_i(t)\delta-\sum_{i=1}^{k}\mu_i(t)\delta+o(\delta), & m=n;\\
        o(\delta), & \text{otherwise},
    \end{cases}
\end{equation*}
\end{theorem} 

\begin{proof}
The proof is similar to that of Theorem \ref{transition.ngsp} and hence omitted.
\end{proof}

\begin{theorem}\label{diffeq for NGFSP}
The state probabilities $p^{\alpha}(n,t)$ of NGFSP satisfy the following recurrence relation
\begin{equation*}
    p^{\alpha}(n,t)=\frac{1}{n}\sum_{j=1}^{k}j\left(\Lambda_{j}(t)p^{\alpha}(n-j,t)-T_{j}(t)p^{\alpha}(n+j,t)\right),\quad n\geq 1\,.
\end{equation*}
Moreover for $ n\in \mathbb{Z}$, the state probabilities satisfy the following system of fractional differential equations with initial conditions $p^{\alpha}(0,0)=1 \,\,\text{and}\,\, p^{\alpha}(n,0)=0,\,\, n\neq 0$:
\begin{multline}
\frac{d^\alpha}{dt^\alpha}p^\alpha(n,t)=\int_{0}^{\infty}\left(\frac{1}{2}p(n-1,t)\sum_{j=1}^{k}\left(\lambda_{j}(t)+\frac{A}{B}\gamma_{j}(t)\right)-p(n,t)\sum_{j=1}^{k}\left(\lambda_{j}(t)+\gamma_{j}(t)\right)
\right.\\\left.
     +\frac{n}{2}p(n,t)\sum_{j=1}^{k}\left(\frac{1}{A}\lambda_{j}(t)-\frac{1}{B}\gamma_{j}(t)\right)+\frac{1}{2}p(n+1,t)\sum_{j=1}^{k}\left(\gamma_{j}(t)+\frac{B}{A}\lambda_{j}(t)\right)\right)h_\alpha(u,t)du\,.\nonumber
\end{multline}
\end{theorem}

\begin{proof}
The proof for recurrence relations satisfied by $p^{\alpha}(n,t)$ is similar to the proof of Proposition \ref{recurrence.ngsp} and hence omitted. Next from Definition \ref{defngfsp}, we have
\begin{align}
p^{\alpha}(n,t)=P(\mathcalboondox{S}(Y_{\alpha}(t))=n)=\int_{0}^{\infty}P(\mathcalboondox{S}(u)=n)h_{\alpha}(u,t)du=\int_{0}^{\infty}p(n,u)h_{\alpha}(u,t)du\label{53}
\end{align}
where $p(n,t)$ is the p.m.f. of $\{\mathcalboondox{S}(t)\}_{t\geq 0}$. Note that $p^{\alpha}(n,0)=p(n,0)$ as $h_{\alpha}(u,0)=\delta_{0}(u)$. Taking R-L fractional derivative on both sides of \eqref{53}, we get
\begin{align*}
D_{t}^{\alpha}p^{\alpha}(n,t)=-\int_{0}^{\infty}p(n,u)\frac{\partial}{\partial u}h_{\alpha}(u,t)du
&= -p(n,u)h_{\alpha}(u,t)\big|_{0}^{\infty}+\int_{0}^{\infty}h_{\alpha}(u,t)\frac{d}{du}p(n,u)du\\
&=p(n,0)h_{\alpha}(0+,t)+\int_{0}^{\infty}h_{\alpha}(u,t)\frac{d}{du}p(n,u)du\\
&=p(n,0)\frac{t^{-\alpha}}{\Gamma(1-\alpha)}+\int_{0}^{\infty}h_{\alpha}(u,t)\frac{d}{du}p(n,u)du
\end{align*}
where we have used the fact
$D_{t}^{\alpha}h_{\alpha}(u,t)=-\frac{\partial}{\partial u}h_{\alpha}(u,t)$ and $h_{\alpha}(0+,t)=\frac{t^{-\alpha}}{\Gamma(1-\alpha)}$ (see \citet{Meerschaert2013}). Now using the relationship between R-L Fractional derivative and the Caputo–Djrbashian fractional derivative (see Section \ref{caputo}) and Proposition \ref{ngsp.de}, we obtain
\begin{align*}
\frac{d^\alpha}{dt^\alpha}p^\alpha(n,t)&=\int_{0}^{\infty}h_\alpha(u,t)\frac{d}{du}p(n,u)du \\
&=\int_{0}^{\infty}\left[\frac{1}{2}p(n-1,t)\sum_{j=1}^{k}\left(\lambda_{j}(t)+\frac{A}{B}\gamma_{j}(t)\right)-p(n,t)\sum_{j=1}^{k}\left(\lambda_{j}(t)+\gamma_{j}(t)\right)
\right.\\
&\left.+\frac{n}{2}p(n,t)\sum_{j=1}^{k}\left(\frac{1}{A}\lambda_{j}(t)-\frac{1}{B}\gamma_{j}(t)\right)+\frac{1}{2}p(n+1,t)\sum_{j=1}^{k}\left(\gamma_{j}(t)+\frac{B}{A}\lambda_{j}(t)\right)\right]h_\alpha(u,t)du.
\end{align*}
\end{proof}

\begin{remark}
The density of the inverse stable subordinator in terms of the Wright function $M_\alpha(.)$  is given by (see \citet{Meerschaert2015})

\begin{equation*}
h_\alpha(x,t)=\frac{1}{t^\alpha}M_\alpha\left(\frac{x}{t^\alpha}\right),  
\qquad M_\alpha(z)=\sum_{n=0}^{\infty}\frac{(-z)^n}{n!\Gamma(1-n\alpha-\alpha)}.
\end{equation*}
Hence from Theorem \ref{20}, the integral representation of the p.m.f. of NGFSP is
\begin{align*}
p^{\alpha}(n,t)&=\int_{0}^{\infty}p(n,u)h_{\alpha}(u,t)du=t^{-\alpha}\left(\frac{A}{B}\right)^{\frac{n}{2}}\int_{0}^{\infty}e^{-(A+B)}I_{|n|}(2\sqrt{AB})M_{\alpha}\left(ut^{-\alpha}\right)du,\quad n\in \mathbb{Z}.
\end{align*} 
\end{remark}

\subsection{NGFSP as a weighted sum}
A non-homogeneous fractional Skellam process is obtained from a non-homogeneous Skellam process by time changing it with an independent inverse stable subordinator. Using the weighted sum representation for NGSP, we can derive the following representation for NGFSP.
\begin{proposition}\label{wsum.ngfsp}
The NGFSP is equal in distribution to the weighted sum of $k$ independent Non-homogeneous fractional Skellam processes (NFSPs). 
\end{proposition}

\subsection{Arrival time and first passage time of NGFSP}
\begin{proposition}\label{first passage.ngfsp}
Let $T_n^{\alpha}$ be the time the first upcrossing of the level $n$ for NGFSP given by
\begin{equation*}
    T_n^{\alpha}=\text{inf}\{s\geq 0:\mathcalboondox{S}^{\alpha}(s)\geq n\}\,.
\end{equation*}
Then 
\begin{equation*}
    F_{T_n^\alpha}(t)= \sum_{i\in (0,t]}\sum_{m=n}^{\infty}\int_{0}^{\infty}e^{-(A+B)}\left(\frac{A}{B}\right)^{\frac{|m|}{2}}I_{|m|}\left(2\sqrt{AB}\right)h_\alpha(u,i)du,
\end{equation*}
where $\sum_{j=1}^{k}\Lambda_{j}(u)=A$ and $\sum_{j=1}^{k}T_{j}(u)=B$.
\end{proposition}

\begin{proof}We have 
    \begin{align*} 
    F_{T^{\alpha}_n}(t)=1-P\{T_n^{\alpha}> t\}=1-P\left(\bigcap_{i\in (0,t]}\left\{\mathcalboondox{S}^\alpha(i)<n\right\}\right)
    =P\left(\bigcap_{i\in (0,t]}\left\{\mathcalboondox{S}^\alpha(i)\le n-1\right\}\right)^c
    &=P\left(\bigcup_{i\in (0,t]}\left\{\mathcalboondox{S}^\alpha(i) > n-1\right\}\right)\\
    &= \sum_{i\in (0,t]}\sum_{m=n}^{\infty}P\left\{\mathcalboondox{S}^\alpha(i)= m\right\}\\
    &=\sum_{i\in (0,t]}\sum_{m=n}^{\infty}\int_{0}^{\infty}p(m,u)h_\alpha(u,i)du.
\end{align*}
The result follows from Theorem \ref{20}. 
\end{proof}

\begin{proposition}\label{arrival.ngfsp}
Let $\tau_n^\alpha$ be the $n^{th}$ arrival time of NGFSP defined as $\tau_n^\alpha=min\{t\geq 0: \mathcalboondox{S}^\alpha(t)=n\}$. Then 
\begin{equation*}
F_{\tau_n^\alpha}(t)= \sum_{i\in(0,t]}\int_{0}^{\infty}e^{-(A+B)}\left(\frac{A}{B}\right)^{\frac{|n|}{2}}I_{|n|}\left(2\sqrt{AB}\right)h_\alpha(u,i)du,
\end{equation*}
where $A=\sum_{j=1}^{k}\Lambda_j(u)$ and $B=\sum_{j=1}^{k}T_j(u)$.
\end{proposition}
\begin{proof} We have
\begin{align*}
F_{\tau_n^\alpha}(t)=1-P\left(\tau_n^\alpha > t\right)=1-P\left(\bigcap_{i\in (0,t]}\left\{\mathcalboondox{S}^\alpha(i)\neq n\right\}\right)
        =1-P\left(\bigcap_{i\in (0,t]}\left\{\mathcalboondox{S}^\alpha(i)= n\right\}^c\right)
        &=P\left(\bigcup_{i\in (0,t]}\left\{\mathcalboondox{S}^\alpha(i)= n\right\}\right)\\
        &= \sum_{i\in(0,t]}P\left\{\mathcalboondox{S}^\alpha(i)= n\right\}\\
        &=\sum_{i\in(0,t]}\int_{0}^{\infty}p(n,u)h_\alpha(u,i)du.
\end{align*}
The result follows from Theorem \ref{20}.
\end{proof}


\subsection{Renewal property}
Let $\{N_{\lambda}(t)\}_{t\ge 0}$ be a homogeneous Poisson process with rate $\lambda$. From \citet{Leonenko2017}, we have $\mathcal{N}(t)=N_{1}(\Lambda(t))$ where $\{\mathcal{N}(t)\}_{t\ge 0}$ is a NPP with cumulative rate function $\Lambda(t)$. As an extension, the NGCP can be written in terms of a GCP as $\mathcal{M}(t)=M_{\bm{1}}\left(\bm{\Lambda}(t)\right)$ where $\bm{1}=(1,1,\ldots,1)$ is a vector of length $k$, $\bm{\Lambda}(t)=\left(\Lambda_1(t),\Lambda_2(t),\ldots,\Lambda_k(t)\right)$
and $\{M_{\bm{1}}(t)\}_{t\ge 0}$ denotes a GCP with all component rates equal to unity, that is, $\lambda_1=\lambda_2=\cdots=\lambda_k=1$. Note that this is also evident from the p.m.f.s of GCP and NGCP (see \citet{Kataria2025}). Next, from \citet{Leonenko2017}, we have $\mathcal{N}^{\alpha}(t)=N_{1}(\Lambda(Y_{\alpha}(t))$ where $\{\mathcal{N}^{\alpha}(t)=\mathcal{N}(Y_{\alpha}(t)\}_{t\ge 0}$ is the fractional variant of $\{\mathcal{N}(t)\}_{t\ge 0}$. Hence a NGFCP can be represented as 
    $\mathcal{M}^\alpha(t)=\mathcal{M}(Y_\alpha(t))=M_{\bm{1}}\left(\bm{\Lambda}(Y_\alpha(t))\right)$ with  $\bm{\Lambda}(Y_\alpha(t))=\left(\Lambda_1(Y_\alpha(t)),\Lambda_2(Y_\alpha(t)),\cdots,\Lambda_k(Y_\alpha(t))\right)$.
Let $\{\mathcal{F}_t^{\mathcal{M}^\alpha}\}_{t\ge 0}$ be the natural filtration of the NGFCP $\{\mathcal{M}^\alpha(t)\}_{t\ge 0}$ where
\begin{equation*}
    \mathcal{F}_t^{\mathcal{M}^\alpha}:=\sigma\left(\{\mathcal{M}^\alpha(s)\}, 0<s\le t\right)
\end{equation*}
and define the non-trivial initial history as
\begin{equation*}
    \mathcal{F}_0:=\sigma\left(\{Y_\alpha(t)\},t\ge 0\right).
\end{equation*}
The overall filtration $\{\mathcal{F}_t\}_{t\ge 0}$ or the intrinsic history is then given by
\begin{equation}
    \mathcal{F}_t:=\mathcal{F}_0\vee\mathcal{F}_t^{\mathcal{M}^\alpha}.\label{filtration}
\end{equation}
Since NGFCP and NGCP are not classical renewal processes, we use the framework of Cox processes (conditional Poisson processes) as defined in \citet{Leonenko2017} to establish their renewal properties.
\begin{definition}\label{cox_definition}
    Let $(\Omega ,\mathcal{F},P)$ be a probability space and $\{N(t))\}_{t\ge 0}$ be a point process adapted to a filtration $\{\mathcal{F}_t^N\}_{t\ge 0}$. Then $\{N(t))\}_{t\ge 0}$ is a Cox process if there exist a right-continuous, increasing process $\{A(t)\}_{t\ge 0}$ such that for any $0 < s <t$
    \begin{align*}
        P\left(N(t)-N(s)=k\big|\mathcal{F}_t\right)&=e^{-(A(t)-A(s))}\frac{(A(t)-A(s))^k}{k!},\quad k=0,1,2,\ldots,
    \end{align*}
    where $\mathcal{F}_t:=\mathcal{F}_0\vee \mathcal{F}_t^N$ and $\mathcal{F}_0=\sigma(\{A(t)\},t\ge 0)$.
    The Cox process $N$ is said to be directed by $A$ also known as its compensator. 
\end{definition}
The next result shows that NGFCP adapted to a suitable filtration is a Cox process.
\begin{proposition}\label{cox_ngfcp}
    A NGFCP adapted to the filtration $\{\mathcal{F}_t\}_{t\ge 0}$ as in \eqref{filtration} is a $\mathcal{F}_t$-Cox process directed by $\{\sum_{j=1}^{k}\Lambda_j\left(Y_\alpha(t)\right)\}_{t\ge 0}$.
\end{proposition}
\begin{proof}
    Let $\{\mathcal{M}_{\mathcal{F}}^\alpha(t)\}_{t\ge 0}$ be a NGFCP adapted to $\{\mathcal{F}_t\}_{t\ge 0}$. As $\{Y_\alpha(t)\}_{t\ge 0}$ is $\mathcal{F}_0$-measurable, the m.g.f. of $\left(\mathcal{M}_{\mathcal{F}}^\alpha(t)-(\mathcal{M}_{\mathcal{F}}^\alpha(s)\right)\big|\mathcal{F}_s$ is given by
    \begin{align}
        \mathbb{E}\left[e^{u\left(\mathcal{M}_{\mathcal{F}}^\alpha(t)-\mathcal{M}_{\mathcal{F}}^\alpha(s)\right)}\big|\mathcal{F}_s\right]&=\mathbb{E}\left[e^{u\left(\mathcal{M}_{\mathcal{F}}^\alpha(t)-\mathcal{M}_{\mathcal{F}}^\alpha(s)\right)}\big|\mathcal{F}_0\vee\mathcal{F}_t^{\mathcal{M}^\alpha}\right]\nonumber\\
        &=\mathbb{E}\left[e^{u\left(M_{\bm{1}}\left(\bm{\Lambda}(Y_\alpha(t))\right)-M_{\bm{1}}\left(\bm{\Lambda}(Y_\alpha(s))\right)\right)}\big|\mathcal{F}_0\vee\mathcal{F}_{\bm{\Lambda}(Y_\alpha(s))}^{M_{\bm{1}}}\right]\nonumber\\
        &=\mathbb{E}\left[e^{u\left(M_{\bm{1}}\left(\bm{\Lambda}(Y_\alpha(t))\right)-M_{\bm{1}}\left(\bm{\Lambda}(Y_\alpha(s))\right)\right)}\big|\mathcal{F}_0\right] \label{renewal_1}\\
        &=e^{\sum_{j=1}^{k}\Lambda_j\left(Y_\alpha(s),Y_\alpha(t)\right)(e^{u}-1)}\nonumber
    \end{align}
    where in \eqref{renewal_1}, we used the fact that a GCP has independent increments. This means that conditioned on $\{\mathcal{F}_t\}_{t\ge 0}$, $\{\mathcal{M}_{\mathcal{F}}^\alpha(t)\}_{t\ge 0}$ has independent increments and
    \begin{equation*}
       \left(\mathcal{M}_{\mathcal{F}}^\alpha(t)-\mathcal{M}_{\mathcal{F}}^\alpha(s)\right)\big| \mathcal{F}_s\sim \text{Poi}\left(\sum_{j=1}^{k}\Lambda_j\left(Y_\alpha(s),Y_\alpha(t)\right)\right)\overset{d}{=}\text{Poi}\left(\sum_{j=1}^{k}\left[\Lambda_j(Y_\alpha(t))-\Lambda_j(Y_\alpha(s))\right]\right).
    \end{equation*}
    Therefore, by definition \ref{cox_definition}, $\{\mathcal{M}_{\mathcal{F}}^\alpha(t)\}_{t\ge 0}$ is a $\mathcal{F}_t$-Cox process directed by $\{\sum_{j=1}^{k}\Lambda_j\left(Y_\alpha(t)\right)\}_{t\ge 0}$.
\end{proof}
\begin{remark}
    For $k=1$, Proposition \ref{cox_ngfcp} reduces to Proposition 1 of \citet{Leonenko2019} which gives the Cox representation of a Fractional Non-homogeneous Poisson process (FNPP).
\end{remark}
The following result gives the Cox representation of a NGCP.
\begin{proposition}\label{cox_ngcp}
    A NGCP adapted to the filtration $\{\mathcal{F}_t=\mathcal{F}_0\vee \mathcal{F}_t^{\mathcal{M}}\}_{t\ge 0}$ where $\mathcal{F}_t^{\mathcal{M}}:=\sigma\left(\{\mathcal{M}(s)\},0<s\le t\right)$ is a $\mathcal{F}_t$-Cox process directed by $\{\sum_{j=1}^{k}\Lambda_j(t)\}_{t\ge 0}$.
\end{proposition}
\begin{proof}
    The proof is similar to that of Proposition \ref{cox_ngfcp} and uses the fact that $Y_{\alpha}(t)=t$ for $\alpha=1$ (see \citet{Maheshwari2019}).
\end{proof}
It follows from Propositions \ref{cox_ngfcp} and \ref{cox_ngcp} that NGFCP and NGCP adapted to suitable filtrations are renewal processes.  
\begin{remark}
It is clear that NGSP and NGFSP are not classical renewal processes. So Lemma 1 and Lemma 2 of \cite{Leonenko2017} are not applicable. Also since the m.g.f.s of NGSP and NGFSP are different from the m.g.f. of a Poisson process irrespective of any conditioning, they are not Cox processes and hence do not have renewal properties. 
\end{remark}
\subsection{Martingale property of NGFSP}
\begin{theorem}\label{martingale_ngfcp}
    Let $\{\mathcal{M}^\alpha_{\mathcal{F}}(t)\}_{t\ge 0}$ be a NGFCP adapted to the filtration $\{\mathcal{F}_t\}_{t\ge 0}$ as defined in \eqref{filtration}. Assume $\mathbb{E}\left[\sum_{j=1}^{k}\Lambda_j(Y_\alpha(t))\right]<\infty$ for $t\ge 0$. Then the stochastic process $\{\mathcal{M}^\alpha_{\mathcal{F}}(t)-\sum_{j=1}^{k}\Lambda_j\left(Y_\alpha(t)\right)\}_{t\ge 0}$ is an $\mathcal{F}_t$-martingale.  
\end{theorem}
\begin{proof} Let $\{\mathcal{M}^\alpha_{\mathcal{F}}(t)\}_{t\ge 0}$ be the NGFCP adapted to its natural filtration $\{\mathcal{F}_t:=\sigma\left(\{\mathcal{M}^{\alpha}(t)\},0<s\le t\right)\}_{t\ge 0}$ and let $\mathcal{F}_{\infty}:=\sigma\left(\mathcal{F}_u:u\in\mathbb{R}^{+}\right)$. 
    Consider
    \begin{align}
        \mathbb{E}\left(\mathcal{M}^\alpha_{\mathcal{F}}(t)-\sum_{j=1}^{k}\Lambda_j\left(Y_\alpha(t)\right)\big|\mathcal{F}_s\right)&=\mathbb{E}\left(\mathbb{E}\left(\mathcal{M}^\alpha_{\mathcal{F}}(s)+\left(\mathcal{M}^\alpha_{\mathcal{F}}(t)-\mathcal{M}^\alpha_{\mathcal{F}}(s)\right)-\sum_{j=1}^{k}\Lambda_j\left(Y_\alpha(t)\right)\big|\mathcal{F}_s\vee\mathcal{F}_{\infty} \right)\big|\mathcal{F}_s\right)\nonumber\\
        &=\mathbb{E}\left(\mathcal{M}^\alpha_{\mathcal{F}}(s)-\sum_{j=1}^{k}\Lambda_j\left(Y_\alpha(t)\right)+\mathbb{E}\left[\left(\mathcal{M}^\alpha_{\mathcal{F}}(t)-\mathcal{M}^\alpha_{\mathcal{F}}(s)\right)\big|\mathcal{F}_s\vee\mathcal{F}_{\infty} \right]\big|\mathcal{F}_s\right)\nonumber\\
        &=\mathbb{E}\left(\mathcal{M}^\alpha_{\mathcal{F}}(s)-\sum_{j=1}^{k}\Lambda_j\left(Y_\alpha(t)\right)+\sum_{j=1}^{k}\Lambda_j\left(Y_\alpha(t)\right)-\sum_{j=1}^{k}\Lambda_j\left(Y_\alpha(s)\right)\big|\mathcal{F}_s\right)\label{martingale_property_ngfcp_1}\\
        &=\mathbb{E}\left(\mathcal{M}^\alpha_{\mathcal{F}}(s)-\sum_{j=1}^{k}\Lambda_j\left(Y_\alpha(s)\right)\big|\mathcal{F}_s \right)\nonumber\\
        &=\mathcal{M}^\alpha_{\mathcal{F}}(s)-\sum_{j=1}^{k}\Lambda_j\left(Y_\alpha(s)\right),\nonumber
    \end{align}
where in \eqref{martingale_property_ngfcp_1} we used the fact that $\left(\mathcal{M}_{\mathcal{F}}^\alpha(t)-\mathcal{M}_{\mathcal{F}}^\alpha(s)\right)\big| \mathcal{F}_s\vee\mathcal{F}_{\infty}\sim\text{Poi}\left(\sum_{j=1}^{k}\left[\Lambda_j(Y_\alpha(t))-\Lambda_j(Y_\alpha(s))\right]\right)$.
\end{proof}
\begin{remark}
    For $k=1$, Theorem \ref{martingale_ngfcp} reduces to Proposition 2 of \citet{Leonenko2019}.
\end{remark}
\begin{theorem}
    Let $\{\mathcalboondox{S}_{\mathcal{F}}^{\alpha}=\mathcal{M}^\alpha_{\mathcal{F}_1}(t)-\mathcal{M}^\alpha_{\mathcal{F}_2}(t)\}_{t\ge 0}$ be a NGFSP adapted to the filtration $\{\mathcal{F}_t\}_{t\ge 0}$ defined as $\mathcal{F}_{t}:=\sigma\left(\{Y_\alpha(t)\},t\ge 0\right)\vee \sigma\left(\{\mathcal{M}_{\mathcal{F}_1}^\alpha(t),\,\mathcal{M}_{\mathcal{F}_2}^\alpha(t)\}\right)$, where $\{\mathcal{M}^\alpha_{\mathcal{F}_1}(t)\}_{t\ge 0}$ and $\{\mathcal{M}^\alpha_{\mathcal{F}_2}(t)\}_{t\ge 0}$ are two independent NGFCPs. Assume $\mathbb{E}\left(\sum_{j=1}^{k}\left[\Lambda_j(Y_\alpha(t))-T_j(Y_\alpha(t))\right]\right)<\infty$ for $t\ge 0$. Then the stochastic process $\{\mathcalboondox{S}^\alpha_{\mathcal{F}}(t)-\sum_{j=1}^{k}\left[\Lambda_j(Y_\alpha(t))-T_j(Y_\alpha(t))\right]\}_{t\ge 0}$ is an $\mathcal{F}_t$-martingale.
\end{theorem}
\begin{proof}
     From Theorem \ref{martingale_ngfcp}, it follows that both the processes $\{\mathcal{M}^\alpha_{\mathcal{F}_1}(t)-\sum_{j=1}^{k}\Lambda_j\left(Y_\alpha(t)\right)\}_{t\ge 0}$ and $\{\mathcal{M}^\alpha_{\mathcal{F}_2}(t)-\sum_{j=1}^{k}T_j\left(Y_\alpha(t)\right)\}_{t\ge 0}$ are $\mathcal{F}_t$-martingales, where $\mathcal{F}_{t}:=\sigma\left(\{Y_\alpha(t)\},t\ge 0\right)\vee \sigma\left(\{\mathcal{M}_{\mathcal{F}_1}^\alpha(t),\,\mathcal{M}_{\mathcal{F}_2}^\alpha(t)\}\right)$.
    Since a linear combination of martingales defined on the same filtration is also a martingale, the difference of the above processes, that is, $\{\mathcalboondox{S}_{\mathcal{F}}^{\alpha}(t) -\sum_{j=1}^{k}\left[\Lambda_j(Y_\alpha(t))-T_j(Y_\alpha(t))\right]\}_{t\ge 0}$ is a $\mathcal{F}_t$-martingale.
\end{proof}

\section{Increment processes of NGSP and NGFSP}\label{sec 5}
In this section, we introduce the increment processes of NGSP and the NGFSP, and study their marginals and governing equations.
\begin{definition}
 The increment process of NGSP is a stochastic process $\{I(t,v)\}_{t\geq 0}$ for $v\geq 0$ defined as
\begin{align*}
I(t,v)&=\mathcalboondox{S}(t+v)-\mathcalboondox{S}(v)
=\left(\mathcal{M}_{1}(t+v)-\mathcal{M}_{1}(v)\right)-\left(\mathcal{M}_{2}(t+v)-\mathcal{M}_{2}(v)\right)= I_1(t,v)-I_2(t,v)\label{62}
\end{align*}
where $\{I_1(t,v)\}_{t\geq 0}$ and $\{I_2(t,v)\}_{t\geq 0}$ are the increment processes of NGCPs $\{\mathcal{M}_1(t)\}_{t\geq 0}$ and $\{\mathcal{M}_2(t)\}_{t\geq 0}$.
\end{definition}

\begin{theorem}\label{63}
The marginals of the increment process of NGSP are given by
\begin{equation*}
q_l(t,v)=P\left(I(t,v)=l\right)=e^{-(A'+B')}\left(\frac{A'}{B'}\right)^{\frac{|n|}{2}}I_{|n|}\left(2\sqrt{A'B'}\right).
\end{equation*}
where
\begin{equation*}
\sum_{j=1}^{k}\Lambda_j(v,t+v)=A' \quad \text{and}\quad \sum_{j=1}^{k}T_j(v,t+v)=B'.
\end{equation*} 
\end{theorem} 

\begin{proof}
Note that
\begin{align*}
q_l(t,v)&=P\left(I_1(t,v)-I_2(t,v)=l\right)\nonumber\\
&=\sum_{m=0}^{\infty}P\left(I_1(t,v)=m+l\right)P\left(I_2(t,v)=m\right)I_{\{l\geq 0\}}+\sum_{m=0}^{\infty}P\left(I_2(t,v)=m+|l|\right)P\left(I_1(t,v)=m\right)I_{\{l < 0\}}\\
&=\sum_{m=0}^{\infty}\left(\sum_{\Omega(k,m+l)}\prod_{j=1}^{k}\frac{(\Lambda_{j}(v,t+v))^{x_j}}{x_j!}e^{-\sum_{j=1}^{k}\Lambda_j(v,t+v)}\right)\left(\sum_{\Omega(k,m)}\prod_{j=1}^{k}\frac{(T_{j}(v,t+v))^{x_j}}{x_j!}e^{-\sum_{j=1}^{k}T_{j}(v,t+v)}\right)I_{\{l\geq 0\}}\\
&+\sum_{m=0}^{\infty}\left(\sum_{\Omega(k,m+l)}\prod_{j=1}^{k}\frac{(\Lambda_{j}(v,t+v))^{x_j}}{x_j!}e^{-\sum_{j=1}^{k}\Lambda_j(v,t+v)}\right)\left(\sum_{\Omega(k,m)}\prod_{j=1}^{k}\frac{(T_{j}(v,t+v))^{x_j}}{x_j!}e^{-\sum_{j=1}^{k}T_{j}(v,t+v)}\right)I_{\{l< 0\}}.\nonumber
\end{align*}
Following the steps of the proof for Theorem \ref{20}, it can be shown that
\begin{equation}
q_l(t,v)=e^{-(A'+B')}\left(\frac{A'}{B'}\right)^{\frac{|n|}{2}}I_{|n|}\left(2\sqrt{A'B'}\right),\quad l\in \mathbb{Z}\,.\nonumber
\end{equation}
\end{proof}

\begin{definition}
The increment process of NGFSP is obtained by time changing the increment process of NGSP with an independent inverse stable subordinator as follows:
\begin{align*}
I_\alpha(t,v)&=\mathcalboondox{S}^\alpha(t+v)-\mathcalboondox{S}^\alpha(v)=\mathcalboondox{S}\left(Y_\alpha(t)+v\right)-\mathcalboondox{S}(v).
\end{align*}
\end{definition}
\begin{proposition}\label{65}
The marginals of the increment process of NGFSP are given by
\begin{align*}
q_{l}^{\alpha}(t,v)&=P\left(\mathcalboondox{S}\left(Y_\alpha(t)+v\right)-\mathcalboondox{S}(v)=l\right)=\int_{0}^{\infty}q_{l}(t,u)h_\alpha(u,t)du,  \quad l\in \mathbb{Z}\,,
\end{align*}
where $q_{l}(t,u)$ is given in Theorem \ref{63}. 
\end{proposition}

\begin{remark}
Putting $v=0$ in Theorem \ref{63} and Proposition \ref{65}, we get the p.m.f.s of NGSP (see Theorem \ref{20}) and NGFSP (see \eqref{53}) respectively.
\end{remark} 

\begin{theorem}\label{diffeq for incre NGFSP}
The p.m.f. of the increment process of NGFSP satisfies the following system of fractional differential-integral equations:
    \begin{multline*}
        \frac{d^\alpha}{dt^\alpha}q_{l}^{\alpha}(t,v)=\int_{0}^{\infty}\left(\frac{1}{2}p(n-1,t)\sum_{j=1}^{k}\left(\lambda_{j}(t)+\frac{A}{B}\gamma_{j}(t)\right)-p(n,t)\sum_{j=1}^{k}\left(\lambda_{j}(t)+\gamma_{j}(t)\right)
        \right. \\ \left.
     +\frac{n}{2}p(n,t)\sum_{j=1}^{k}\left(\frac{1}{A}\lambda_{j}(t)-\frac{1}{B}\gamma_{j}(t)\right)+\frac{1}{2}p(n+1,t)\sum_{j=1}^{k}\left(\gamma_{j}(t)+\frac{B}{A}\lambda_{j}(t)\right)\right)h_{\alpha}(u,t)du.
    \end{multline*}
\end{theorem}

\begin{proof}
    The characteristic function of $I_\alpha(t,v)$ is 
    \begin{equation}\label{incre1}
        \phi_{I_\alpha(t,v)}(\xi)=\int_{0}^{\infty}\phi_{I(u,v)}(\xi)h_{\alpha}(u,t)du
    \end{equation}
    where $\phi_{I(u,v)}(\xi)$ is the characteristic function of $I(t,v)$ given by
    \begin{equation}\label{incre2}
        \phi_{I(u,v)}(\xi)=\exp\left(\sum_{j=1}^{k}\left(\Lambda_{j}(v,t+v)(e^{i\xi j}-1)+T_{j}(v,t+v)(e^{-i\xi j}-1)\right)\right), \quad i=\sqrt{-1}.
    \end{equation}
    Now taking the Laplace transform of \eqref{incre1} and using \eqref{incre2} along with the Laplace transform of inverse stable subordinator density (see Section \ref{stable subordinator}), we get
    \begin{align*}
        \mathcal{L}\left\{\phi_{I_\alpha(t,v)}(\xi)\right\}(s)&=\int_{0}^{\infty}\phi_{I(u,v)}(\xi)\tilde{h}_\alpha(u,s)du\\       &=s^{\alpha-1}\int_{0}^{\infty}\exp\left(\sum_{j=1}^{k}\left(\Lambda_{j}(v,u+v)(e^{i\xi j}-1)+T_{j}(v,u+v)(e^{-i\xi j}-1)\right)\right)e^{-us^{\alpha}}du \\
        &=s^{\alpha-1}\left\{\left[\frac{e^{-us^{\alpha}}}{s^{\alpha}}\exp\left(\sum_{j=1}^{k}\left(\Lambda_{j}(v,u+v)(e^{i\xi j}-1)+T_{j}(v,u+v)(e^{-i\xi j}-1)\right)\right)\right]_{0}^{\infty}
        \right. \\ 
        &\left.+\frac{1}{s^\alpha}\int_{0}^{\infty}\sum_{j=1}^{k}\left(\lambda_{j}(u+v)(e^{i\xi j}-1)+\gamma_{j}(u+v)(e^{-i\xi j}-1)\right)
        \right. \\ 
        & \left.\times \exp\left(\sum_{j=1}^{k}\left(\Lambda_{j}(v,u+v)(e^{i\xi j}-1)+T_{j}(v,u+v)(e^{-i\xi j}-1)\right)\right)\times e^{-us^{\alpha}}du
        \right\}.
    \end{align*}
    It follows that
    \begin{align*}
        s^\alpha \mathcal{L}\left\{\phi_{I_\alpha(t,v)}(\xi)\right\}(s)-s^{\alpha-1} &= \int_{0}^{\infty}\sum_{j=1}^{k}\left(\lambda_{j}(u+v)(e^{i\xi j}-1)+\gamma_{j}(u+v)(e^{-i\xi j}-1)\right)\\
    &\times\exp\left(\sum_{j=1}^{k}\left(\Lambda_{j}(v,u+v)(e^{i\xi j}-1)+T_{j}(v,u+v)(e^{-i\xi j}-1)\right)\right)s^{\alpha-1}e^{-us^{\alpha}}du \\
    \implies\mathcal{L}\left\{\frac{d^\alpha}{dt^\alpha}\phi_{I_\alpha(t,v)}(\xi)\right\}(s)&= \int_{0}^{\infty}\sum_{j=1}^{k}\left(\lambda_{j}(u+v)(e^{i\xi j}-1)+\gamma_{j}(u+v)(e^{-i\xi j}-1)\right)\\
    &\times\exp\left(\sum_{j=1}^{k}\left(\Lambda_{j}(v,u+v)(e^{i\xi j}-1)+T_{j}(v,u+v)(e^{-i\xi j}-1)\right)\right)\tilde{h}_\alpha(u,s)du
    \end{align*}
    where we used the Laplace transform of fractional Caputo–Djrbashian derivative (see Section \ref{caputo}). Next, on taking the inverse Laplace transform and using \eqref{incre2}, we get
    \begin{align*}
        \frac{d^\alpha}{dt^\alpha}\phi_{I_\alpha(t,v)}(\xi)        &=\int_{0}^{\infty}\sum_{j=1}^{k}\left(\lambda_{j}(u+v)(e^{i\xi j}-1)+\gamma_{j}(u+v)(e^{-i\xi j}-1)\right)\phi_{I(u,v)}(\xi)h_{\alpha}(u,t)du.
    \end{align*}
    The result follows by using the inversion formula for the characteristic functions $\phi_{I_\alpha(t,v)}$ and $\phi_{I(u,v)}$. 
    \end{proof}

\begin{remark}
On substituting $v=0$ in Theorem \ref{diffeq for incre NGFSP}, we get the system of fractional differential equations satisfied by the state probabilities $p^{\alpha}(n,t)$ of NGFSP (see Theorem \ref{diffeq for NGFSP}).
\end{remark}

\section{An Alternative Version of NGFSP}\label{sec 6}
In this section, we introduce an alternative version of NGFSP and study its properties. The main advantage of this version over the usual one is its closed-form p.m.f. making it suitable for further analysis. First we need to discuss some related processes which give rise to this version. 

\citet{Leonenko2017} studied the fractional version of a non-homogeneous Poisson process by time changing it with an independent inverse stable subordinator. They also mentioned an alternative fractional version, which we call the Non-homogeneous Time Fractional Poisson Process (NTFPP)
    \begin{equation*}
   \Tilde{N}^\alpha(t)=N_1(Y_\alpha(\Lambda(t)))\,
    \end{equation*}
    where $\{N_\lambda(t)\}_{t\ge 0}$ is the homogeneous Poisson process with rate $\lambda$ and $\Lambda(t)$ is the rate function of $\{\Tilde{N}^\alpha(t)\}_{t\ge 0}$. The p.m.f. of TFPP is given by
\begin{equation*}
    p^\alpha(n,t)=\frac{(\lambda t^\alpha)^n}{n!}\sum_{k=0}^{\infty}\frac{(k+n)!}{k!}\frac{(-\lambda t^\alpha)^k}{\Gamma\left((k+n)\alpha+1\right)},
\end{equation*}
using which the p.m.f. of NTFPP can be obtained by substituting $\lambda=1$ and $t=\Lambda(t)$ as follows:
\begin{equation}\label{tfpp p.m.f.}
    \Tilde{p}^\alpha(n,t)=\frac{\left( \Lambda(t)\right)^{n\alpha}}{n!}\sum_{k=0}^{\infty}\frac{(k+n)!}{k!}\frac{\left(-\left( \Lambda(t)\right)^\alpha\right)^k}{\Gamma\left((k+n)\alpha+1\right)}=\left(\Lambda(t)\right)^{n\alpha}E_{\alpha, n\alpha+1}^{n+1}\left(-(\Lambda(t))^\alpha\right), \quad n\geq 0.
\end{equation}
Note that the RHS of \eqref{tfpp p.m.f.} may be written as
\begin{equation}\label{eq1}
\left(\Lambda(t)\right)^{n\alpha}E_{\alpha, n\alpha+1}^{n+1}\left(-(\Lambda(t))^\alpha\right)=\left(\Lambda(t)\right)^{n\alpha}E_{\alpha, n\alpha+1}^{n}\left(-(\Lambda(t))^\alpha\right)-\left(\Lambda(t)\right)^{(n+1)\alpha}E_{\alpha, (n+1)\alpha+1}^{n+1}\left(-(\Lambda(t))^\alpha\right).
\end{equation}
Let $\tilde{T}^{\alpha}_n$ denote the $n$th arrival time of NTFPP. By definition, we have
\begin{equation}\label{eq2}
P\{\Tilde{N}^\alpha(t)=n\} = P\{\Tilde{N}^\alpha(t)\le n\} - P\{\Tilde{N}^\alpha(t) < n\} = P\{\tilde{T}^{\alpha}_{n+1}>t\} - P\{\tilde{T}^{\alpha}_{n}>t\} = \tilde{F}^{\alpha}_{n}(t) - \tilde{F}^{\alpha}_{n+1}(t)  
\end{equation}
where $\tilde{F}^{\alpha}_n$ denotes the distribution function of $\tilde{T}^{\alpha}_n$. From \eqref{eq1} and \eqref{eq2}, it follows that
\begin{equation}\label{eq3}
\tilde{F}^{\alpha}_n(t)=\left(\Lambda(t)\right)^{n\alpha}E_{\alpha, n\alpha+1}^{n}\left(-(\Lambda(t))^\alpha\right)
\end{equation}
using mathematical induction. Next, the probability generating function (p.g.f.) of TFPP is $G(u,t)=E_{\alpha,1}(\lambda t^{\alpha}(u-1))$ for $|u|\le 1$. So the p.g.f. of NTFPP is 
\begin{equation}\label{p.g.f..ntfpp}        
\tilde{G}^{\alpha}(u,t)=\mathbb{E}\left(u^{\Tilde{N}^\alpha(t)}\right)=E_{\alpha,1}((\Lambda(t))^{\alpha}(u-1)).
\end{equation}
For other properties of NTFPP, see \citet{NHSTFPP}. 
As a generalization of NTFPP, we introduce an alternative version of NGFCP as follows.
\begin{definition}
We define the Non-Homogeneous Generalized Fractional Counting process (NHGFCP) as
    \begin{equation*}   \mathcal{\widetilde{M}}^\alpha(t)=\mathbb{M}(Y_{\alpha}(\Lambda(t))
    \end{equation*}
 where $\{\mathbb{M}(t)\}_{t\ge 0}$ is a GCP with rates $\lambda_1,\ldots,\lambda_k$ and $\Lambda(t)$ is the rate function of $\{\mathcal{\widetilde{M}}^\alpha(t)\}_{t\ge 0}$.   
\end{definition}
Similar to the compound Poisson representation of GFCP (see \citet{Crescenzo2016}), the corresponding representation for NHGFCP is given by
    \begin{equation}\label{def1}
        \mathcal{\widetilde{M}}^\alpha(t)\overset{d}{=}\sum_{i=1}^{\Tilde{N}^\alpha(t)}X_i
    \end{equation}
    where $\{\Tilde{N}^\alpha(t)\}_{t\geq 0}$ is a NTFPP with rate function $\Lambda(t)=\sum_{j=1}^{k}\Lambda_j(t)$ and $\Lambda_j(t)=\int_{0}^{t}\lambda_j(u)du$. Moreover $\{X_n\}_{n\geq 1}$ is a sequence of i.i.d. random variables, independent of $\{\Tilde{N}^\alpha(t)\}_{t\geq 0}$ such that 
    \begin{equation}\label{def2}
        P\{X_n=j\}=\frac{\Lambda_j(t)}{\Lambda(t)}\,, \qquad j=1,2,...,k\,.
    \end{equation}
We are now ready to introduce an alternative version of NGFSP.
\begin{definition}
The Non-Homogeneous Generalized Fractional Skellam Process NHGFSP is defined as
\begin{equation*}
\widetilde{\mathcalboondox{S}}^{\alpha}(t) = \widetilde{\mathcal{M}}^{\alpha}_1(t) - \widetilde{\mathcal{M}}^{\alpha}_2(t) 
\end{equation*}
where $\{\widetilde{\mathcal{M}}^{\alpha}_1(t)\}_{t\ge 0}$ and $\{\widetilde{\mathcal{M}}^{\alpha}_2(t)\}_{t\ge 0}$ are two independent NHGFCPs with rate functions $\Lambda(t)=\sum_{j=1}^{k}\Lambda_j(t)$ and $T(t)=\sum_{j=1}^{k}T_j(t)$ respectively.  
\end{definition} 
\begin{proposition}
The moment generating function (m.g.f.) of NHGFSP is given by
\begin{equation*}
M_{\widetilde{\mathcalboondox{S}}^{\alpha}(t)}(s) = E_{\alpha, 1}\left(\sum_{j=1}^{k}\Lambda_j(t)(e^{sj}-1)(\Lambda(t))^{\alpha-1}\right)\times E_{\alpha, 1}\left(\sum_{j=1}^{k}T_j(t)(e^{sj}-1)(T(t))^{\alpha-1}\right). 
\end{equation*}
\end{proposition}
\begin{proof}
The m.g.f. of $Y(t)=\mathcal{\widetilde{M}}^\alpha_1(t)$ is given by
    \begin{align}\label{m.g.f..NTFPP}
        \mathbb{E}\left(e^{sY(t)}\right)= \mathbb{E}\left(\mathbb{E}\left(e^{sY(t)}|\Tilde{N}^\alpha(t)\right)\right)
        &= \sum_{x=0}^{\infty}\mathbb{E}\left(e^{s\sum_{i=1}^{x}X_i}\right)P\{\Tilde{N}^\alpha(t)=x\} \nonumber\\
         &=\mathbb{E}\left(\left(\mathbb{E}\left(e^{sX_1}\right)\right)^{\Tilde{N}^\alpha(t)}\right) \nonumber\\
         &=\mathbb{E}\left(\left(\frac{1}{\Lambda(t)}\sum_{j=1}^{k}\Lambda_j(t)e^{sj}\right)^{\Tilde{N}^\alpha(t)}\right) \nonumber\\
         &= E_{\alpha, 1}\left(\sum_{j=1}^{k}\Lambda_j(t)(e^{sj}-1)(\Lambda(t))^{\alpha-1}\right) \quad(\text{using \eqref{p.g.f..ntfpp}}). \end{align}
Similarly, the m.g.f.  of $\mathcal{\widetilde{M}}_2^\alpha(t)$ is     $E_{\alpha, 1}\left(\sum_{j=1}^{k}T_j(t)(e^{sj}-1)(T(t))^{\alpha-1}\right)$. Since $\{\widetilde{\mathcal{M}}^{\alpha}_{1}(t)\}_{t\ge 0}$ and $\{\widetilde{\mathcal{M}}^{\alpha}_{2}(t)\}_{t\ge 0}$ are independent, the m.g.f. of NHGFSP is
\begin{equation*}
M_{\widetilde{\mathcalboondox{S}}^{\alpha}(t)}(s) = E_{\alpha, 1}\left(\sum_{j=1}^{k}\Lambda_j(t)(e^{sj}-1)(\Lambda(t))^{\alpha-1}\right)\times E_{\alpha, 1}\left(\sum_{j=1}^{k}T_j(t)(e^{sj}-1)(T(t))^{\alpha-1}\right). 
\end{equation*}
\end{proof}
\begin{theorem}\label{p.m.f..NGHFSP}
    The p.m.f. of NHGFSP for $n\in \mathbb{Z}$, $\alpha\in (0,1]$ and $t\geq 0$ is given by        
    \begin{align*}
    P\{\widetilde{\mathcalboondox{S}}^{\alpha}(t)=n\}&=\sum_{m=0}^{\infty}[P\{\widetilde{\mathcal{M}}^{\alpha}_{1}(t)=m+n\}P\{\widetilde{\mathcal{M}}^{\alpha}_{2}(t)=m\}\mathbb{I}_{\{n\geq0\}}+P\{\widetilde{\mathcal{M}}^{\alpha}_{2}(t)=m+|n|\}P\{\widetilde{\mathcal{M}}^{\alpha}_{1}(t)=m\}\mathbb{I}_{\{n<0\}}]
    \end{align*}
where for $j\in\mathbb{N}$
\begin{align*}
        P\{\mathcal{\widetilde{M}}^\alpha_1(t)=j\}&=\sum_{r=0}^{j}\sum_{\substack{\alpha_1+\alpha_2+...+\alpha_k=r \\ \alpha_1+2\alpha_2+...+k\alpha_k=j}}\binom{r}{\alpha_1, \alpha_2,\ldots,\alpha_k} \prod_{s=1}^{k}\left(\Lambda_s(t)\right)^{\alpha_s}\times\frac{1}{\left(\Lambda(t)\right)^{r(1-\alpha)}}E_{\alpha, r\alpha+1}^{r+1}\left(-(\Lambda(t))^\alpha\right),
    \end{align*}
\begin{align*}
        P\{\mathcal{\widetilde{M}}^\alpha_2(t)=j\}&=\sum_{r=0}^{j}\sum_{\substack{\alpha_1+\alpha_2+...+\alpha_k=r \\ \alpha_1+2\alpha_2+...+k\alpha_k=j}}\binom{r}{\alpha_1, \alpha_2,\ldots,\alpha_k} \prod_{s=1}^{k}\left(T_s(t)\right)^{\alpha_s}\times\frac{1}{\left(T(t)\right)^{r(1-\alpha)}}E_{\alpha, r\alpha+1}^{r+1}\left(-(T(t))^\alpha\right).
    \end{align*}
\end{theorem}
\begin{proof} From \eqref{def1}, we have
   \begin{align}
       P\{\mathcal{\widetilde{M}}^\alpha_1(t)=j\}=\sum_{r=0}^{j}P\{X_1+X_2+...+X_k=j\}P\{\Tilde{N}^\alpha(t)=r\}.\label{p.m.f.}
   \end{align}
   Since $X_1, X_2, ...,X_k$ are identically distributed, it follows from \eqref{def2} that
   \begin{align}
       P\{X_1+X_2+...+X_k=j\}=\sum_{\substack{\alpha_1+\alpha_2+...+\alpha_k=r \\ \alpha_1+2\alpha_2+...+k\alpha_k=j}}\binom{r}{\alpha_1, \alpha_2, ...\alpha_k}\prod_{s=1}^{k}\left(\frac{\Lambda_s(t)}{\Lambda(t)}\right)^{\alpha_s}.\label{p.m.f.1}
   \end{align}
   Therefore from \eqref{tfpp p.m.f.}, \eqref{p.m.f.} and \eqref{p.m.f.1}, we get
   \begin{align*}
        P\{\mathcal{\widetilde{M}}^\alpha_1(t)=j\}&=\sum_{r=0}^{j}\sum_{\substack{\alpha_1+\alpha_2+...+\alpha_k=r \\ \alpha_1+2\alpha_2+...+k\alpha_k=j}}\binom{r}{\alpha_1, \alpha_2,\ldots,\alpha_k} \prod_{s=1}^{k}\left(\Lambda_s(t)\right)^{\alpha_s}\times\frac{1}{\left(\Lambda(t)\right)^{r(1-\alpha)}}E_{\alpha, r\alpha+1}^{r+1}\left(-(\Lambda(t))^\alpha\right).
    \end{align*}
Similarly we have 
\begin{align*}
        P\{\mathcal{\widetilde{M}}^\alpha_2(t)=j\}&=\sum_{r=0}^{j}\sum_{\substack{\alpha_1+\alpha_2+...+\alpha_k=r \\ \alpha_1+2\alpha_2+...+k\alpha_k=j}}\binom{r}{\alpha_1, \alpha_2,\ldots,\alpha_k} \prod_{s=1}^{k}\left(T_s(t)\right)^{\alpha_s}\times\frac{1}{\left(T(t)\right)^{r(1-\alpha)}}E_{\alpha, r\alpha+1}^{r+1}\left(-(T(t))^\alpha\right).
    \end{align*}
Since $\{\widetilde{\mathcal{M}}^{\alpha}_{1}(t)\}_{t\ge 0}$ and $\{\widetilde{\mathcal{M}}^{\alpha}_{2}(t)\}_{t\ge 0}$ are independent, the p.m.f. of NHGFSP is
\begin{align*}
P\{\widetilde{\mathcalboondox{S}}^{\alpha}(t)=n\}
    &=P\{\widetilde{\mathcal{M}}^{\alpha}_{1}(t)-\widetilde{\mathcal{M}}^{\alpha}_{2}(t)=n\}\mathbb{I}_{\{n\geq0\}}+P\{\widetilde{\mathcal{M}}^{\alpha}_{1}(t)-\widetilde{\mathcal{M}}^{\alpha}_{2}(t)=n\}\mathbb{I}_{\{n<0\}}\nonumber\\
     &= \sum_{m=0}^{\infty}[P\{\widetilde{\mathcal{M}}^{\alpha}_{1}(t)=m+n\}P\{\widetilde{\mathcal{M}}^{\alpha}_{2}(t)=m\}\mathbb{I}_{\{n\geq0\}}+P\{\widetilde{\mathcal{M}}^{\alpha}_{2}(t)=m+|n|\}P\{\widetilde{\mathcal{M}}^{\alpha}_{1}(t)=m\}\mathbb{I}_{\{n<0\}}].
\end{align*}  
\end{proof}
  
\subsection{Mean, Variance and Covariance}
 Using the p.g.f. in \eqref{p.g.f..ntfpp}, the first and second moments of NTFPP are given by
\begin{equation*}
\mathbb{E}\left(\tilde{N}^\alpha(t)\right) = \frac{\left(\Lambda(t)\right)^{\alpha}}{\Gamma(\alpha+1)},\quad \mathbb{V}\left(\tilde{N}^\alpha(t)\right) = \frac{\left(\Lambda(t)\right)^{\alpha}}{\Gamma(\alpha+1)}\left[1+\frac{\left(\Lambda(t)\right)^{\alpha}}{\Gamma(\alpha+1)}\left\{\frac{\alpha B(\alpha,1/2)}{2^{2\alpha}-1}-1\right\}\right]  
\end{equation*} 
where $B(a,b)=\Gamma(a)\Gamma(b)/\Gamma(a+b)$. Hence using \eqref{def1}, the mean and variance of NHGFCP are
\begin{align*}
\mathbb{E}\left(\mathcal{\widetilde{M}}^\alpha(t)\right)&=\mathbb{E}\left(\tilde{N}^\alpha(t)\right)\mathbb{E}\left(X_1\right)=\left(\frac{\left(\Lambda(t)\right)^{\alpha}}{\Gamma(\alpha+1)}\right)\left(\frac{1}{\Lambda(t)}\sum_{j=1}^{k}j\Lambda_j(t)\right), \\
\mathbb{V}\left(\mathcal{\widetilde{M}}^\alpha(t)\right)&=\mathbb{E}\left(\tilde{N}^\alpha(t)\right)\mathbb{V}(X_1) + [\mathbb{E}\left(X_1\right)]^2\mathbb{V}\left(\tilde{N}^\alpha(t)\right) \\
&= \frac{\left(\Lambda(t)\right)^{\alpha}}{\Gamma(\alpha+1)}\left[\left(\frac{1}{\Lambda(t)}\sum_{j=1}^{k}j^2\Lambda_j(t)\right)+\left(\frac{1}{\Lambda(t)}\sum_{j=1}^{k}j\Lambda_j(t)\right)^2\frac{\left(\Lambda(t)\right)^{\alpha}}{\Gamma(\alpha+1)}\left\{\frac{\alpha B(\alpha,1/2)}{2^{2\alpha-1}}-1\right\}\right].
\end{align*}
Note that
\begin{align*}
\mathbb{V}\left(\mathcal{\widetilde{M}}^\alpha(t)\right)-\mathbb{E}\left(\mathcal{\widetilde{M}}^\alpha(t)\right) &= \frac{\left(\Lambda(t)\right)^{2\alpha}}{\Gamma^2(\alpha+1)}\left(\frac{1}{\Lambda(t)}\sum_{j=1}^{k}j\Lambda_j(t)\right)^2\left\{\frac{\alpha B(\alpha,1/2)}{2^{2\alpha-1}}-1\right\} > 0 
\end{align*}
since $\alpha B(\alpha,1/2)>2^{2\alpha-1}$. Thus NHGFCP is over-dispersed. Next for $s<t$,
\begin{align*}
\text{Cov}[\mathcal{\widetilde{M}}^\alpha(s),\mathcal{\widetilde{M}}^\alpha(t)]
    &=\mathbb{E}[\mathcal{\widetilde{M}}^\alpha(s)(\mathcal{\widetilde{M}}^\alpha(t)-\mathcal{\widetilde{M}}^\alpha(s))+(\mathcal{\widetilde{M}}^\alpha(s))^2]-\mathbb{E}[\mathcal{\widetilde{M}}^\alpha(s)]\mathbb{E}[\mathcal{\widetilde{M}}^\alpha(t)]\\
    &=\mathbb{E}[\mathcal{\widetilde{M}}^\alpha(s)]\mathbb{E}(\mathcal{\widetilde{M}}^\alpha(t)-\mathcal{\widetilde{M}}^\alpha(s))+\mathbb{E}(\mathcal{\widetilde{M}}^\alpha(s))^2-\mathbb{E}[\mathcal{\widetilde{M}}^\alpha(s)]\mathbb{E}[\mathcal{\widetilde{M}}^\alpha(t)]
    =\mathbb{V}(\mathcal{\widetilde{M}}^\alpha(s)).
\end{align*}
It follows that for arbitrary $s\ge 0,~t\ge 0$, we have $\text{Cov}[\mathcal{\widetilde{M}}^\alpha(s),\mathcal{\widetilde{M}}^\alpha(t)]=\mathbb{V}(\mathcal{\widetilde{M}}^\alpha(s\wedge t))$. Using the above expressions, we can obtain the mean, variance and covariance of NHGFSP as follows:
\begin{align*}
\mathbb{E}\left(\widetilde{\mathcalboondox{S}}^{\alpha}(t)\right) &= \left(\frac{\left(\Lambda(t)\right)^{\alpha}}{\Gamma(\alpha+1)}\right)\left(\frac{1}{\Lambda(t)}\sum_{j=1}^{k}j\Lambda_j(t)\right) + \left(\frac{\left(T(t)\right)^{\alpha}}{\Gamma(\alpha+1)}\right)\left(\frac{1}{\Lambda(t)}\sum_{j=1}^{k}jT_j(t)\right), \nonumber \\
\mathbb{V}\left(\widetilde{\mathcalboondox{S}}^{\alpha}(t)\right) &=\frac{\left(\Lambda(t)\right)^{\alpha}}{\Gamma(\alpha+1)}\left[\left(\frac{1}{\Lambda(t)}\sum_{j=1}^{k}j^2\Lambda_j(t)\right)+\left(\frac{1}{\Lambda(t)}\sum_{j=1}^{k}j\Lambda_j(t)\right)^2\frac{\left(\Lambda(t)\right)^{\alpha}}{\Gamma(\alpha+1)}\left\{\frac{\alpha B(\alpha,1/2)}{2^{2\alpha-1}}-1\right\}\right] \nonumber \\
& + \frac{\left(T(t)\right)^{\alpha}}{\Gamma(\alpha+1)}\left[\left(\frac{1}{T(t)}\sum_{j=1}^{k}j^2 T_j(t)\right)+\left(\frac{1}{\Lambda(t)}\sum_{j=1}^{k}j T_j(t)\right)^2\frac{\left(T(t)\right)^{\alpha}}{\Gamma(\alpha+1)}\left\{\frac{\alpha B(\alpha,1/2)}{2^{2\alpha-1}}-1\right\}\right], \nonumber \\
\text{Cov}[\widetilde{\mathcalboondox{S}}^\alpha(s),\widetilde{\mathcalboondox{S}}^\alpha(t)] &= \mathbb{V}(\mathcal{\widetilde{M}}^\alpha_1(s\wedge t)) + \mathbb{V}(\mathcal{\widetilde{M}}^\alpha_2(s\wedge t))=\mathbb{V}\left(\widetilde{\mathcalboondox{S}}^{\alpha}(s\wedge t)\right).
\end{align*}
Moreover since $\mathbb{V}\left(\widetilde{\mathcalboondox{S}}^{\alpha}(t)\right)>\mathbb{E}\left(\widetilde{\mathcalboondox{S}}^{\alpha}(t)\right)$, NHGFSP is over-dispersed.

\subsection{Correlation structure}
For fixed $s\ge 0,~t\ge 0$, the correlation function of NHGFSP where $s<t$ is 
\begin{align*}
\text{Corr}(\widetilde{\mathcalboondox{S}}^\alpha(s),\widetilde{\mathcalboondox{S}}^\alpha(t))&= \frac{\text{Cov}(\widetilde{\mathcalboondox{S}}^\alpha(s),\widetilde{\mathcalboondox{S}}^\alpha(t))}{\sqrt{\mathbb{V}(\widetilde{\mathcalboondox{S}}^\alpha(s))}\sqrt{\mathbb{V}(\widetilde{\mathcalboondox{S}}^\alpha(t))}} = \sqrt{\frac{\mathbb{V}(\widetilde{\mathcalboondox{S}}^\alpha(s))}{\mathbb{V}(\widetilde{\mathcalboondox{S}}^\alpha(t))}}.
\end{align*}
Note that
\begin{equation*}
\mathbb{V}\left(\widetilde{\mathcalboondox{S}}^\alpha(t)\right)>\frac{\left(\Lambda(t)\right)^{2\alpha}+\left(T(t)\right)^{2\alpha}}{[\Gamma(\alpha+1)]^2}
\end{equation*}
Thus we have
\begin{align*}
\text{Corr}(\widetilde{\mathcalboondox{S}}^\alpha(s),\widetilde{\mathcalboondox{S}}^\alpha(t)) &< c(s)(\left(\Lambda(t)\right)^{2\alpha}+\left(T(t)\right)^{2\alpha})^{-1/2} \quad\text{where}~c(s)=\sqrt{\mathbb{V}\left(\widetilde{\mathcalboondox{S}}^\alpha(s)\right)}>0.
\end{align*}
Suppose $\Lambda(t)=(t/b)^a$ and $T(t)=(t/d)^c$ (Weibull's rate functions) where $a,b,c,d>0$ are some constants. Then NHGFSP has the LRD property if $a\wedge c<1/\alpha$ and the SRD property if $1/\alpha<a\wedge c<2/\alpha$. 

\subsection{Waiting time}
We want to find the distribution of the waiting time for NHGFCP until the first occurrence of a jump of size $j$ for $1\le j\le k$. First note that the NHGFCP can be written as 
\begin{equation*}
\mathcal{\widetilde{M}}^\alpha(t)) = \sum_{j=1}^k j\mathcal{\widetilde{M}}^\alpha_j(t)
\end{equation*}
where $\mathcal{\widetilde{M}}^\alpha_j(t)=\sum_{j=1}^{\tilde{N}^\alpha(t)}\bm{1}_{\{X_1=j\}}$ is the number of jumps of size $j$ performed by the NHGFCP in the interval $(0,t]$. Next consider the random variables
\begin{equation*}
H_j = \inf\{s>0:\mathcal{\widetilde{M}}^\alpha_j(s)=1\}, \quad G_j\sim\text{Geo}\left(\frac{\Lambda_j(t)}{\Lambda(t)}\right).
\end{equation*}
Here $H_j$ is the first occurrence time of a jump of size $j$ for the NHGFCP and $G_j$ is a geometric random variable with parameter $\Lambda_j(t)/\Lambda(t)$ denoting the order of the first jump of size $j$ among all jumps. 
\begin{theorem}
The distribution function of $H_j$ is
\begin{equation*}
P(H_j\le t) = \Lambda_j(t)(\Lambda(t))^{\alpha-1}E_{\alpha,\alpha+1}\left(-\Lambda_j(t)(\Lambda(t))^{\alpha-1}\right).    
\end{equation*}
\end{theorem}
\begin{proof}
We have
\begin{align*}
P(H_j\le t) &= \sum_{n=1}^{\infty}P(H_j\le t\mid G_j=n)P(G_j=n) \\
&=\sum_{n=1}^{\infty}\tilde{F}^{\alpha}_n(t)\left(\frac{\Lambda_j(t)}{\Lambda(t}\right)\left(1-\frac{\Lambda_j(t)}{\Lambda(t}\right)^{n-1} \\
&=\sum_{n=1}^{\infty}\left(\Lambda(t)\right)^{n\alpha}E_{\alpha, n\alpha+1}^{n}\left(-(\Lambda(t))^\alpha\right)\left(\frac{\Lambda_j(t)}{\Lambda(t}\right)\left(1-\frac{\Lambda_j(t)}{\Lambda(t}\right)^{n-1}\quad(\text{from}~\eqref{eq3}) \\
&=\Lambda_j(t)(\Lambda(t))^{\alpha-1}\sum_{n=0}^{\infty}\left(\Lambda(t)\right)^{n\alpha}\left(1-\frac{\Lambda_j(t)}{\Lambda(t)}\right)^{n}E_{\alpha, (n+1)\alpha+1}^{n+1}\left(-(\Lambda(t))^\alpha\right).
\end{align*}
Using formula (2.3.1) of \citet{Mathai2008}
\begin{equation*}
E^{\delta}_{\beta,\gamma+\alpha}(z) = \frac{1}{\Gamma(\alpha)}\int_{0}^{1}u^{\gamma -1}(1-u)^{\alpha-1}E^{\delta}_{\beta,\gamma}(zu^{\beta})
\end{equation*}
for $\beta=\alpha$, $\gamma=n\alpha+1$, $\delta=n+1$ and $z=-(\Lambda(t))^\alpha$, we have
\begin{align*}
P(H_j\le t) &= \frac{\Lambda_j(t)(\Lambda(t))^{\alpha-1}}{\Gamma(\alpha)}\sum_{n=0}^{\infty}\left(\Lambda(t)\right)^{n\alpha}\left(1-\frac{\Lambda_j(t)}{\Lambda(t)}\right)^{n}\int_{0}^1 u^{n\alpha}(1-u)^{\alpha-1}E_{\alpha,n\alpha+1}^{n+1}\left(-(\Lambda(t))^\alpha u^{\alpha}\right) \\
&= \frac{\Lambda_j(t)(\Lambda(t))^{\alpha-1}}{\Gamma(\alpha)}\int_{0}^1 (1-u)^{\alpha-1}\sum_{n=0}^{\infty}\left[\left(\Lambda(t)\right)^{\alpha}\left(1-\frac{\Lambda_j(t)}{\Lambda(t)}\right)u^{\alpha}\right]^n E_{\alpha,n\alpha+1}^{n+1}\left(-(\Lambda(t))^\alpha u^{\alpha}\right)du. 
\end{align*}
Using formula (2.30) of \citet{Beghin2010}
\begin{equation*}
\sum_{n=0}^{\infty}(\lambda wt^{\alpha})^n E_{\alpha,n\alpha+1}^{n+1}(-\lambda t^{\alpha}) = E_{\alpha,1}(\lambda(w-1)t^{\alpha}, \quad |w|\le 1, t>0,
\end{equation*}
we have
\begin{equation*}
P(H_j\le t) = \frac{\Lambda_j(t)(\Lambda(t))^{\alpha-1}}{\Gamma(\alpha)}\int_{0}^1 (1-u)^{\alpha-1}E_{\alpha,1}\left(-\Lambda_j(t)(\Lambda(t))^{\alpha-1}u^{\alpha}\right)du.
\end{equation*}
Using formula (2.2.14) of \citet{Mathai2008}
\begin{equation*}
\int_{0}^1 z^{\beta-1}(1-z)^{\sigma -1}E_{\alpha,\beta}(xz^{\alpha})dz = \Gamma(\sigma)E_{\alpha,\sigma+\beta}(x)
\end{equation*}
for $\sigma=\alpha$, $\beta=1$, $x = -\Lambda_j(t)(\Lambda(t))^{\alpha-1}$ and $z=u$, we have
\begin{equation*}
P(H_j\le t) = \Lambda_j(t)(\Lambda(t))^{\alpha-1}E_{\alpha,\alpha+1}\left(-\Lambda_j(t)(\Lambda(t))^{\alpha-1}\right).
\end{equation*}
\end{proof}
Note that it's difficult to obtain the distribution of the first occurrence time of a jump of size $j$ for the NHGFSP. However for $n\in\mathbb{Z}$ and $t>0$, the distribution functions for the $n$th arrival time $\widetilde{\tau}^{\alpha}_n$ and the first passage time $\widetilde{T}^{\alpha}_n$ for state $n$ of a NHGFSP are given by
\begin{align*}
F_{\widetilde{\tau}_n^\alpha}(t)=P\left(\widetilde{\tau}^{\alpha}_n\le t\right)=\sum_{i\in(0,t]}P\left\{\widetilde{\mathcalboondox{S}}^\alpha(i)= n\right\},\quad F_{\widetilde{T}_n^\alpha}(t)=P\{\widetilde{T}_n^{\alpha}\le t\}= \sum_{i\in (0,t]}\sum_{m=n}^{\infty}P\left\{\widetilde{\mathcalboondox{S}}^\alpha(i)= m\right\},
\end{align*}
where the p.m.f. of NHGFSP is given in Theorem \ref{p.m.f..NGHFSP}.


\subsection{Convergence results}
Here we study the asymptotic behaviour of NTFPP and NHGFSP as their parameters become large. 
\begin{theorem}\label{convergence.NTFPP}
Let $\alpha\in(0,1]$. For a fixed $t>0$, as $\Lambda(t)\rightarrow\infty$, we have
\begin{equation*}
\frac{\tilde{N}^\alpha(t)}{\mathbb{E}(\tilde{N}^\alpha(t))}\xrightarrow{p} 1.
\end{equation*}
\end{theorem}
\begin{proof}
Using the triangular inequality, we have
\begin{equation*}
\mathbb{E}\left[\left|\frac{\tilde{N}^\alpha(t)}{\mathbb{E}(\tilde{N}^\alpha(t))}-1\right|\right]\le 2.
\end{equation*}
Hence applying the dominated convergence theorem, we obtain
\begin{equation}\label{eq4}
\lim_{\Lambda(t)\rightarrow\infty}\mathbb{E}\left[\left|\frac{\tilde{N}^\alpha(t)}{\mathbb{E}(\tilde{N}^\alpha(t))}-1\right|\right] = \lim_{\Lambda(t)\rightarrow\infty}\sum_{n=0}^{\infty}\left|\frac{n}{\frac{\left(\Lambda(t)\right)^{\alpha}}{\Gamma(\alpha+1)}}-1\right|\left(\Lambda(t)\right)^{n\alpha}E_{\alpha, n\alpha+1}^{n+1}\left(-(\Lambda(t))^\alpha\right).
\end{equation}
Note that the generalized Mittag-Leffler function for large $x$ behaves as $E^{\delta}_{\alpha,\beta}(x)\sim\mathcal{O}(|x|^{-\delta})$ for $|x|>1$ (see
\citet{Saxena2004}) so the limit in \eqref{eq4} is $0$. Thus the random variable $\frac{\tilde{N}^\alpha(t)}{\mathbb{E}(\tilde{N}^\alpha(t))}$ converges in mean to 1, which implies $\frac{\tilde{N}^\alpha(t)}{\mathbb{E}(\tilde{N}^\alpha(t))}\xrightarrow{p} 1$.   
\end{proof}
The $r^{th}$ order moment of a TFPP is given by $\mathbb{E}[N^\alpha(t)]^r=\sum_{q=0}^{r}S_{\alpha}(r,q)(\lambda t^{\alpha})^q$ where $S_{\alpha}(r,q)$ is the fractional Stirling number (see Eqs. (32) and (40) of \citet{Laskin2010}). So the $r^{th}$ order moment of a NTFPP is given by $\mathbb{E}[\tilde{N}^\alpha(t)]^r=\sum_{q=0}^{r}S_{\alpha}(r,q)(\Lambda(t))^{\alpha q})$. Using this expression, the following result can be obtained similarly as in Theorem \ref{convergence.NTFPP}. 
\begin{proposition}\label{moment.NTFPP}
Let $\alpha\in(0,1]$ and $r\in\mathbb{N}$. For a fixed $t>0$, as $\Lambda(t)\rightarrow\infty$, we have
\begin{equation*}
\frac{[\tilde{N}^\alpha(t)]^r}{\mathbb{E}[\tilde{N}^\alpha(t)]^r}\xrightarrow{p} 1.
\end{equation*}
\end{proposition}
Next we discuss a convergence result for NHGFSP.
\begin{theorem}\label{convergence.NHGFSP}
Let $\alpha\in(0,1]$ and $r\in\mathbb{N}$. For $1\le j\le k$ and fixed $t>0$, as $\Lambda_j(t), T_j(t)\rightarrow\infty$, we have
\begin{equation*}
\frac{[\widetilde{\mathcalboondox{S}}^\alpha(t)]^r}{[\mathbb{E}(\widetilde{\mathcalboondox{S}}^\alpha(t))]^r}\xrightarrow{p} 1.    
\end{equation*}
\end{theorem}

\begin{proof} 
Using Hoppe's formula and \eqref{m.g.f..NTFPP}, the $r^{th}$ order moment of the process $\{\widetilde{\mathcal{M}}^{\alpha}_1(t)\}_{t\ge 0}$ is
\begin{align*}
\mathbb{E}[\widetilde{\mathcal{M}}^{\alpha}_1(t)]^r &= \frac{d^r}{ds^r}E_{\alpha, 1}\left(\sum_{j=1}^{k}(\Lambda_j(t)(e^{sj}-1)(\Lambda(t))^{\alpha-1}\right)\Bigg |_{s=0} \\
&= \sum_{n=0}^r\frac{(E_{\alpha,1}(x))^{(n)}\big|_{x = \sum_{j=1}^{k}\Lambda_j(t)(e^{sj}-1)(\Lambda(t))^{\alpha-1}}}{n!}A_{r,n}\left(\sum_{j=1}^{k}(\Lambda_j(t)(e^{sj}-1)(\Lambda(t))^{\alpha-1}\right)\Bigg |_{s=0},  
\end{align*}
where $f^{(n)}$denotes the $n$th derivative of $f$ and
\begin{align*}
A_{r,n}\left(\sum_{j=1}^{k}\Lambda_j(t)(e^{sj}-1)(\Lambda(t))^{\alpha-1}\right) &= \sum_{h=0}^n\binom{n}{h}\left(-\sum_{j=1}^{k}\Lambda_j(t)(e^{sj}-1)(\Lambda(t))^{\alpha-1}\right)^{n-h} \nonumber \\
&\times\frac{d^r}{ds^r}\left(\sum_{j=1}^{k}\Lambda_j(t)(e^{sj}-1)(\Lambda(t))^{\alpha-1}\right)^{h}. 
\end{align*}
After some algebra, we obtain 
\begin{equation}\label{moment.NHGFCP}
\mathbb{E}[\widetilde{\mathcal{M}}^{\alpha}(t)]^r =\sum_{n=0}^{r}\frac{\left(\Lambda(t)\right)^{n\alpha}}{\Gamma(n\alpha+1)}\sum_{i_1+\cdots+i_k=n}\binom{n}{i_1,\ldots,i_k}\prod_{l=1}^k(\Lambda_l(t))^{i_l}\times\sum_{j_1+\cdots+j_k=r}\binom{r}{j_1,\ldots,j_k}\left[\prod_{l=1}^k\frac{d^{j_l}}{ds^{j_l}}(e^{ls}-1)^{i_l}\right]\Bigg |_{s=0}.
\end{equation}
Using \eqref{moment.NHGFCP}, it can be shown that $\frac{[\widetilde{\mathcal{M}}^\alpha_1(t)]^r}{\mathbb{E}([\widetilde{\mathcal{M}}^\alpha_1(t)]^r)}$ converges in mean to $1$ as $\Lambda_j(t)\rightarrow\infty$ similar to Theorem \ref{convergence.NTFPP}. This implies $\frac{[\widetilde{\mathcal{M}}^\alpha_1(t)]^r}{\mathbb{E}([\widetilde{\mathcal{M}}^\alpha_1(t)]^r)}\xrightarrow{p} 1$. Similarly it can be shown that $\frac{[\widetilde{\mathcal{M}}^\alpha_2(t)]^r}{\mathbb{E}([\widetilde{\mathcal{M}}^\alpha_2(t)]^r)}\xrightarrow{p} 1$ as $T_j(t)\rightarrow\infty$. In other words, $[\widetilde{\mathcal{M}}^\alpha_1(t)]^r\xrightarrow{p}\mathbb{E}([\widetilde{\mathcal{M}}^\alpha_1(t)]^r)$ and $[\widetilde{\mathcal{M}}^\alpha_1(t)]^r\xrightarrow{p}\mathbb{E}([\widetilde{\mathcal{M}}^\alpha_1(t)]^r)$. Hence for $r\in\mathbb{N}$ and fixed $t>0$, we have
\begin{align*}
[\widetilde{\mathcalboondox{S}}^\alpha(t)]^r = \sum_{j=0}^r\binom{r}{j}[\widetilde{\mathcal{M}}^\alpha_1(t)]^{r-j}[\widetilde{\mathcal{M}}^\alpha_2(t)]^{j}&\xrightarrow{p}\sum_{j=0}^r\binom{r}{j}\mathbb{E}[\widetilde{\mathcal{M}}^\alpha_1(t)]^{r-j}\mathbb{E}[\widetilde{\mathcal{M}}^\alpha_2(t)]^{j} = [\mathbb{E}(\widetilde{\mathcalboondox{S}}^\alpha(t))]^r  
\end{align*}
as $\Lambda_j(t),T_j(t)\rightarrow\infty$ for $1\le j\le k$, thereby proving the result.
\end{proof}

Proposition \ref{moment.NTFPP} and Theorem \ref{convergence.NHGFSP} show that the processes $\{\tilde{N}^\alpha(t)\}^r_{t\ge 0}$ and $\{\widetilde{\mathcalboondox{S}}^\alpha(t)\}^r_{t\ge 0}$ where $r\in\mathbb{N}$ exhibit cut-off behaviour at mean times (see \citet{Barrera2009}) with respect to the relevant parameters, that is, they converge abruptly to equlibrium as the parameters grow large. 

\section{Running average processes of GSP and GCP}\label{sec 7}
In this section, we introduce the running average processes of GSP and GCP and study their properties.
\subsection{Running Average of GSP}
Let $\{\mathcal{S}(t)\}_{t\ge 0}$ be a GSP such that $\mathcal{S}(t)=\mathbb{M}_1(t)-\mathbb{M}_2(t)$ where $\{\mathbb{M}_1(t)\}_{t\ge 0}$ and $\{\mathbb{M}_2(t)\}_{t\ge 0}$ are independent GCPs with intensity parameters $\lambda_1,\lambda_2,\ldots,\lambda_k$ and $\mu_1,\mu_2,\ldots,\mu_k$ respectively.
\begin{definition}
We define the running average process of a GSP by taking the time-scaled integral of its path (see \citet{Xia2018, Gupta2020}) as follows:
\begin{equation*}
\mathcal{S}_A(t)=\frac{1}{t}\int_{0}^{t}\mathcal{S}(s)ds. \label{78}
\end{equation*}
\end{definition}
Note that $\{\mathcal{S}_A(t)\}_{t \geq 0}$ satisfies the following differential equation with initial condition $\mathcal{S}_A(0)=0$:
\begin{equation*}
    \frac{d}{dt}\left(\mathcal{S}_A(t)\right)=\frac{1}{t}\mathcal{S}_A(t)-\frac{1}{t^2}\int_{0}^{t}\mathcal{S}_A(s)ds    
\end{equation*}
which shows that it has continuous sample paths of bounded total variation. 
\begin{proposition}\label{ch.gsp}
The characteristic function of the running average process of a GSP is given by
\begin{align*}
    \phi_{\mathcal{S}_A(t)}(u)=e^{t\left\{\sum_{j=1}^{k}\lambda_j\left(\frac{e^{iuj}-1}{iuj}-1\right)+\sum_{j=1}^{k}\mu_j\left(\frac{1-e^{-iuj}}{iuj}-1\right)\right\}}.
\end{align*}
\end{proposition}
\begin{proof}
Let us denote $K(t)=\int_{0}^{t}\mathcal{S}(s)ds$. Since $\{\mathcal{S}(t)\}_{t\ge 0}$ is a L\'{e}vy process, using Lemma \ref{running average lemma}, the characteristic function of $K(t)$ is given as
\begin{equation*}
    \phi_{K(t)}(u)=e^{t\left(\int_{0}^{1}\log \phi_{\mathcal{S}(1)}(tuz)dz\right)}.
\end{equation*}
Using \eqref{2}, the characteristic function of $\{\mathcal{S}(t)\}_{t\ge 0}$ is
\begin{equation*}
\phi_{\mathcal{S}(t)}(u) = \exp\left(\sum_{j=1}^{k}t\left(\lambda_{j}(e^{iuj}-1)+\mu_{j}(e^{-iuj}-1)\right)\right)
\end{equation*}
which implies
\begin{equation*}
\int_{0}^{1}\log \phi_{\mathcal{S}(1)}(tuz)dz=\sum_{j=1}^{k}\lambda_j\left(\frac{e^{ituj}-1}{ituj}-1\right)+\sum_{j=1}^{k}\mu_j\left(\frac{1-e^{-ituj}}{ituj}-1\right). 
\end{equation*}
Hence the characteristic function of $\{\mathcal{S}_A(t)\}_{t\ge 0}$ is 
\begin{equation*}
\phi_{\mathcal{S}_{A}(t)}(u)=\phi_{K(t)/t}(u)=\phi_{K(t)}(u/t)=e^{t\left(\int_{0}^{1}\log \phi_{\mathcal{S}(1)}(uz)dz\right)}=e^{t\left\{\sum_{j=1}^{k}\lambda_j\left(\frac{e^{iuj}-1}{iuj}-1\right)+\sum_{j=1}^{k}\mu_j\left(\frac{1-e^{-iuj}}{iuj}-1\right)\right\}}.
\end{equation*}
\end{proof} 
Note that $\phi_{\mathcal{S}_{A}(t)}(u)$ has a singularity at $u=0$, which can be removed by defining $\phi_{\mathcal{S}_A(t)}(0)=1$.
\begin{proposition}
    The L\'{e}vy measure for a GSP (see \citet{Kataria2022b}) is given by
\begin{equation*}
    \nu_{\mathcal{S}}(x)=\sum_{i=1}^{k}\lambda_i\delta_i(x)+\sum_{i=1}^{k}\mu_i\delta_i(x)\label{75}
\end{equation*}
where $\delta_{i}(.)$ is the Dirac measure.
\end{proposition}
\begin{proof}
    The proof follows using the independence of two GCPs used in the definition of GSP (see \citet{Kataria2022a} and Proposition 7 of \citet{Gupta2020}).
\end{proof}
The following result provides the distribution of the running average process of a GSP.
\begin{theorem}\label{compound.gsp}
Let $\{Y(t)\}_{t\ge 0}$ be a compound Poisson process given by
\begin{equation*}
   Y(t)=\sum_{i=1}^{N(t)}X_i
\end{equation*}
where $\{N(t)\}_{t\ge 0}$ is a Poisson process with intensity parameter $\Lambda+T$, $\Lambda=\sum_{j=1}^k\lambda_j$ and $T=\sum_{j=1}^k\mu_j$. If $X_i$s are i.i.d. random variables with mixed double uniform distribution function which are independent of $\{N(t)\}_{t\ge 0}$, then
\begin{equation*}
Y(t) \overset{d}{=} \mathcal{S}_A(t).
\end{equation*}
\end{theorem}

\begin{proof}
Since $X_i$s are i.i.d. random variables with mixed double uniform distribution, their pdf is
\begin{align*}
    f_{X_1}(x)=&wf_{U[0,j]}(x)I_{[0,j]}(x)+(1-w)f_{U[-j,0]}(x)I_{[-j,0]}(x).
\end{align*}
Let $w=\frac{\Lambda}{\Lambda+T}$ so $1-w=\frac{T}{\Lambda+T}$. For $-j\leq x\leq j$ and $1\le j\le k$, we have
\begin{align*}
    f_{X_1}(x)&=\frac{\Lambda}{\Lambda+T}f_{U[0,j]}(x)I_{[0,j]}(x)+\frac{T}{\Lambda+T}f_{U[-j,0]}(x)I_{[-j,0]}(x)\,\\
    &=\frac{\Lambda}{\Lambda+T}I_{[0,j]}(x)\frac{1}{j}+\frac{T}{\Lambda+T}I_{[-j,0]}(x)\frac{1}{j} \qquad\qquad \\
    &=\frac{1}{\Lambda+T}\sum_{j=1}^{k}\lambda_j\frac{1}{j}I_{[0,j]}(x)+\frac{1}{\Lambda+T}\sum_{j=1}^{k}\mu_j\frac{1}{j}I_{[-j,0]}(x).
\end{align*}
The characteristic function of $X_1$ is given by
\begin{align*}
    \phi_{X_1}(u)=\int_{0}^{j}e^{iux}f_X(x)dx
    =&\frac{1}{\Lambda+T}\sum_{j=1}^{k}\int_{0}^{j}\frac{e^{iux}\lambda_j}{j}dx+\frac{1}{\Lambda+T}\sum_{j=1}^{k}\int_{-j}^{0}\frac{e^{iux}\mu_j}{j}dx\\
    =&\frac{1}{\Lambda+T}\sum_{j=1}^{k}\lambda_j\left(\frac{e^{iuj}-1}{iuj}\right)+\frac{1}{\Lambda+T}\sum_{j=1}^{k}\mu_j\left(\frac{1-e^{-iuj}}{iuj}\right).
\end{align*}
So the characteristic function of $Y(t)$ is
\begin{align*}
    \phi_{Y(t)}(u)
    =\mathbb{E}\left[\exp\left(iu\sum_{i=1}^{N(t)}X_i\right)\right]
    &= \mathbb{E}\left[\mathbb{E}\left(\exp\left(iu\sum_{i=1}^{N(t)}X_i\right)\big|N(t)\right)\right]\\
    &= \sum_{n=0}^{\infty}\mathbb{E}\left(\exp\left(iu\sum_{i=1}^{n}X_i\right)\right)P\left(N(t)=n\right) \\
    &=\sum_{n=0}^{\infty}\left(\phi_{X_1}(u)\right)^n P\left(N(t)=n\right) \,\,= P^*\left(\phi_{X_1}(u)\right)   
\end{align*}
where $P^*(.)$ is the p.g.f. of $\{N(t)\}_{t\ge 0}$. Since $N(t)\sim \text{Poisson}\left((\Lambda+T) t\right))$, we have $P^*(s)=e^{-(\Lambda+T)t(1-s)}$. Hence
\begin{equation*}
    \phi_{Y(t)}(u)=e^{-(\Lambda+T) t\left(1-\phi_{X_1}(u)\right)}
    =e^{-(\Lambda+T) t\left(1-\frac{1}{\Lambda+T}\sum_{j=1}^{k}\lambda_j\left(\frac{e^{iuj}-1}{iuj}\right)-\frac{1}{\Lambda+T}\sum_{j=1}^{k}\mu_j\left(\frac{1-e^{-iuj}}{iuj}\right)\right)}\\
    = \phi_{\mathcal{S}_A(t)}(u), 
\end{equation*}
thereby proving the result.
\end{proof}

\begin{corollary}
If $k=1$, then the running average process of a GSP and hence its compound Poisson representation reduce to those of a Skellam process. However, if $\lambda_1=\cdots=\lambda_k$ and $\mu_1=\cdots=\mu_k$, then the running average process of a GSP and hence its compound Poisson representation reduce to those of a Skellam process of order $k$ (see \citet{Gupta2020}). 
\end{corollary}

The $r^{th}$ order moment of $X_i$ can be calculated using the definition
\begin{equation*}
 m_r=\mathbb{E}[X_i^r]=(-i)^r\frac{d^r\phi_{X_i}(u)}{du^r}
\end{equation*}
and the Taylor series expansion of the characteristic function $\phi_{X_1}(u)$ around 0 such that
\begin{equation*}
    \frac{(e^{iuj}-1)}{iuj}=1+\sum_{n=1}^{\infty}\frac{(iuj)^n}{(n+1)!}\quad \And \quad \frac{(1-e^{-iuj})}{iuj}=1+\sum_{n=1}^{\infty}\frac{(-iuj)^n}{(n+1)!}.
\end{equation*}
Using the properties of a compound Poisson process, we have
\begin{align*}
    \mathbb{E}[\mathcal{S}_A(t)]&=\mathbb{E}[N(t)]\mathbb{E}[X_1]=\frac{t}{2}\sum_{j=1}^{k}j\left(\lambda_j-\mu_j\right),\quad
    \mathbb{V}[\mathcal{S}_A(t)]=\mathbb{E}[N(t)]\mathbb{E}[X_1^2]=\frac{t}{3}\sum_{j=1}^{k}j^2\left(\lambda_j+\mu_j\right),\\
    \text{Cov}[\mathcal{S}_A(s),\mathcal{S}_A(t)]
    &=\mathbb{E}[\mathcal{S}_A(s)]\mathbb{E}[\mathcal{S}_A(t-s)]+\mathbb{E}[(\mathcal{S}_A(s))^2]-\mathbb{E}[\mathcal{S}_A(s)]\mathbb{E}[\mathcal{S}_A(t)]=\frac{s}{3}\sum_{j=1}^{k}j^2\left(\lambda_j+\mu_j\right),~s<t.
\end{align*}
The mean and variance of GSP are $\mathbb{E}[\mathcal{S}(t)]=t\sum_{j=1}^{k}j\left(\lambda_j-\mu_j\right)$ and $\mathbb{V}[\mathcal{S}(t)]=t\sum_{j=1}^{k}j^2\left(\lambda_j+\mu_j\right)$ so
\begin{equation*}
    \frac{\mathbb{E}[\mathcal{S}_A(t)]}{\mathbb{E}[\mathcal{S}(t)]}=\frac{1}{2} \qquad\text{and} \qquad \frac{\mathbb{V}[\mathcal{S}_A(t)]}{\mathbb{V}[\mathcal{S}(t)]}=\frac{1}{3}.
\end{equation*}    
Note that
\begin{equation*}
    \mathbb{V}[\mathcal{S}_A(t)]-\mathbb{E}[\mathcal{S}_A(t)]=\frac{t}{3}\sum_{j=1}^{k}j^2\left(\lambda_j+\mu_j\right)-\frac{t}{2}\sum_{j=1}^{k}j\left(\lambda_j-\mu_j\right)
    =t\sum_{j=1}^{k}\frac{(2j^2-3j)\lambda_j+(2j^2+3j)\mu_j}{6}.
\end{equation*}
For $k=1$, we have $\mathbb{V}[\mathcal{S}_A(t)]-\mathbb{E}[\mathcal{S}_A(t)]>0 (<0)$ if $5\mu_1>\lambda_1 (5\mu_1<\lambda_1)$. Accordingly the running average process of a Poisson process is over-dispersed (under-dispersed). Next for $k>1$ with $\sum_{j=1}^{k}\left(2j^2+3j\right)\mu_j>\sum_{j=1}^{k}\left(3j-2j^2\right)\lambda_j$, we have $\mathbb{V}[\mathcal{S}_A(t)]-\mathbb{E}[\mathcal{S}_A(t)]>0$ which implies the running average process of a GSP is over-dispersed. Otherwise, it is under-dispersed. The next result shows the dependence structure of a GSP. 
\begin{proposition}\label{lrd.gsp}
    The running average process of a GSP has the LRD property.
\end{proposition}
\begin{proof}
    For $s<t$, we have
    \begin{align*}
\text{Corr}\left(\mathcal{S}_A(s),\mathcal{S}_A(t)\right)=\frac{\text{Cov}[\mathcal{S}_A(s),\mathcal{S}_A(t)]}{\sqrt{\mathbb{V}[\mathcal{S}_A(s)]}\sqrt{\mathbb{V}[\mathcal{S}_A(t)]}}
    =&\frac{\frac{s}{3}\sum_{j=1}^{k}j^2\left(\lambda_j+\mu_j\right)}{\sqrt{\frac{s}{3}\sum_{j=1}^{k}j^2\left(\lambda_j+\mu_j\right)}\sqrt{\frac{t}{3}\sum_{j=1}^{k}j^2\left(\lambda_j+\mu_j\right)}}\\
    =&\sqrt{s}t^{-\frac{1}{2}}\,\,
    =C(s)t^{-\theta}, 
\end{align*}
where $C(s)=\sqrt{s}$ and $\theta=\frac{1}{2}<1$. Hence $\{\mathcal{S}_A(t)\}_{t\geq 0}$ has the LRD property.
\end{proof}

\begin{remark}
    As NGSP and NGFSP are not L\'{e}vy processes due to non-stationary state probabilities and dependent increments, we can not define their L\'{e}vy measures. Moreover Lemma \ref{running average lemma} can't be used to find the distributions of their running average processes.
\end{remark}

\subsection{Running Average of GCP}
Let $\{\mathbb{M}(t)\}_{t\ge 0}$ be a GCP with intensity parameters $\lambda_1,\lambda_2,\ldots,\lambda_k$.
Its running average process $\{\mathbb{M}_A(t)\}_{t \geq 0}$ can be defined by taking the time-scaled integral of its path similar to that of a GSP. It follows that $\{\mathbb{M}_A(t)\}_{t \geq 0}$ has continuous sample paths of bounded total variation.   
\begin{proposition}\label{77}
The characteristic function of the running average process of a GCP is given by
\begin{align*}
    \phi_{\mathbb{M}_A(t)}(u)=e^{t \sum_{j=1}^{k}\lambda_j\left(\frac{e^{iuj}-1}{iuj}-1\right)}.
\end{align*}    
\end{proposition}

\begin{proof}
The proof is similar to that of Proposition \ref{ch.gsp} and hence omitted.
\end{proof}

We set $\phi_{\mathbb{M}_A(t)}(0)=1$ to remove the singularity of $\phi_{\mathbb{M}_A}(t)$ at $u=0$. The next result shows the distribution of the running average process of a GCP.

\begin{theorem}\label{th.gcp} The running average process of GCP has a compound Poisson representation given by
\begin{equation*}
\mathbb{M}_A(t)\overset{d}{=}\sum_{i=1}^{N(t)}X_i,
\end{equation*}
where $X_i$s are independent and identically distributed random variables, independent of a Poisson process $\{N(t)\}_{t\ge 0}$ with parameter $\Lambda t$ and $\Lambda=\sum_{j=1}^{k}\lambda_j$. 
Moreover the random variable $X_i$ has the pdf
\begin{equation*}
f_{X
_i}(x)=\sum_{j=1}^{k}p_{V_j}(x)f_{U_j}(x),
\end{equation*}
where 
 \begin{align*}
    p_{V_j}(x)=P(V_j=x)=\frac{\lambda_j}{\Lambda}; \quad j=1,2,...,k \quad  \text{and} \quad U_j\sim \text{Uniform}[0,j]; \quad 0\leq x\leq j. 
\end{align*}
\end{theorem} 

\begin{proof}
The proof is similar to that of Theorem \ref{compound.gsp} and hence omitted.
\end{proof}



\begin{corollary}
If $k=1$, then the running average process of a GCP and hence its compound Poisson representation  reduce to those of a standard Poisson process (see \citet{Xia2018}), and to those of a Poisson process of order $k$ (see \citet{Gupta2020}) if $\lambda_1=\cdots=\lambda_k$. 
\end{corollary}

\begin{remark}
In Theorem \ref{th.gcp}, $V_j$s have the same distribution as that of the random variables $X_i$s in the compound Poisson representation of GCP (see \citet{Crescenzo2016}).
\end{remark}
Using the properties of a compound Poisson process, it can be shown that
\begin{equation*}
    \mathbb{E}[\mathbb{M}_A(t)]=\frac{t}{2}\sum_{j=1}^{k}j\lambda_j,\quad
    \mathbb{V}[\mathbb{M}_A(t)]=\frac{t}{3}\sum_{j=1}^{k}j^2\lambda_j,\quad
     \text{Cov}[\mathbb{M}_A(s),\mathbb{M}_A(t)]
    =\frac{s}{3}\sum_{j=1}^{k}j^2\lambda_j\,, \qquad s<t.
\end{equation*}
The mean and variance of GCP are $\mathbb{E}[\mathbb{M}(t)]=t\sum_{j=1}^{k}j\lambda_j$ and $\mathbb{V}[\mathbb{M}(t)]=t\sum_{j=1}^{k}j^2\lambda_j$ so
\begin{equation*}
    \frac{\mathbb{E}[\mathbb{M}_A(t)]}{\mathbb{E}[\mathbb{M}(t)]}=\frac{1}{2} \qquad\text{and} \qquad \frac{\mathbb{V}[\mathbb{M}_A(t)]}{\mathbb{V}[\mathbb{M}(t)]}=\frac{1}{3}.
\end{equation*}
Note that 
    \begin{equation*}
    \mathbb{V}[\mathbb{M}_A(t)]-\mathbb{E}[\mathbb{M}_A(t)]=\frac{t}{3}\sum_{j=1}^{k}j^2\lambda_j-\frac{t}{2}\sum_{j=1}^{k}j\lambda_j
    =t\sum_{j=1}^{k}\lambda_j\left(\frac{j^2}{3}-\frac{j}{2}\right).
\end{equation*}
For $k=1$, we have $\mathbb{V}[\mathbb{M}_A(t)]-\mathbb{E}[\mathbb{M}_A(t)]<0$ which implies that the running average process of a Poisson process is always under-dispersed. However for $k>1$, $\mathbb{V}[\mathbb{M}_A(t)]-\mathbb{E}[\mathbb{M}_A(t)]>0 (<0)$ if $\sum_{j=1}^{k}\lambda_j(3j^2-2j)>0 (<0)$. So the running average process of a GCP may be over-dispersed or under-dispersed. The next result about the dependence structure of a GCP follows from Proposition \ref{lrd.gsp}.
\begin{proposition}
The running average process of a GCP has the LRD property.
\end{proposition}

Finally, we discuss the importance of running average processes in the following remark.
\begin{remark}
     Note that the Riemann integral in Lemma \ref{running average lemma} can be restated as $Y(t)=\int_{0}^{t}X(s)ds=\int_{a}^{a+t}X(s-a)ds$, that is, the process $\{X(s)\}_{s\geq 0}$ is shifted back in time by `$a$' units for the interval $(a,a+t)$. This implies that we can examine the running average of the original process for any time interval (of same length) by time shifting. Also the microscopic structure of the process during various time intervals is revealed by studying the corresponding running averages. For example, if we are modeling share price data for a particular day, then we can study its mean behavior during various time periods of trading activity, which differ in practice. 
\end{remark}

\section{Simulation}\label{sec 8}
In this section, we present algorithms and plots of simulated sample paths for GFSP, NGFSP and NHGFSP along with the p.m.f. for NGSP. We need to simulate two GFCPs, two NGFCPs and two NHGFCPs, and then take their differences to simulate GFSP, NGFSP and NHGFSP respectively. 

For simulation of GFCP, we use its compound Poisson representation given by \citet{Crescenzo2016}:
\begin{equation*}
\mathbb{M}^{\alpha}(t)\overset{d}{=}\sum_{i=1}^{N^{\alpha}(t)}X_i 
\end{equation*}
where $X_i$s are i.i.d. random variables such that $P\{X_1=j\}=\frac{\lambda_j}{\Lambda}$
and $N^{\alpha}(t)$ is an independent Fractional Poisson process with rate $\Lambda=\sum_{j=1}^{k}\lambda_j$. For simulation of NGFCP, we first modify the compound Poisson representation of GCP as
\begin{equation*}
\mathbb{M}(t)\overset{d}{=}\sum_{i=1}^{N(t)}X_i 
\end{equation*}
where $X_i$s are i.i.d. random variables such that $P\{X_1=j\}=\frac{\Lambda_j(t)}{\Lambda(t)}$
and $N(t)$ is an independent Poisson process with parameter $\Lambda(t)=\sum_{j=1}^{k}\Lambda_j(t)$. Without loss of generality, we take $\Lambda_j(t)=(a_j/b_j)(e^{b_jt}-1)+\mu_j(t)$ (Gompertz Makeham rate function) where $a_j,b_j,\mu_j>0$. Then we time change the GCP by an independent inverse stable subordinator to obtain NGFCP. For simulation of NHGFCP, see \eqref{def1}. First we provide the algorithm for simulating sample path of a GFSP. \\

\noindent
{\bf Algorithm for simulation of GFSP} 
\begin{table}[H]
\resizebox{17.0cm}{!}{
\begin{tabular}{@{}llll@{}}
\toprule
\multicolumn{4}{l}{\begin{tabular}[c]{@{}l@{}}\textbf{Input:} $T^* > 0$ and fractional index $\alpha$  \\
Choose rate parameters $\lambda_1, \lambda_2,...,\lambda_k$ for $\{\mathbb{M}^{\alpha}_{1}(t)\}_{t\geq 0}$ and $\mu_1, \mu_2,...,\mu_k$ for $\{\mathbb{M}^{\alpha}_{2}(t)\}_{t\geq 0}$ such that $\Lambda=T$ \\where $\Lambda=\sum_{j=1}^k\lambda_j$ and $T=\sum_{j=1}^k\mu_j$. \\

\textbf{Output:} $\mathcal{S}^{\alpha}(t)$, simulated sample path for GFSP.\\
\midrule
\textit{\quad Initialization: $t=0$, $\mathbb{M}^{\alpha}_{1}(t)=0$, $\mathbb{M}^{\alpha}_{2}(t)=0$.}\\

1: \textbf{while} $t<T^*$ do\\



2: Using Monte-Carlo simulation, generate independent discrete random variables $X_{1}$ and $X_{2}$ with \\

\quad $P(X_{1}=j)=\frac{\lambda_j}{\Lambda}$ and $P(X_{2}=j)=\frac{\mu_j}{T}$ for $1\le j\le k$.\\



3: Generate three independent uniform random variables $U_i \sim U(0,1);\quad i=1,2,3$.\\

4: Set $t = t + \left[\frac{|\ln U_1|^{1/\alpha}}{\Lambda^{1/\alpha}}\frac{sin(\alpha \pi U_2)[sin(1-\alpha)\pi U_2]^{1/\alpha}}{[sin (\pi U_2)]^{1/\alpha}|\ln U_3|^{1/\alpha -1}}\right]$.\\

5: Set $\mathbb{M}^{\alpha}_{1}(t)= \mathbb{M}^{\alpha}_{1}(t) + X_{1}$ and $\mathbb{M}^{\alpha}_{2}(t)= \mathbb{M}^{\alpha}_{2}(t) + X_{2} $\,. \\

6: Compute $\mathcal{S}^{\alpha}(t)= \mathbb{M}^{\alpha}_{1}(t)-\mathbb{M}^{\alpha}_{2}(t)$. \\

7: \textbf{end while} \\

8: \textbf{return  $\mathcal{S}^{\alpha}(t)$}. \end{tabular}}       \\

\midrule
\midrule
\multicolumn{4}{l}{\begin{tabular}[c]{@{}l@{}} 
\textbf{Remark 1:} For simulation of GSP, take $\alpha=1$ and change step 4 to $t = t + \left[-\frac{1}{\Lambda} \ln U\right]$ where $U\sim U(0,1)$. \\
\textbf{Remark 2:} For simulation of NGSP, in addition to changes in Remark 1, take $t=0.0001$ (initialization) \\ and $\lambda_j=\Lambda_j(t)$, $\Lambda=\Lambda(t)$, $\mu_j=T_j(t)$, $T=T(t)$.

\end{tabular}}\\
 \bottomrule 
\end{tabular}
}
\end{table}


\begin{figure}[H]
\centering
\includegraphics[scale=0.72]{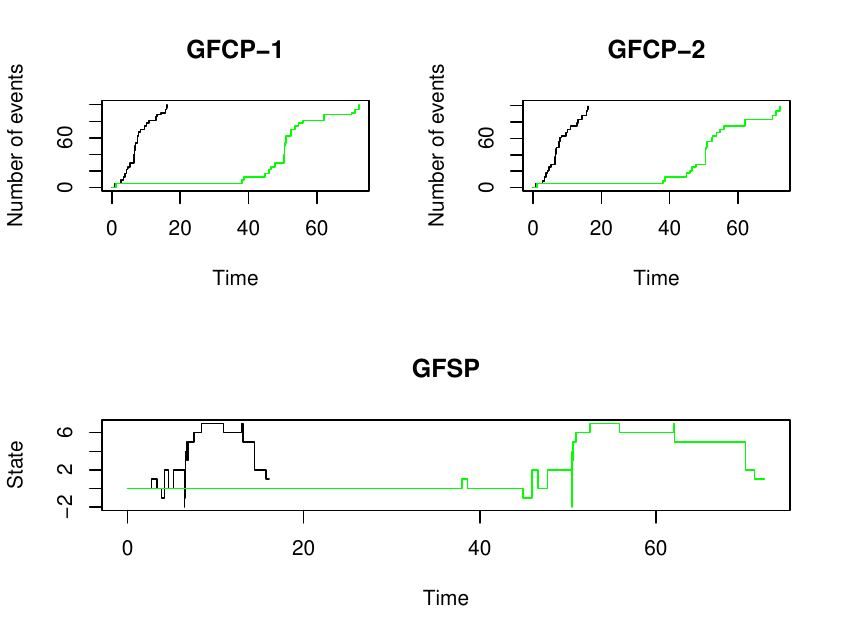}
\caption{Simulated sample paths of GFCP and GFSP for fractional indices $\alpha=0.8$ (black) and $\alpha=0.4$ (green). We considered the intensity parameters $\lambda_1=0.1,~\lambda_2=0.3,~\lambda_3=0.2,~\lambda_4=0.4,~\lambda_5=0.2$ for GFCP-1 and $\mu_1=0.2,~\mu_2=0.2,~\mu_3=0.2,~\mu_4=0.3,~\mu_5=0.3$ for GFCP-2.} \label{fig.gfsp}
\end{figure}

From Figure \ref{fig.gfsp}, we observe that a GFCP with a larger value of $\alpha$ attains any particular state usually in less time than a GFCP with a smaller $\alpha$ value, that is, at any fixed time point, a GFCP with a larger fractional index usually has more occurrences than a GFCP with a smaller fractional index. Now consider the time intervals for which the GFCPs have the same number of occurrences as indicated by flat regions in the graph. We observe that these intervals are usually shorter for larger $\alpha$ values than those for smaller $\alpha$ values. Finally, all the above observations for various fractional indices of GFCPs hold true for GFSPs also. 

Next consider the simulation of sample path of a NGFSP, which involves the following steps.
\begin{itemize}
\item[1.] We will simulate inverse stable subordinator values $Y_\alpha(t_i)$ at time points $t_i$.
\item[2.] Using the simulated values $Y_\alpha(t_i)$ as the new time points, we will simulate the NGFSP.
\end{itemize} 
\noindent
{\bf Algorithm for simulation of inverse stable subordinator}(see \citet{Maheshwari2019}) 

\begin{table}[H]
\resizebox{17.0cm}{!}{
\begin{tabular}{@{}llll@{}}
\toprule
\multicolumn{4}{l}{\textbf{Input:} Choose the fractional index $\alpha$ and $n$ uniformly spaced time points $0=\tau_1,\ldots,\tau_n=T^*$ such that} \\ { $\tau_{j}-\tau_{j-1}=h$ for $2\le j\le n$.} \\

\,\,\textbf{Output:} $Y_\alpha(t)$, the inverse stable subordinator value. \\ 
\midrule
\multicolumn{4}{l}{\begin{tabular}[c]{@{}l@{}}
\textit{Initialization:} $i=0$, $t_0=0$, $Q_0=0$.\\
1: \textbf{while} $t_i<T^*$ do\\

2: Generate increment $Q_{i+1}$ of an independent stable subordinator $D_\alpha(t)$ as follows. \\
    \quad (a) Generate $U\sim U[0,\pi]$ and $V\sim \exp(1)$. \\
    \quad(b) Compute $Q_{i+1}$ for the time interval $(\tau_{j-1}, \tau_{j})$ as \\ 
    \quad $Q_{i+1}=D_\alpha(\tau_j)-D_\alpha(\tau_{j-1})\stackrel{d}{=}D_\alpha(h)=h^{1/\alpha}\frac{\left(\sin \alpha U\right)\left(\sin (1-\alpha)U\right)^{(1-\alpha)/\alpha}}{\left(\sin U\right)^{1/\alpha}V^{(1-\alpha)/\alpha}}$.\\
3:\, Set $Y_i=Y(\ceil{t_i/h} + 1)=\cdots= Y(\lfloor (t_i+Q_{i+1})/h\rfloor+1)=h*i$.\\
4:\, $t_{i+1}=t_i+Q_{i+1}$, $i=i+1$.\\
5: \textbf{end while}\\
6: \textbf{return} $Y_\alpha(t_i)$.\\
\quad The discretized sample path of $D_\alpha(t)$ at $t_i$ is $\sum_{k=0}^i Q_k$ while the discretized sample path of $Y_\alpha(t)$ at $t_i$ is $Y_i$.\\
\end{tabular}}       \\
\bottomrule 
\end{tabular}
}
\end{table}
\noindent
{\bf Algorithm for simulation of NGFSP} 
\begin{table}[H]
\resizebox{17.0cm}{!}{
\begin{tabular}{@{}llll@{}}
\toprule
\multicolumn{4}{l}{\begin{tabular}[c]{@{}l@{}}\textbf{Input:} $Y_\alpha(t_i)$, the simulated inverse stable subordinator values. Choose cumulative rate functions \\ $\Lambda_1(t),\ldots,\Lambda_k(t)$ for $\{\mathcal{M}^{\alpha}_{1}(t)\}_{t\geq 0}$ and $T_1(t),\ldots,T_k(t)$ for $\{\mathcal{M}^{\alpha}_{2}(t)\}_{t\geq 0}$ such that  $\Lambda(t)=T(t)$ for all $t>0$ \\ where $\Lambda(t)=\sum_{j=1}^k\Lambda_j(t)$ and $T(t)=\sum_{j=1}^kT_j(t)$. \\

\textbf{Output:} $\mathcalboondox{S}^{\alpha}(t)$, simulated sample path for NGFSP.\\
\midrule
\textit{\quad Initialization:} $t=0.0001$, $Z_{1i}=0$, $Z_{2i}=0$, $\mathcal{M}_{1}(Y_\alpha(0))=0$, $\mathcal{M}_{2}(Y_\alpha(0))=0$.\\

1: \textbf{while} $t\le Y_\alpha(t_i)$ do\\

2: Using Monte-Carlo simulation, generate independent discrete random variables $X_{1i}$ and 
$X_{2i}$ with \\

\quad $P(X_{1i}=j)=\frac{\Lambda_j(t)}{\Lambda(t)}$ and $P(X_{2i}=j)=\frac{T_j(t)}{T(t)}$ for $1\le j\le k$. \\

3: Set $Z_{1i}=Z_{1i}+X_{1i}$ and $Z_{2i}=Z_{2i}+X_{2i}$. \\



4: Generate a uniform random variable $U \sim U(0,1) $.\\

5. Set $t= t + \left(\frac{-1}{\Lambda(t)}\ln{U}\right) $.\\ 

6: \textbf{end while} \\

7: Set $\mathcal{M}_{1}(Y_\alpha(t_i))=\sum_{j=1}^{i}Z_{1j}$ and $ \mathcal{M}_{2}(Y_\alpha(t_i))=\sum_{j=1}^{i}Z_{2j} $\,. \\

8: Compute $\mathcalboondox{S}^{\alpha}(t_i)= \mathcal{M}_{1}(Y_\alpha(t_i))-\mathcal{M}_{2}(Y_\alpha(t_i))$. \\

9: \textbf{return  $\mathcalboondox{S}^{\alpha}(t_i)$}. \end{tabular}} \\


 \bottomrule 
\end{tabular}
}
\end{table}


\begin{figure}[H]
\centering
\includegraphics[scale=0.8]{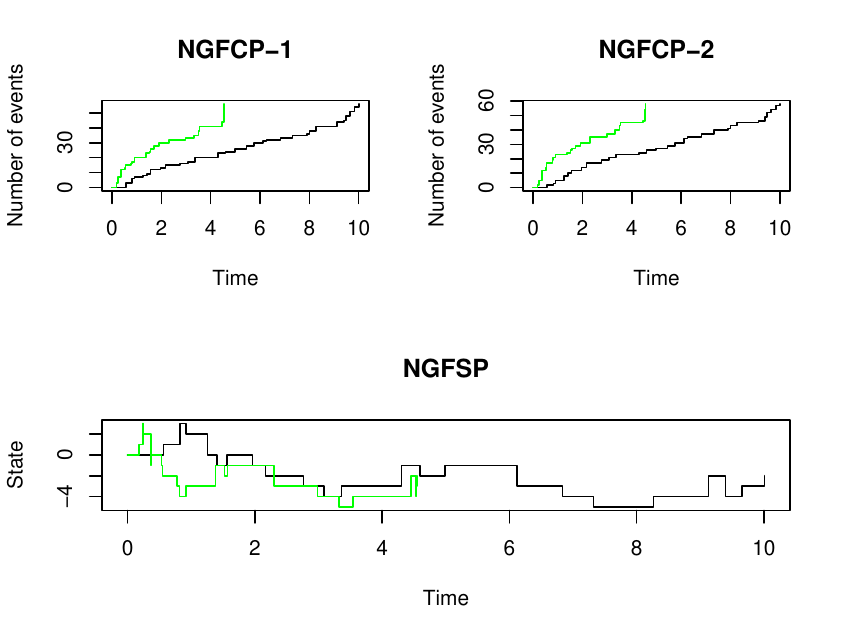}
\caption{Simulated sample paths of NGFCP and NGFSP for fractional indices $\alpha=0.8$ (black) and $\alpha=0.4$ (green). We considered three different  Gompertz-Makeham's rate functions with parameters $a_1=0.6,~b_1=0.1,~\mu_1=5,~a_2=0.7,~b_2=0.2,~\mu_2=4,~a_3=0.4,~b_3=0.3,~\mu_3=7$ for NGFCP-1 and parameters $c_1=0.7,~d_1=0.2,~\nu_1=4,~c_2=0.4,~d_2=0.3,~\nu_2=7,~c_3=0.6,~d_3=0.1,~\nu_3=5$ for NGFCP-2.} \label{fig.ngfsp}
\end{figure}

From Figure \ref{fig.ngfsp}, it's clear that the earlier observations related to GFCP and GFSP for different fractional indices also hold true for NGFCP and NGFSP.\\

\noindent
{\bf Algorithm for simulation of NHGFSP}

\begin{table}[H]
\resizebox{17.0cm}{!}{
\begin{tabular}{@{}llll@{}}
\toprule
\multicolumn{4}{l}{\begin{tabular}[c]{@{}l@{}}\textbf{Input:} $T^* > 0$ and fractional index $\alpha$. Choose cumulative rate functions  $\Lambda_1(t),\ldots,\Lambda_k(t)$ for $\{\widetilde{\mathcal{M}}^{\alpha}_{1}(t)\}_{t\geq 0}$ \\and $T_1(t),\ldots,T_k(t)$ for $\{\widetilde{\mathcal{M}}^{\alpha}_{2}(t)\}_{t\geq 0}$ such that  $\Lambda(t)=T(t)$ for all $t>0$. \\ 

\textbf{Output:} $\widetilde{\mathcalboondox{S}}^{\alpha}(t)$, simulated sample path for NHGFSP.\\
\midrule
\textit{\quad Initialization: $t=0.0001$, $\widetilde{\mathcal{M}}^{\alpha}_{1}(t)=0$, $\widetilde{\mathcal{M}}^{\alpha}_{2}(t)=0$.}\\

1: \textbf{while} $t<T^*$ do\\



2: Using Monte-Carlo simulation, generate independent discrete random variables $X_{1}$ and $X_{2}$ with \\

\quad $P(X_{1}=j)=\frac{\Lambda_j(t)}{\Lambda(t)}$ and $P(X_{2}=j)=\frac{T_j(t)}{T(t)}$ for $1\le j\le k$.\\



3: Generate three independent uniform random variables $U_i \sim U(0,1);\quad i=1,2,3$.\\

4: Set $t = t + \left[\frac{|\ln U_1|^{1/\alpha}}{(\Lambda(t))^{1/\alpha}}\frac{sin(\alpha \pi U_2)[sin(1-\alpha)\pi U_2]^{1/\alpha-1}}{[sin (\pi U_2)]^{1/\alpha}|\ln U_3|^{1/\alpha -1}}\right]$.\\

5: Set $\widetilde{\mathcal{M}}^{\alpha}_{1}(t)= \widetilde{\mathcal{M}}^{\alpha}_{1}(t) + X_{1}$ and $\widetilde{\mathcal{M}}^{\alpha}_{2}(t)= \widetilde{\mathcal{M}}^{\alpha}_{2}(t) + X_{2} $\,. \\

6: Compute $\widetilde{\mathcalboondox{S}}^{\alpha}(t)= \widetilde{\mathcal{M}}^{\alpha}_{1}(t)-\widetilde{\mathcal{M}}^{\alpha}_{2}(t)$. \\

7: \textbf{end while} \\

8: \textbf{return $\widetilde{\mathcalboondox{S}}^{\alpha}(t)$}. \end{tabular}}       \\
     \\
\bottomrule 
\end{tabular}
}
\end{table}

\begin{figure}[H]
\centering
\includegraphics[scale=0.8]{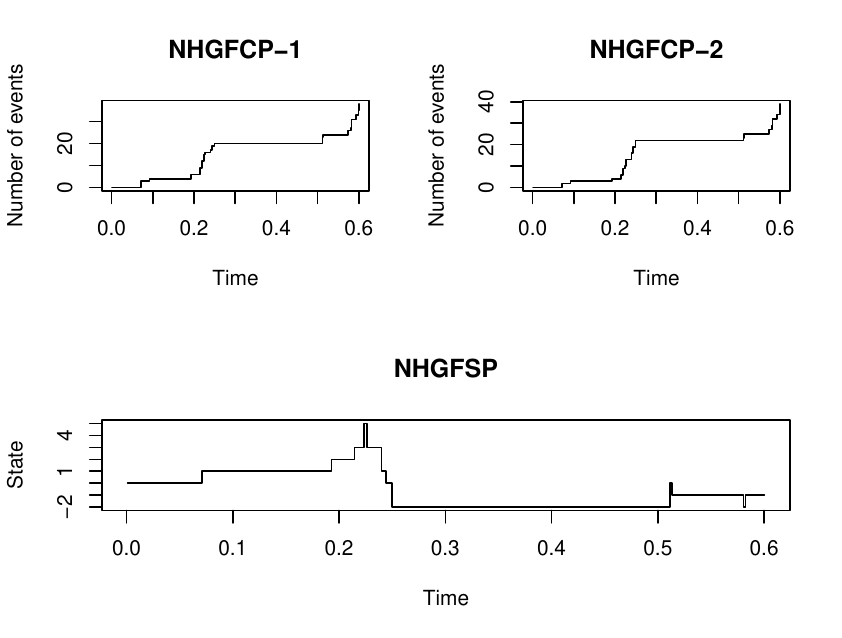}
\caption{Simulated sample paths of NHGFCP and NHGFSP for fractional index $\alpha=0.7$. We considered three different  Gompertz-Makeham's rate functions with parameters $a_1=5,~b_1=0.5,~\mu_1=20,~a_2=2,~b_2=0.2,~\mu_2=22,~a_3=4,~b_3=0.3,~\mu_3=17$ for NHGFCP-1 and parameters $c_1=2,~d_1=0.2,~\nu_1=22,~c_2=4,~d_2=0.3,~\nu_2=17,~c_3=5,~d_3=0.5,~\nu_3=20$ for NHGFCP-2.} \label{fig.nhgfsp}
\end{figure}

\begin{figure}[H] 
\centering
\includegraphics[scale=0.8]{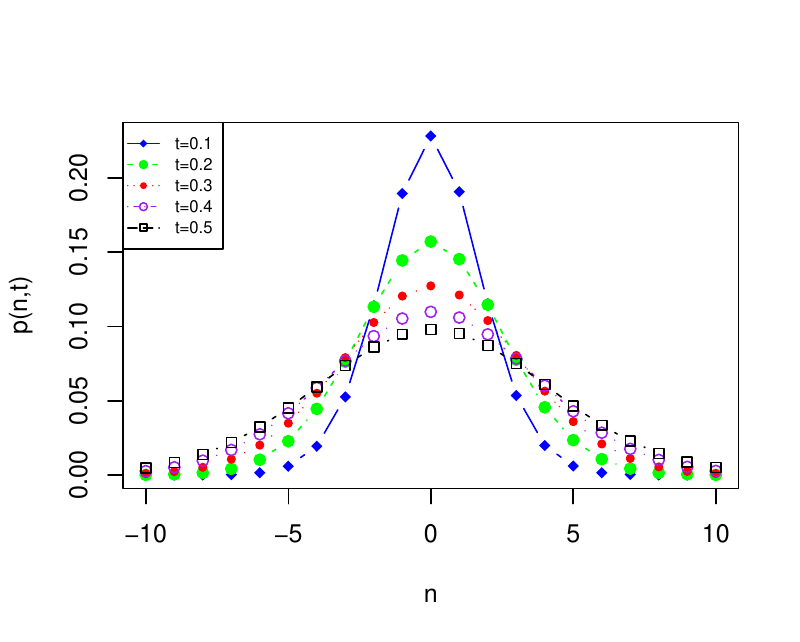}
\caption{p.m.f. of NGSP at times $t = 0.1$, $0.2$, $0.3$, $0.4$, $0.5$ using the Gompertz-Makeham's rate function}\label{fig.ngsp}
\end{figure} 

From Figure \ref{fig.ngsp}, we can observe that the p.m.f. of NGSP is symmetric around zero irrespective of the time instances. This is due to symmetry of the modified Bessel function of the first kind (see Theorem \ref{20}). 

\begin{remark}
From \eqref{53}, note that the p.m.f. of NGFSP involves the density of inverse stable subordinator $h_\alpha(x,t)$ which can be expressed in terms of the Wright generalized Bessel function (see \citet{Leonenko2017}) as follows:

\begin{align*}
h_\alpha(x,t)=\frac{t}{\alpha x^{1+\frac{1}{\alpha}}}g_\alpha \left(\frac{t}{x^{\frac{1}{\alpha}}}\right), \,\, x,t\geq 0 \quad \text{for} \quad g_\alpha(z)=\frac{1}{z}W_{-\alpha,0}\left(-\frac{1}{z^\alpha}\right),\,\,z>0 \qquad 
\end{align*}
where
\begin{equation}\label{82}
W_{\gamma,\beta}= \sum_{k=0}^{\infty} \frac{z^k}{\Gamma(1+k)\Gamma(\beta+\gamma k)}, \quad z\in \mathbb{C}, \gamma > -1,~\beta\in \mathbb{R}.
\end{equation} 
However, this function is not pre-defined in R software and defined for only non-negative parameter values in Python unlike \eqref{82}. Hence the the p.m.f. of NGFSP can't be plotted using these softwares.  
\end{remark} 

\section{Application}\label{sec 9}
In this section, we discuss an application of NGFSP in high-frequency financial data analysis to model complex changing patterns in high-frequency  trade executions during rapid market movements.

Previous studies on this subject include various models based on the difference of two point processes. \citet{Barndorff2012} and \citet{Carr2011} considered the difference of Poisson processes (also known as Skellam process) to model financial transactions whereas \citet{Bacry2013a, BACRY20132475}) evaluated the use of Hawkes processes. The main drawback of these models is their exponential inter-arrival times (i.e. time between trades), which contradict empirical evidence. To overcome this drawback, \citet{Kerss2014} used the fractional Skellam process of type I which is the difference between two time fractional Poisson processes. The main advantage of this process is the Mittag-Leffler inter-arrival times, which are empirically observed to have long memory. Additionally, for an application of integer-valued L\'{e}vy processes in financial econometrics, we refer the reader to \citet{Barndorff2012}, where they have used the difference between two stable subordinators.

\citet{Kerss2014} reviewed three-month transaction records of a forward asset (called future) containing 5,465,779 timestamped transactions recorded over market opening hours. They defined the bid-price function on the basis of the price of the forward asset at time $t$ denoted by $F(t)$. The bid-ask price spread was fixed to a single tick of 0.0001. \citet{Kerss2014} defined $F(t)^{bid}$ as follows:
\begin{equation}\label{81}
   F(t)^{bid}= \begin{cases}
F(t-1)^{bid} \qquad \text{if}\,\, F(t-1)^{bid}\leq F(t)\leq F(t-1)^{ask}\\
F(t) \quad\qquad\qquad \text{if}\,\, F(t)< F(t-1)^{bid}\\
F(t)-0.0001 \quad \text{if}\,\, F(t)> F(t-1)^{ask}.
\end{cases}
\end{equation}
The data was filtered by deleting all the entries where the bid price had not changed from the previous bid price (case 1 of \eqref{81}), i.e., no up or down jump had occurred, leaving 682,550 records. The price data was modeled using the fractional Skellam process of type I, as shown in the first graph in Figure 2 of \citet{Kerss2014}, which shows both the original and the filtered data. Up jump data (the upward changes in prices as defined by case 3 of \eqref{81}) and the down jump data (the downward changes in prices as defined by case 2 of \eqref{81}) can be modeled using the time fractional Poisson process as shown in the second graph in Figure 2 of \citet{Kerss2014}. The up and down jump time series was constructed by removing all trades with negative jumps as well as removing all the duplicate time stamps, leaving 253,092 and 281,833 observations for up and down jumps, respectively. Finally, Figure 3 of \citet{Kerss2014} showed that the Mittag-Leffler distribution provides a better fit than the exponential distribution for inter-arrival times (time between trades).

Next, we explain how our proposed model NGFSP better fits both the first and second graphs in Figure 2 of \citet{Kerss2014}. It is observed from the original data that almost ninety-eight per cent of jumps are of single tick magnitude. However, to make the model more sensible, it needs to be extended to allow jumps greater than one tick, thereby accounting for the remaining data. In practice, prices can fluctuate in both directions at greater magnitudes based on a number of different reasons. Such price changes should be incorporated into a model by allowing larger jump sizes with varying intensities. It has also been commonly observed that fluctuations in asset prices are not always consistent. The market conditions or volatility may vary based on the time of day or on various days of the month, and the fluctuations may follow specific patterns. For example, volatility is usually higher during the opening or closing hours on a given trading day or on days affected by some catastrophic news. To address this non-homogeneity in price movement, a model should accommodate various possible non-homogeneous rates rather than just homogeneous rates.

Our proposed model allows jumps of different magnitudes with different intensities and also incorporates different non-homogeneous rate functions. This makes NGFSP well-suited to analyze high-frequency financial data. Moreover note that NGFCP has Mittag-Leffler inter-arrival times which provide a better fit than exponentially distributed inter-arrival times for such data. In summary, NGFCP may be used to model the up and down jump processes while NGFSP may be used to model the original stock price data. 


\section*{Declarations}
\noindent 
{\bf Conflict of Interest}. The authors declare that they have no conflict of interest.

\bibliographystyle{plainnat}
\bibliography{NGFSP}

\begin{thebibliography}{38}
\providecommand{\natexlab}[1]{#1}
\providecommand{\url}[1]{\texttt{#1}}
\expandafter\ifx\csname urlstyle\endcsname\relax
  \providecommand{\doi}[1]{doi: #1}\else
  \providecommand{\doi}{doi: \begingroup \urlstyle{rm}\Url}\fi

\bibitem[Abramowitz and Stegun(1972)]{Abramowitz1972}
Milton Abramowitz and Irene~Ann Stegun.
\newblock \emph{{Handbook of mathematical functions with formulas, graphs, and mathematical tables}}.
\newblock National Bureau of Standards Applied Mathematics Series, 1972.

\bibitem[Bacry et~al.(2013{\natexlab{a}})Bacry, Delattre, Hoffmann, and Muzy]{BACRY20132475}
Emmanuel Bacry, Sylvain Delattre, Marc Hoffmann, and Jean~Francois Muzy.
\newblock Some limit theorems for hawkes processes and application to financial statistics.
\newblock \emph{Stochastic Processes and their Applications}, 123\penalty0 (7):\penalty0 2475--2499, 2013{\natexlab{a}}.
\newblock A Special Issue on the Occasion of the 2013 International Year of Statistics.

\bibitem[Bacry et~al.(2013{\natexlab{b}})Bacry, Delattre, Hoffmann, and Muzy]{Bacry2013a}
Emmanuel Bacry, Sylvain Delattre, Marc Hoffmann, and Jean~Francois Muzy.
\newblock Modelling microstructure noise with mutually exciting point processes.
\newblock \emph{Quantitative Finance}, 13\penalty0 (1):\penalty0 65--77, 2013{\natexlab{b}}.

\bibitem[Barndorff-Nielsen et~al.(2012)Barndorff-Nielsen, Pollard, and Shephard]{Barndorff2012}
Ole~Eiler Barndorff-Nielsen, David~G. Pollard, and Neil Shephard.
\newblock Integer-valued lévy processes and low latency financial econometrics.
\newblock \emph{Quantitative Finance}, 12\penalty0 (4):\penalty0 587--605, 2012.

\bibitem[Barrera and Fernandez(2009)]{Barrera2009}
Bertoncini Barrera and Roberto Fernandez.
\newblock Abrupt convergence and escape behavior for birth and death chains.
\newblock \emph{Journal of Statistical Physics}, 137:\penalty0 595--623, 2009.

\bibitem[Beghin and Orsingher(2010)]{Beghin2010}
Luisa Beghin and Enzo Orsingher.
\newblock Poisson-type processes governed by fractional and higher-order recursive differential equations.
\newblock \emph{Electronic Journal of Probability}, 15:\penalty0 684--709, 2010.

\bibitem[Carr(2011)]{Carr2011}
Peter~Paul Carr.
\newblock Semi-static hedging of barrier options under poisson jumps.
\newblock \emph{International Journal of Theoretical and Applied Finance}, 14:\penalty0 1091--1111, 2011.

\bibitem[Di~Crescenzo et~al.(2016)Di~Crescenzo, Martinucci, and Meoli]{Crescenzo2016}
Antonio Di~Crescenzo, Barbara Martinucci, and Alessandra Meoli.
\newblock A fractional counting process and its connection with the poisson process.
\newblock \emph{ALEA, Latin American Journal Probability and Mathematical Statistics}, 13:\penalty0 291--307, 2016.

\bibitem[D’Ovidio and Nane(2014)]{DOVIDIO2014}
Mirko D’Ovidio and Erkan Nane.
\newblock Time dependent random fields on spherical non-homogeneous surfaces.
\newblock \emph{Stochastic Processes and their Applications}, 124\penalty0 (6):\penalty0 2098--2131, 2014.

\bibitem[Gupta et~al.(2020)Gupta, Kumar, and Leonenko]{Gupta2020}
Neha Gupta, Arun Kumar, and Nikolai Leonenko.
\newblock Skellam type processes of order k and beyond.
\newblock \emph{Entropy}, 22:\penalty0 1193(1--21), 2020.

\bibitem[Gupta et~al.(2024)Gupta, Kumar, Leonenko, and Vaz]{Gupta2024}
Neha Gupta, Arun Kumar, Nikolai Leonenko, and Jayme Vaz.
\newblock Generalized fractional derivatives generated by dickman subordinator and related stochastic processes.
\newblock \emph{Fractional Calculus and Applied Analysis}, 27\penalty0 (4):\penalty0 1527--1563, 2024.
\newblock \doi{https://doi.org/10.1007/s13540-024-00289-x}.

\bibitem[Hwang et~al.(2007)Hwang, Kim, and Kweon]{Hwang2007}
Youngbae Hwang, Jun-Sik Kim, and In-So Kweon.
\newblock Sensor noise modeling using the skellam distribution: Application to the color edge detection.
\newblock In \emph{{Proceedings of Computer Society Conference on Computer Vision and Pattern Recognition}}, pages 1--8, 2007.

\bibitem[Irwin(2018)]{Irwin2018}
Joseph~Oscar Irwin.
\newblock {The Frequency Distribution of the Difference between Two Independent Variates Following the Same Poisson Distribution}.
\newblock \emph{Journal of the Royal Statistical Society}, 100\penalty0 (3):\penalty0 415--416, 2018.

\bibitem[Karlis and Ntzoufras(2008)]{Karlis2008}
Dimitris Karlis and Ioannis Ntzoufras.
\newblock Bayesian modelling of football outcomes: Using the skellam's distribution for the goal difference.
\newblock \emph{IMA Journal of Management Mathematics}, 20\penalty0 (2):\penalty0 133--145, 2008.

\bibitem[Kataria and Khandakar(2022{\natexlab{a}})]{Kataria2022a}
Kuldeep~Kumar Kataria and Mostafizar Khandakar.
\newblock Generalized fractional counting process.
\newblock \emph{Journal of Theoretical Probability}, 35\penalty0 (4):\penalty0 2784--2805, 2022{\natexlab{a}}.

\bibitem[Kataria and Khandakar(2022{\natexlab{b}})]{Kataria2022b}
Kuldeep~Kumar Kataria and Mostafizar Khandakar.
\newblock Skellam and time-changed variants of the generalized fractional counting process.
\newblock \emph{Fractional Calculus and Applied Analysis}, 25\penalty0 (5):\penalty0 1873--1907, 2022{\natexlab{b}}.
\newblock \doi{https://doi.org/10.1007/s13540-022-00091-7}.

\bibitem[Kataria and Khandakar(2024)]{Kataria2024}
Kuldeep~Kumar Kataria and Mostafizar Khandakar.
\newblock Fractional skellam process of order k.
\newblock \emph{Journal of Theoretical Probability}, 37:\penalty0 1333--1356, 2024.

\bibitem[Kataria et~al.(2025)Kataria, Khandakar, and Vellaisamy]{Kataria2025}
Kuldeep~Kumar Kataria, Mostafizar Khandakar, and Palaniappan Vellaisamy.
\newblock Non-homogeneous and time-changed versions of generalized counting processes.
\newblock \emph{Advances in Applied probability}, pages 1--37, 2025.
\newblock \doi{10.1017/apr.2025.21}.

\bibitem[Kerss et~al.(2014)Kerss, Leonenko, and Sikorskii]{Kerss2014}
Alexander Kerss, Nikolai Leonenko, and Alla Sikorskii.
\newblock Fractional skellam processes with applications to finance.
\newblock \emph{Fractional Calculus and Applied Analysis}, 17\penalty0 (2):\penalty0 532--551, 2014.
\newblock \doi{https://doi.org/10.2478/s13540-014-0184-2}.

\bibitem[Kilbas et~al.(2006)Kilbas, Srivastava, and Trujillo]{Kilbas2006}
Anatoly~Alexandrovich Kilbas, Hari~Mohan Srivastava, and Juan~Jose Trujillo.
\newblock \emph{{Theory and Applications of Fractional Differential Equations}}.
\newblock Elsevier Science B.V., Amsterdam, 2006.

\bibitem[Kostadinova and Minkova(2013)]{Kostadinova2013}
Krasimira Kostadinova and Leda Minkova.
\newblock On the poisson process of order k.
\newblock \emph{Pliska Studia Mathematica Bulgarica}, 22:\penalty0 117--128, 2013.

\bibitem[Laskin(2003)]{LASKIN2003201}
Nick Laskin.
\newblock Fractional poisson process.
\newblock \emph{Communications in Nonlinear Science and Numerical Simulation}, 8\penalty0 (3):\penalty0 201--213, 2003.

\bibitem[Laskin(2009)]{Laskin2010}
Nick Laskin.
\newblock Some applications of the fractional poisson probability distribution.
\newblock \emph{Journal of Mathematical Physics}, 50\penalty0 (11):\penalty0 113513, 2009.

\bibitem[Leonenko et~al.(2014)Leonenko, Meerschaert, Schilling, and Sikorskii]{Leonenko2014}
Nikolai Leonenko, Mark Meerschaert, René Schilling, and Alla Sikorskii.
\newblock Correlation structure of time-changed lévy processes.
\newblock \emph{Communications in Applied and Industrial Mathematics}, 6, 2014.

\bibitem[Leonenko et~al.(2017)Leonenko, Scalas, and Trinh]{Leonenko2017}
Nikolai Leonenko, Enrico Scalas, and Mailan Trinh.
\newblock The fractional non-homogeneous poisson process.
\newblock \emph{Statistics and Probability Letters}, 120:\penalty0 147--156, 2017.

\bibitem[Leonenko et~al.(2019)Leonenko, Scalas, and Trinh]{Leonenko2019}
Nikolai Leonenko, Enrico Scalas, and Mailan Trinh.
\newblock Limit theorems for the fractional non-homogeneous poisson process.
\newblock \emph{Journal of Applied Probability}, 56\penalty0 (1):\penalty0 246--264, 2019.

\bibitem[Maheshwari and Vellaisamy(2019{\natexlab{a}})]{Maheshwari2019}
Aditya Maheshwari and Palaniappan Vellaisamy.
\newblock Fractional poisson process time-changed by lévy subordinator and its inverse.
\newblock \emph{Journal of Theoretical Probability}, 32:\penalty0 1278--1305, 2019{\natexlab{a}}.

\bibitem[Maheshwari and Vellaisamy(2019{\natexlab{b}})]{NHSTFPP}
Aditya Maheshwari and Palaniappan Vellaisamy.
\newblock Non-homogeneous space-time fractional poisson processes.
\newblock \emph{Stochastic Analysis and Applications}, 37\penalty0 (2):\penalty0 137--154, 2019{\natexlab{b}}.

\bibitem[Mathai and Haubold(2008)]{Mathai2008}
Arak~Mathai Mathai and Hans~Joachim Haubold.
\newblock \emph{Special Functions for Applied Scientist}.
\newblock Springer, 2008.

\bibitem[Meerschaert and Straka(2013)]{Meerschaert2013}
Mark Meerschaert and Peter Straka.
\newblock Inverse stable subordinators.
\newblock \emph{Mathematical modelling of natural phenomena}, 8\penalty0 (2):\penalty0 1--16, 2013.

\bibitem[Meerschaert et~al.(2011)Meerschaert, Nane, and Vellaisamy]{Meerschaert2011}
Mark Meerschaert, Erkan Nane, and Palaniappan Vellaisamy.
\newblock The fractional poisson process and the inverse stable subordinator.
\newblock \emph{Electronic Journal of Probability}, 16:\penalty0 1600--1620, 2011.

\bibitem[Meerschaert et~al.(2014)Meerschaert, Schilling, and Sikorskii]{Meerschaert2015}
Mark Meerschaert, René Schilling, and Alla Sikorskii.
\newblock Stochastic solutions for fractional wave equations.
\newblock \emph{Nonlinear Dynamics}, 80:\penalty0 1685--1695, 2014.

\bibitem[Saxena et~al.(2004)Saxena, Mathai, and Haubold]{Saxena2004}
Ram~Kishore Saxena, Arak~Mathai Mathai, and Hans~Joachim Haubold.
\newblock Unified fractional kinetic equation and a fractional diffusion equation.
\newblock \emph{Astrophysics and Space Science}, 290\penalty0 (3):\penalty0 299--310, 2004.

\bibitem[Sengar et~al.(2019)Sengar, Maheshwari, and Upadhye]{Sengar2020}
Ayushi Sengar, Aditya Maheshwari, and Neelesh Upadhye.
\newblock Time-changed poisson processes of order k.
\newblock \emph{Stochastic Analysis and Applications}, 38\penalty0 (1):\penalty0 1--25, 2019.

\bibitem[Skellam(1946)]{Skellam1946}
John~Gordon Skellam.
\newblock The frequency distribution of the difference between two poisson variates belonging to different populations.
\newblock \emph{Journal of Royal Statistical Society, Series A}, 109\penalty0 (3):\penalty0 296--296, 1946.

\bibitem[Sneddon(1966)]{Sneddon1966}
Ian~Naismith Sneddon.
\newblock \emph{Special functions of Mathematical Physics and Chemistry}.
\newblock Oliver and Boyd Limited, Edinburg and London, 1966.

\bibitem[Vellaisamy and Maheshwari(2018)]{Maheshwari2016}
Palaniappan Vellaisamy and Aditya Maheshwari.
\newblock Fractional negative binomial and {P}olya processes.
\newblock \emph{Probability and Mathematical Statistics}, 38\penalty0 (1):\penalty0 77--101, 2018.

\bibitem[Xia(2018)]{Xia2018}
Weixuan Xia.
\newblock On the distribution of the running average of a {S}kellam process.
\newblock \emph{International Journal of Pure and Applied Mathematics}, 119\penalty0 (3):\penalty0 461--473, 2018.

\end{thebibliography}

\end{document}